\documentclass[11pt, leqno]{amsart}

\usepackage{graphicx} 
\usepackage{amsmath,amsthm,amssymb} 
\usepackage{amsfonts} 
\usepackage{mathrsfs} 
\usepackage[T1]{fontenc} 
\usepackage[utf8]{inputenc} 
\usepackage{mathtools} 
\usepackage{comment} 

\usepackage[dvipsnames, svgnames]{xcolor} 
\usepackage[shortlabels]{enumitem} 

\usepackage{breakcites}
\usepackage[colorlinks,citecolor=citec,hypertexnames=false,urlcolor=urlc,breaklinks=true, pagebackref, breaklinks=true]{hyperref} 
\usepackage{thmtools} 
\usepackage{accents} 

\usepackage{amsbsy} 
\usepackage{amscd} 
\usepackage{latexsym} 
\usepackage{txfonts} 
\usepackage{exscale} 
\usepackage{bbm} 

\usepackage{geometry} 

\newif\ifsoldark
\newif\ifsollight
\newif\ifclassic
\newif\ifplain

\usepackage{cite} 

\renewcommand*\backref[1]{\ifx#1\relax \else (p. #1) \fi} 

\numberwithin{equation}{section}

\theoremstyle{plain}
\newtheorem{theorem}[equation]{Theorem}

\newtheorem{lemma}[equation]{Lemma}
\newtheorem{corollary}[equation]{Corollary}
\newtheorem{proposition}[equation]{Proposition}

\theoremstyle{definition}
\newtheorem{definition}[equation]{Definition}

\theoremstyle{remark}
\newtheorem{remark}[equation]{Remark}

\newtheorem{claim}[equation]{Claim}

\newcommand{\RR}{\mathbb{R}}
\newcommand{\NN}{\mathbb{N}}

\newcommand{\CC}{\mathbb{C}}
\newcommand{\DD}{\mathbb{D}}
\newcommand{\DDt}{{\mathbb{D}_t}}
\newcommand{\rn}{{\mathbb{R}^n}}
\newcommand{\reu}{{\mathbb{R}^{n+1}_+}}
\newcommand{\ree}{\mathbb{R}^{n+1}}

\renewcommand{\div}{\textup{div}}
\newcommand{\pd}{\partial}

\newcommand{\m}{\mathcal{M}}

\newcommand{\lprn}{{L^p(\rn)}}
\newcommand{\lprni}{{L^{p,\infty}(\rn)}}
\newcommand{\ltrn}{{L^2(\rn)}}

\newcommand{\ltnu}{{L^2(\nu)}}
\newcommand{\lnrn}{{L^n(\rn)}}

\newcommand{\vp}{\varphi}
\newcommand{\vpe}{{\vp_\varepsilon}}

\DeclareMathOperator{\supp}{supp}

\newcommand{\dint}{\int\!\!\!\!\!\int}
\renewcommand{\L}{\mathcal{L}}
\renewcommand{\sl}{\mathcal{S}^{\mathcal{L}}}

\newcommand{\thto}{\Theta_{t,1}}
\newcommand{\thtm}{\Theta_{t,m}}
\newcommand{\thtmt}{\theta_{t,m-1}}
\newcommand{\thsmt}{\theta_{s,m-1}}
\newcommand{\thtmo}{\Theta_{t,m+1}}
\newcommand{\lthtm}{\theta_{t,m}}
\newcommand{\thtmb}{\Theta_{t,m}^B}
\newcommand{\thtmp}{\Theta_{t,m}^{\|}}
\newcommand{\s}{\mathbb{S}}
\renewcommand{\v}{\mathbb{V}}

\newcommand{\mntm}{\widetilde{\mathcal{N}}}
\newcommand{\nuxt}{\nu_{x,t}}
\newcommand{\nuyt}{\nu_{y,t}}

\newcommand{\tithm}{\widetilde{\Theta}_{t,m}}
\newcommand{\tit}{\widetilde{T}_t}
\newcommand{\rtt}{\widetilde{R}_t}
\newcommand{\reto}{R_t^{[1]}}
\newcommand{\rett}{R_t^{[2]}}
\newcommand{\vwt}{\dfrac{\nu(x)}{\nuxt}}
\newcommand{\cc}{\mathcal{C}}
\newcommand{\mntmo}{\widetilde{\mathcal{N}}_1}
\newcommand{\mntmt}{\widetilde{\mathcal{N}}_2}

\newcommand{\mntmq}{\widetilde{\mathcal{N}}_q}

\newcommand{\lmntmq}{\widetilde{\mathcal{N}}_q^\varepsilon}

\newcommand{\Q}{\mathcal{Q}}
\newcommand{\Qt}{\widetilde{\mathcal{Q}}}
\newcommand{\Pt}{\widetilde{P}}

\newcommand{\cxt}{{\mathcal{C}_{x,t}}}
\newcommand{\cxtd}{{\mathcal{C}^*_{x,t}}}

\def\Xint#1{\mathchoice
	{\XXint\displaystyle\textstyle{#1}}%
	{\XXint\textstyle\scriptstyle{#1}}%
	{\XXint\scriptstyle\scriptscriptstyle{#1}}%
	{\XXint\scriptscriptstyle\scriptscriptstyle{#1}}%
	\!\int}
\def\XXint#1#2#3{{\setbox0=\hbox{$#1{#2#3}{\int}$}
		\vcenter{\hbox{$#2#3$}}\kern-.5\wd0}}

\def\dashint{\Xint-}

\def\Yint#1{\mathchoice
	{\YYint\displaystyle\textstyle{#1}}%
	{\YYint\textstyle\scriptstyle{#1}}%
	{\YYint\scriptstyle\scriptscriptstyle{#1}}%
	{\YYint\scriptscriptstyle\scriptscriptstyle{#1}}%
	\!\dint}
\def\YYint#1#2#3{{\setbox0=\hbox{$#1{#2#3}{\iint}$}
		\vcenter{\hbox{$#2#3$}}\kern-.51\wd0}}
\def\longdash{{-}\mkern-3.5mu{-}} 

\def\fiint{\Yint\longdash}

\newcommand{\aqm}{a_{q,m}}

\newcommand{\ele}{E_{\lambda,\varepsilon}}

\newcommand{\dfint}{\fiint}

\newcommand{\divp}{\div_{\|}}
\newcommand{\nbp}{\nabla_{\|}}
\newcommand{\nb}{\nabla}
\newcommand{\ap}{A_{\|}}
\newcommand{\Lp}{\L_{\|}}

\newcommand{\dl}{\mathcal{D}^{\L}}

\newcommand{\vb}{\vec{b}}
\newcommand{\va}{\vec{a}}

\newcommand{\rjt}{\widetilde{R}_j}

\newcommand{\yot}{Y^{1,2}(\ree)}
\newcommand{\dyot}{Y^{1,2}(\ree)^*}
\newcommand{\yotp}{Y^{1,2}(\reu)}
\newcommand{\yotn}{{Y^{1,2}(\rn)}}
\newcommand{\dyotn}{{Y^{1,2}(\rn)^*}}
\newcommand{\yoh}{\dot{H}^{1/2}(\rn)}
\newcommand{\dyoh}{\dot{H}^{-1/2}(\rn)}

\newcommand{\sltp}{S^2_+}

\newcommand{\dltp}{D^2_+}

\newcommand{\cD}{\mathcal{D}}

\newcommand{\nablap}{\nabla_{\|}}
\newcommand{\cM}{\mathcal{M}}
\newcommand{\cA}{\mathcal{A}}
\newcommand{\dno}{D_{n+1}}

\newcommand{\tr}{\operatorname{Tr}}
\newcommand{\Tr}{\text{\normalfont Tr}}
\newcommand{\tro}{{\tr_0}}
\newcommand{\trt}{{\tr_t}}
\newcommand{\trtau}{{\tr_\tau}}

\newcommand{\F}{\mathcal{F}}

\newcommand{\cL}{\mathcal{L}}
\newcommand{\loc}{\operatorname{loc}}
\newcommand{\cn}{{\mathbb{C}^n}}
\renewcommand{\fint}{\dashint}
\newcommand{\ra}{\rightarrow}
\newcommand{\bb}{\mathbb}
\newcommand{\ep}{\varepsilon}

\newcommand{\slz}{\mathcal{S}^{\mathcal{L}_z}}
\newcommand{\tmu}{\widetilde{\mu}}

\newcommand{\dlp}{\mathcal{D}^{\L ,+}}
\newcommand{\om}{\Omega}
\newcommand{\kap}{K_{\alpha,p}}
\newcommand{\aqa}{A_{Q,\alpha}}

\def\Yint#1{\mathchoice
	{\YYint\displaystyle\textstyle{#1}}%
	{\YYint\textstyle\scriptstyle{#1}}%
	{\YYint\scriptstyle\scriptscriptstyle{#1}}%
	{\YYint\scriptscriptstyle\scriptscriptstyle{#1}}%
	\!\dint}
\def\YYint#1#2#3{{\setbox0=\hbox{$#1{#2#3}{\dint}$}
		\vcenter{\hbox{$#2#3$}}\kern-.51\wd0}}

\def\longdash{{-}\mkern-3.5mu{-}}

\DeclareUnicodeCharacter{218A}{\SneiS}

\newcommand{\Di}{\operatorname{(D)}}
\newcommand{\Ne}{\operatorname{(N)}}
\newcommand{\Reg}{\operatorname{(R)}}

\newcommand{\jllg}[1]{\marginpar{\scriptsize\color{jllgcolor}\textbf{JLLG:} #1}}

\allowdisplaybreaks

\setlist{nosep} 

\plaintrue 

\ifplain

\colorlet{citec}{blue}
\colorlet{urlc}{blue}
\colorlet{toc}{red}
\colorlet{hyperc}{blue}
\colorlet{bpcolor}{OliveGreen}
\colorlet{smcolor}{purple}
\colorlet{jllgcolor}{Red} 
\colorlet{shcolor}{Orange} 
\colorlet{sbcolor}{NavyBlue} 
\colorlet{impcolor}{ProcessBlue}
\colorlet{eqcolor}{blue}
\colorlet{lmcolor}{black}
\colorlet{propcolor}{black}
\colorlet{thmcolor}{black}
\colorlet{defcolor}{black}
\colorlet{rmcolor}{black}
\colorlet{excolor}{black}

\fi

\newcommand{\bbm}[1]{\mathbbm{#1}}
\newcommand{\n}[1]{\mathscr{#1}}

\newcommand{\Hf}{H^{\frac12}_0(\mathbb R^n)}
\newcommand{\Hfm}{H^{-\frac12}(\mathbb R^n)}
\newcommand*{\dt}[1]{%
	\accentset{\mbox{\Large\bfseries .}}{#1}}

\newcommand{\cee}{\mathbb{C}^{n+1}}
\newcommand{\real}{\operatorname{\Re e\,}}

\begin{document}
	
	\title[Critical Perturbation Theory, Part II]{Critical Perturbations for Second Order Elliptic Operators. Part II: Non-tangential maximal function estimates}
	\author{S. Bortz}
	
	\address{Simon Bortz \\
		Department of Mathematics, University of Alabama, Tuscaloosa, AL, 35487, USA}
	\email{sbortz@ua.edu} 
	
	\author{S. Hofmann}
	
	\address{Steve Hofmann \\
		Department of Mathematics, University of Missouri, Columbia, MO 65211, USA}
	\email{hofmanns@missouri.edu}
	
	\author{J. L. Luna Garcia}
	
	\address{Jos\'e Luis Luna Garcia \\
		Department of Mathematics and Statistics, McMaster University, Ontario, L8S 4K1, Canada}
	\email{jlwwc@mail.missouri.edu}
	
	\author{S. Mayboroda}
	
	\address{Svitlana Mayboroda \\
		School of Mathematics,	University of Minnesota, Minneapolis, MN 55455, USA} 
	\email{svitlana@math.umn.edu}
	
	\author{B. Poggi}
	
	\address{Bruno Poggi \\
		Department of Mathematics,	Universitat Autònoma de Barcelona, Barcelona, Catalonia}
	\email{bgpoggi.math@gmail.com}
	
	\thanks{This material is based upon work supported by National Science Foundation under Grant No.\ DMS-1440140 while the  authors were in residence at the MSRI in Berkeley, California, during the Spring 2017 semester. S. Bortz was  supported by the Simons foundation grant ``Travel support for Mathematicians'' (grant number 959861). S. Bortz and S. Mayboroda were partly supported by NSF INSPIRE Award DMS-1344235.  S. Hofmann was supported by NSF grant DMS-2000048. S. Mayboroda and B. Poggi were  supported in part by the NSF RAISE-TAQS grant DMS-1839077 and the Simons foundation grant 563916, SM.  B. Poggi was also supported by the University of Minnesota Doctoral Dissertation Fellowship, and by the European Research Council (ERC) under the European Union's Horizon 2020 research and innovation programme (grant agreement 101018680). S. Bortz would like to thank Moritz Egert for some helpful conversations. B. Poggi would like to thank Max Engelstein for some helpful discussions.}
	
\begin{abstract} This is the final part of a series of papers where we  study perturbations of divergence form second order elliptic operators $-\div A \nabla$ by first and zero order terms, whose complex coefficients lie in critical spaces, via the method of layer potentials. In particular, we show that the $L^2$ well-posedness (with natural non-tangential maximal function estimates) of the Dirichlet, Neumann and regularity problems for complex Hermitian, block form, or constant-coefficient divergence form elliptic operators in the upper half-space are all stable under such perturbations. Due to the lack of the classical De Giorgi-Nash-Moser theory in our setting, our method to prove the non-tangential maximal function estimates relies on a completely new argument: We obtain a certain weak-$L^p$ ``$N<S$'' estimate, which we eventually couple with square function bounds,  weighted extrapolation theory, and a bootstrapping argument to recover the full $L^2$ bound. Finally, we show the existence and uniqueness of solutions in a relatively broad class. 
	
As a corollary, we claim the first results in   an unbounded domain concerning the $L^p$-solvability of boundary value problems for the magnetic Schr\"odinger operator $-(\nabla-i{\bf a})^2+V$ when the magnetic potential ${\bf a}$ and the electric potential $V$ are accordingly small in the norm of a scale-invariant Lebesgue space.
\end{abstract}
	
	\maketitle

	\date{\today}


	{
		\hypersetup{linkcolor=toc}
		\tableofcontents
	}
	\hypersetup{linkcolor=hyperc}

\section{Introduction}

This is the second in a series of two papers, where we study the $L^2$ Dirichlet, Neumann and regularity problems for critical perturbations of second-order divergence-form equations by lower order terms. In the first paper \cite{bhlmp}, we obtained square function estimates and uniform $L^2$ estimates on slices of the layer potentials (see Theorem \ref{BHLMP1-Main.thrm}). In the present manuscript, we complete the $L^2$ theory for these operators, by proving the   non-tangential maximal function estimates (Theorem \ref{NTmaxestimatesthrm.thrm}), as well as the existence and uniqueness of solutions to the boundary value problems (Theorems \ref{L2Solv2.thrm} and \ref{L2Solv.thrm}).

Consider operators of the form
\begin{equation}\label{opdef.eq}
\cL :=-\text{div }(A\nabla+B_1)+B_2\cdot\nabla+V
\end{equation}
defined on $\bb R^{n+1}= \{(x,t): x\in \bb R^n, t\in \bb R\}$, $n \ge 3$, where $A=A(x)$ is an $(n+1)\times(n+1)$ matrix of $L^{\infty}$ complex coefficients, defined on $\bb R^n$ (independent of $t$) and satisfying a uniform ellipticity condition:
\begin{equation}\label{elliptic} 
\lambda|\xi|^2\leq~\real\langle A(x)\xi,\xi\rangle:=\real\sum\limits_{i,j=1}^{n+1}A_{ij}(x)\xi_j\overline{\xi_i},\qquad\Vert A\Vert_{L^{\infty}(\bb R^n)}\leq\frac1{\lambda},
\end{equation}
for some $\lambda>0$, and for all $\xi\in\bb C^{n+1}, x\in\bb R^n$. The first order complex coefficients $B_1=B_1(x),B_2=B_2(x)\in\big(L^n(\bb R^n)\big)^n$ (independent of $t$) and the complex potential $V=V(x)\in L^{\frac n2}(\bb R^n)$ (again independent of $t$) are such that
\begin{equation}\label{eq.small}
\max\big\{\Vert B_1\Vert_{L^n(\bb R^n)},~\Vert B_2\Vert_{L^n(\bb R^n)},~\Vert V\Vert_{L^{\frac n2}(\bb R^n)}\big\}\leq \rho
\end{equation}
for some $\rho$ depending on dimension and the ellipticity of $A$ in order to ensure the accretivity of the form associated to the operator $\cL$ on the space
$$Y^{1,2}(\ree) : = \big\{ u \in L^{2^*_{n+1}}(\ree): \nabla u \in L^2(\ree)\big\}$$
equipped with the norm
$$\lVert u \rVert_{Y^{1,2}(\ree)} := \lVert u \rVert_{L^{2^*_{n+1}}(\ree)} + \lVert \nabla u \rVert_{L^2(\ree)},$$
where $2^*_{n+1} := \tfrac{2(n+1)}{n-1}$ is the Sobolev exponent in $n+1$ dimensions.  Our smallness assumption on the critical norms (\ref{eq.small}) of the lower order terms is very natural in the $t$-independent setting, as it implies a small \emph{Carleson perturbation} assumption; see Remark \ref{rm.natural}.

We remark at this stage that, under the current hypotheses on the coefficients, the potential term $V$ may be absorbed into the drift terms $B_1, B_2$ by writing $V= -\div \nabla (-\Delta)^{-1/2}(-\Delta)^{-1/2}V$, and as a consequence, we will not explicitly mention the potential term in any of our estimates. A more detailed account of this reduction can be found in \cite[Lemma 2.17]{bhlmp}.

We interpret solutions of $\cL u=0$ in the weak sense; that is, $u\in W^{1,2}_{\loc}(\ree)$ is a solution of $\cL u=0$ in $\om\subset \ree$ if for every $\varphi \in C_c^\infty(\om)$, it holds that
\begin{equation}\nonumber
\dint_{\ree} \Big( (A\nabla u + B_1 u)\cdot\overline{\nabla \varphi} + B_2\cdot \nabla u \overline{\varphi} \Big) =  0.
\end{equation}

Our methods here are of a perturbative nature and we construct solutions via layer potentials. We denote $\mathcal{S}$ and $\mathcal{D}$ as the (abstract) single and double layer potentials, respectively. Further, smallness on the lower order terms $B_1$, $B_2$, and $V$ is imposed (depending on dimension and ellipticity) in order to guarantee boundedness of these layer potentials in natural Banach spaces\footnote{This means $L^p$ bounds on certain square functions or non-tangential maximal functions.}, but it is important to note that no structural assumptions are made on the matrix $A$ until we begin to prove existence and uniqueness (starting with Section 7). When we do prove existence and uniqueness, we must ensure that the ``boundary operators'' for the layer potentials associated to the operator\footnote{More generally, $L_0$ could be an operator of the same form as $L$, whose lower order terms are small enough to ensure that the  non-tangential maximal function estimates hold (see Theorem \ref{L2Solv.thrm}).}   $L_0 = \div A \nabla$ are invertible and we must also ensure that the lower order terms are small depending on dimension, ellipticity and operator norm of the inverses of the boundary operators (see Theorem \ref{L2Solv.thrm}). In special cases, where $A$ has some structural assumption, such as being Hermitian, we already know that the operator norm of the inverses of the boundary operators is uniformly bounded in terms of dimension and ellipticity so this restriction is redundant (see Theorem \ref{L2Solv2.thrm}).

The first paper in this series \cite{bhlmp} established $L^2$ square function and ``slice'' estimates for layer potential operators. The following theorem summarizes these results. We denote by $\sl$ and $\mathcal{D}^{\L,+}$ the single and double layer potentials, respectively (see Definitions \ref{def.sl} and \ref{def.dl}). 
\begin{theorem}[\cite{bhlmp}]\label{BHLMP1-Main.thrm} Let 
\[\cL :=-\div (A\nabla+B_1)+B_2\cdot\nabla+V\]
where $A, B_1,B_2, V$ are as above. There exists $\widetilde{\rho}_1>0$ depending on dimension and the ellipticity of $A$ such that if 
\[\max\big\{\Vert B_1\Vert_{L^n(\bb R^n)},~\Vert B_2\Vert_{L^n(\bb R^n)},~\Vert V\Vert_{L^{n/2}(\bb R^n)}\big\} <  \widetilde{\rho}_1,\]
then the following estimates hold for the single and double layer potentials.
\begin{enumerate}[(i)] 
	\item\label{item.est1} \[\dint_{\ree_+} \big|t^m\partial_t^{m}\nabla\mathcal{S}_t^{\cL} f(x)\big|^2\,  \frac{dx \, dt}{t} \leq C_m \| f \|_{L^2(\rn)}^2, \quad \text{for each } m \in \mathbb{N},\]
	\item\label{item.est2} \[\sup_{\tau > 0} \| \tr_\tau \mathcal{S}^{\L} f\|_{L^{\frac{2n}{n-2}}(\rn)}+ \sup_{\tau > 0}\Vert\Tr_{\tau}\nabla \mathcal{S}^{\L} f\Vert_{L^2(\bb R^n)}\le C \|f\|_{L^2(\rn)},\]
	\item\label{item.est3} \[\dint_{\ree_+} \big|t^m\partial_t^{m - 1} \nabla \mathcal{D}_t^{\cL,+} f(x)\big|^2\,  \frac{dx \, dt}{t} \leq C_m \| f \|_{L^2(\rn)}^2 \quad \text{for each } m \in \mathbb{N},\]
	\item\label{item.est4} \[ \sup_{\tau > 0} \|\tr_\tau \mathcal{D}^{\L,+} f\|_{L^2(\rn)} \le  C \| f \|_{L^2(\rn)}.\]
\end{enumerate}
Here, $C$ depends on dimension and ellipticity, while $C_m$ depends on $m$, dimension, and ellipticity.
\end{theorem}

Items \ref{item.est3} and \ref{item.est4} were not treated explicitly in \cite{bhlmp}. However, using the identity in Lemma \ref{lem-double-layer} of the present article, and the square function estimates for the single layer obtained in \cite[Theorem 1.3]{bhlmp}, estimate \ref{item.est3} follows. Finally, \ref{item.est4} then follows from \ref{item.est3} and \cite[Theorem 6.17]{bhlmp}.

In fact, the analysis in \cite{bhlmp} (primarily these estimates) along with the existence and uniqueness sections of this work are enough to prove the existence and uniqueness for solutions with square function estimates. On the other hand, we desire to have the  more natural non-tangential maximal function estimates for the single and double layer potentials, under (essentially) the same hypothesis as in Theorem \ref{BHLMP1-Main.thrm}. The non-tangential maximal function estimates are significantly stronger than the uniform slice estimates \ref{item.est2} and \ref{item.est4}.  This is where we place a significant amount of our effort in this work. Along the way, we will further be able to extrapolate the $L^2$ estimates to $L^p$ estimates in a window around $2$. We prove:

\begin{theorem}\label{NTmaxestimatesthrm.thrm}
Let  $\cL :=-\div (A\nabla+B_1)+B_2\cdot\nabla+V$, where $A, B_1,B_2, V$ are as above. There exist  $\rho_1 \in (0, \widetilde{\rho}_1)$ and $\ep_1>0$ depending on $n$ and $\lambda$ such that if 
\[\max\big\{\Vert B_1\Vert_{L^n(\bb R^n)},~\Vert B_2\Vert_{L^n(\bb R^n)},~\Vert V\Vert_{L^{\frac n2}(\bb R^n)}\big\} < \rho_1,\]
then the following estimates hold for each $p\in(2-\ep_1,2+\ep_1)$:
\begin{enumerate}[(i)]
	\item $\|\widetilde{N}_2(\nabla \mathcal{S}^{\L} f)\|_{L^p(\rn)} \le C \|f\|_{L^p(\rn)},$
	\item $\|\widetilde{N}_2(\mathcal{D}^{\L,+} f)\|_{L^p(\rn)} \le C \|f\|_{L^p(\rn)}$.
\end{enumerate}
Here, the constant $C$ depends only on dimension and ellipticity, and $\widetilde{N}_2$ is the modified non-tangential maximal function (see Definition \ref{def-ntmax} below).
\end{theorem}

The idea to proving Theorem \ref{NTmaxestimatesthrm.thrm} begins with a weak ``$N<S$'' result; namely, we show a  weak-$L^p$ bound ($L^{p,\infty}$ bound) for the non-tangential maximal function in terms of the $L^p$ norms of the vertical and conical square functions (see Lemma \ref{lem-travelling-down-ntmax-weak-lp}). Then, interpolation will show that Theorem \ref{NTmaxestimatesthrm.thrm} holds   provided that the vertical and conical square functions are bounded in $L^p$ for an open interval (in $p$) around $2$. The starting point for obtaining such bounds  for the square functions  is to prove {\it general} bounds for operators with sufficient off-diagonal decay which satisfy a local reverse-H\"older inequality using the extrapolation theory from weighted norm inequalities \cite{rdf84, gcrdf85, drdf86, cump, cump12} (see Lemmas \ref{lem-extrapolation-conical} and \ref{lem-extrapolation-vertical}). Arguments similar to ours have been used in \cite{pa19} to treat square function estimates for operators built out of the heat or Poisson semigroups associated to an elliptic operator; however, in our case we must grapple with the added difficulty of having very mild off-diagonal decay. On the other hand, the local energy inequality for the equation (the Caccioppoli inequality) allows us to obtain the necessary off-diagonal decay for related operators with added (transversal) derivatives. Having done so, our remaining task is to ``remove'' these additional derivatives, a process which we call ``traveling down''. Due to its definition, this process for the vertical square function is a relatively simple integration by parts computation. For the conical square function, the additional spatial average impedes the simple integration by parts and our argument for this object requires the boundedness of the non-tangential maximal function with the {\it same} number of derivatives. Luckily, our Lemma \ref{lem-travelling-down-ntmax-weak-lp}, when combined with Proposition \ref{prop-improvement-ntmax}, gives that the non-tangential maximal function bounds (for this family of operators) depend on square functions with {\it more}\footnote{Note that in Lemma \ref{lem-travelling-down-ntmax-weak-lp}, we may use that  $\| \v(\Theta_{t,1} f)\|_{L^p(\rn)} \lesssim_{m} \| \v(\Theta_{t,m+1} f)\|_{L^p(\rn)}$, by the aforementioned integration by parts argument. The subscript $m$ refers to the number of transversal derivatives.} derivatives. This allows us to employ a two-step induction scheme where one alternates between bounding the $L^p$ norm for a non-tangential maximal function by the $L^p$ norm of square functions (with more derivatives) and then bounding the $L^p$ norm of the conical square function by the $L^p$ norm of a non-tangential maximal function (with the same number of derivatives). Thus, in finitely many steps, we remove these additional derivatives. (Recall we can start this process, that is, obtain $L^p$ bounds for the vertical and conical square functions, by introducing enough transversal derivatives.)

With the square function and non-tangential maximal function bounds for layer potentials in hand, we turn our attention to the solvability of the following boundary value problems with data in $L^p$ spaces:
We consider the Dirichlet problem  
\begin{equation}\label{eq.d2}
	\Di_p \begin{cases}
		\cL u = 0\qquad\text{in } \ree_+,\\[1mm]
		\lim_{t \to 0}u(\cdot ,t) = f\quad\text{strongly in } L^p(\rn) \text{ and }\quad u\longrightarrow f\quad\text{non-tangentially}\footnotemark, \\[1mm]
		\mntm_2 u\in L^p(\bb R^n),
	\end{cases}
\end{equation}
\footnotetext{Since the solutions $u$ do not satisfy pointwise bounds, non-tangential convergence is also understood in an averaged sense; see Definition \ref{def-ntmax}.} the Neumann problem
\begin{equation}\label{eq.n2}\Ne_p \begin{cases}
		\cL u = 0\qquad\text{in } \ree_+, \\[1mm]
		\tfrac{\partial u}{\partial \nu^{\cL}}: = -e_{n+1}(A\nabla u + B_1u)(\cdot,0) = g \in L^p(\rn),\footnotemark \\[1mm]
		\mntm_2 (\nabla u)\in L^p(\bb R^n), 
	\end{cases}
\end{equation}
\footnotetext{The boundary data is achieved in the distributional sense, see Section 2.} 
and the regularity problem
\begin{equation}\label{eq.r2}\Reg_p \begin{cases}
		\cL u = 0 \qquad\text{in } \ree_+,\\[1mm]
		u(\cdot, t) \to f\qquad\text{weakly in } Y^{1,p}(\rn) \text{ and non-tangentially},\\[1mm]
		\mntm_2 (\nabla u)\in L^p(\bb R^n), 
	\end{cases}
\end{equation}

\begin{remark}
	At this stage we would like to point out a couple of things related to the definition above. First, we chose to state the boundary value problems in terms of the (modified) nontangential maximal function as this is typically the quantity of interest. Second, if the solution $u$ is given by layer potentials (as will always be the case for us), appropriate square function estimates are \textit{always} available, regardless of solvability (see Remark \ref{rm.more} and Theorem \ref{BHLMP1-Main.thrm}).
\end{remark}
 
We are ready to state the main result of this series of articles. 
\begin{theorem}\label{L2Solv2.thrm}
Let $\cL_0$ be a divergence form operator of the form
\[\cL_0 :=-\div A_0 \nabla,\]
where $A_0$ is either Hermitian, block form or constant. Then there exist  $\rho_0 > 0$ and $\ep_0>0$ depending only on dimension and ellipticity such that if
\[\cL_1 = -\div ((A + M)\nabla+B_1)+B_2\cdot\nabla+V,\]
and
\[ \max\big\{\|M\|_{L^\infty(\rn)}, \Vert B_1\Vert_{L^n(\bb R^n)},~\Vert B_2\Vert_{L^n(\bb R^n)},~\Vert V\Vert_{L^{\frac n2}(\bb R^n)}\big\}< \rho_0 \]
then for each $p\in(2-\ep_0,2+\ep_0)$, the problems $\Di_p$,  $\Ne_p$, and $\Reg_p$ are uniquely\footnote{See Remark \ref{amongunique.rmk}.}  solvable\footnote{Solvability throughout this paper means that we have accompanying $L^p$ bounds for the non-tangential maximal function.} for the operator $\cL_1$, and the solutions can be represented by layer potentials.
\end{theorem}

Remark that the previous theorem gives the first solvability results for boundary value problems with control on the (modified) non-tangential maximal function for second-order elliptic operators with complex-valued lower-order terms.

Our most general theorem concerning   boundary value problems with $p=2$, is as follows\footnote{Theorem \ref{L2Solv.thrm} also has an appropriate analogue for $p$ sufficiently near $2$, but then the boundary operators have different domains and ranges; see Section \ref{sec.lp}.}.
\begin{theorem}\label{L2Solv.thrm}
Let $\cL_0$ be an operator of the form 
\[\cL_0 :=-\div (A\nabla+B_1)+B_2\cdot\nabla+V\]
where $A, B_1,B_2, V$ are as above and 
\[\max\big\{\Vert B_1\Vert_{L^n(\bb R^n)},~\Vert B_2\Vert_{L^n(\bb R^n)},~\Vert V\Vert_{L^{n/2}(\bb R^n)}\big\} <   \rho_1,\]
where $\rho_1$ is as in Theorem \ref{NTmaxestimatesthrm.thrm}. Suppose further that the associated boundary operators\footnote{See Section \ref{existence.sect} for the definitions of the operators $\sl_0$, $K$, and $\widetilde K$.}
\begin{gather}\nonumber
\mathcal{S}^{\cL_0}_0: L^2(\rn) \to Y^{1,2}(\rn), \pm\frac{1}{2}I+ \widetilde{K}^{\cL_0}: L^2(\rn) \to L^2(\rn), 
\\ \mp\frac{1}{2}I + K^{\cL_0}: L^2(\rn) \to L^2(\rn)
\end{gather}
are all invertible. Then the boundary value problems $\Di_2, \Ne_2, \Reg_2$ are uniquely\footnote{Again, see Remark \ref{amongunique.rmk}.} solvable for the operator $\cL_0$, with solutions given by the appropriate layer potentials. 

Moreover, there exists $\rho = \rho(\cL_0) > 0$ such that if
\[\cL_1 = -\div (\widetilde{A}\nabla+\widetilde{B}_1)+ \widetilde{B}_2\cdot\nabla+\widetilde{V}\]
with $\widetilde A, \widetilde{B}_1,\widetilde{B}_2, \widetilde{V}$ as above and satisfying
\[\max\big\{\Vert \widetilde{A} - A\Vert_{L^\infty(\rn)}, \Vert \widetilde{B}_1 - B_1\Vert_{L^n(\bb R^n)},~\Vert \widetilde{B}_2 - B_2\Vert_{L^n(\bb R^n)},~\Vert \widetilde{V} - V\Vert_{L^{\frac n2}(\bb R^n)}\big\}\leq   \rho,\]
then the   boundary operators, $\mathcal{S}^{\cL_0}_0, \pm\frac{1}{2}I+ \widetilde{K}^{\cL_0},  \mp\frac{1}{2}I + K^{\cL_0}$, are invertible 
and the   problems $\Di_2, \Ne_2, \Reg_2$ are uniquely solvable for the operator $\cL_1$, with the corresponding layer potential representations.

Here, the constant $\rho(\cL_0)$ is chosen with two constraints. The first is to ensure that $\widetilde{A}$ has ellipticity constant less than twice that for $A$ and
\[\max\big\{\Vert \widetilde{B}_1\Vert_{L^n(\bb R^n)},~\Vert \widetilde{B}_2\Vert_{L^n(\bb R^n)},~\Vert \widetilde{V}\Vert_{L^{\frac n2}(\bb R^n)}\big\} <  \rho'_1\]
where $\rho'_1$ is as in Theorem \ref{NTmaxestimatesthrm.thrm} for matrices with ellipticity twice that of $A$. The second constraint depends on the operator norms of the inverses of $\mathcal{S}^{\cL_0}_0, \pm\frac{1}{2}I+ \widetilde{K}^{\cL_0},  \mp\frac{1}{2}I + K^{\cL_0}$.
\end{theorem}

The solvability result for Theorem \ref{L2Solv.thrm} is proved in Theorem  \ref{existence.thm}, while the uniqueness results are argued in Theorems \ref{Dgooduniquethrm.thrm}, \ref{Ngooduniquethrm.thrm}, and \ref{Rgooduniquethrm.thrm}. These theorems together also resolve  the case $p=2$ of Theorem \ref{L2Solv2.thrm}. The case $p\neq2$ of Theorem \ref{L2Solv2.thrm} will be addressed in Section \ref{sec.lp}.

Theorem \ref{L2Solv.thrm} is modeled after the results in \cite{AAAHK}, where a purely second order perturbation theory is developed, while in the presence of De Giorgi-Nash-Moser estimates for solutions of $\L$ and $\L^*$. We remark that it is only the failure of these estimates that prevents the application of the results in \cite{AAAHK} to a complex elliptic operator, such as $\L_0$ in block form or Hermitian. In this sense our results also bridge the gap between the ``standard" layer potential approach and the abstract, first-order approach used in \cite{Aus-Axel-McI} and \cite{Aus-Axel-Hof}, to obtain perturbation results. 

\begin{remark}\label{rm.more} The solutions to the problems $\Di_p$, $\Ne_p$, and $\Reg_p$ satisfy two further properties: square function estimates and decay at infinity. More precisely, the solution to $\Di_p$  satisfies that
\begin{gather*}
\lim_{t\ra\infty}u(\cdot,t)=0\quad\text{in the sense of distributions, and} \\
\Vert\bb S(t\nabla u)\Vert_{L^p(\bb R^n)}<\infty,
\end{gather*}
where $\bb S$ is the conical  square function from Definition \ref{def-square-functions}. The solutions to the Neumann and regularity problems satisfy:
\begin{gather*}
	\lim_{t\ra\infty}\nabla u(\cdot,t)=0\quad\text{in the sense of distributions, and} \\
	\Vert\bb S(t\partial_t\nabla u)\Vert_{L^p(\bb R^n)}<\infty.
\end{gather*}
\end{remark}

\begin{remark}\label{amongunique.rmk} Uniqueness, under the background hypothesis of invertible layer potentials and sufficient smallness of the lower order terms, is established among what we call ``good $\mathcal{D}$ solutions'' (in the case of $\Di_2$) and ``good $\mathcal{N}/\mathcal{R}$ solutions'' (in the case of $\Ne_2$ and  $\Reg_2$). We show that non-tangential maximal function estimates or square function estimates imply that solutions are ``good''. For instance, if $p=2$ and under the aforementioned background hypothesis, suppose that $u \in W^{1,2}_{\loc}(\ree_+)$ solves
\[\operatorname{(\widetilde{D})_2}\begin{cases}
\cL u = 0\qquad\text{in } \ree_+,\\[1mm]
\lim_{t \to 0}u(\cdot ,t) = f\qquad\text{strongly in } L^2(\rn),
\end{cases}
\] 
for some $f \in L^2$, and suppose that $u$ has {\it one} of the following properties:
\begin{itemize}
	\item $u$ is a good $\mathcal{D}$ solution,
	\item $\|\widetilde N_2u\|_{L^2(\rn)} < \infty$ or
	\item $\iint_{\ree_+}t |\nabla u(x,t)|^2 \, dx \, dt < \infty$.
\end{itemize} 
Then $u$ is the unique such solution and has the other two properties. In the case of the Neumann problem, our solutions are unique modulo constants if the operator $\mathcal{L}$ annihilates constants.
\end{remark}

Let us mention some well-known operators in mathematical physics for which our solvability results are new. For the magnetic Schr\"odinger operator $-(\nabla-i{\bf a})^2$ when ${\bf a}\in L^n(\bb R^n)^{n+1}$ is $t-$independent and has small $L^n(\bb R^n)$ norm, we have as a corollary to our Theorem \ref{L2Solv2.thrm} the first $L^p$ well-posedness results of the Dirichlet, Neumann and regularity problems on an unbounded domain. Another operator which satisfies our hypotheses is the Schr\"odinger operator $-\operatorname{div}A\nabla+V$ where $V\in L^{\frac n2}(\bb R^n)$ is $t-$independent, complex-valued, and has small $L^{\frac n2}(\bb R^n)$ norm, and thus we obtain new solvability results in this setting as well. However, we mention that, for non-negative $V$, there are solvability results for boundary value problems under the assumption that $V$ belongs to a reverse H\"older class \cite{shenn, sak, mt}.

\begin{remark}\label{rm.natural} Small Carleson perturbation conditions have been shown to preserve solvability of boundary value problems in the same $L^p$ space, at least when the operator is a  purely second-order divergence-form elliptic operator \cite{AA, Hof-May-Mour, ahmt, dp2019}. Let us point out the connection between our perturbation condition (\ref{eq.small}) and the theory of Carleson perturbations. It is easy to prove that, under the assumption that $B_1$ and $B_2$ are complex-valued and $t$-independent, if they satisfy the condition (\ref{eq.small}), then each of the measures $d\mu=(|B_1|+|B_2|)\,dx\,dt$ and $d\mu=(|B_1|^2+|B_2|^2)\,dx\,tdt$ satisfy the  Carleson measure condition
\begin{equation}\label{eq.measure}
\mu(B(x_0,r)\cap\overline{\Omega})\lesssim \rho r^n,\qquad\text{for each }x_0\in\bb R^n, r>0.
\end{equation} 
Note that this embedding of our condition (\ref{eq.small}) into the small Carleson perturbation condition (\ref{eq.measure}) does not hold for the analogue perturbative conditions on the second-order term. Indeed, suppose $A,A_0$ are $t$-independent complex matrices such that $\Vert A-A_0\Vert_{L^{\infty}(\bb R^n)}\leq\ep_0$. Then solvability results for the $L^2$ Dirichlet, Neumann, and regularity problems have been obtained   \cite{Aus-Axel-Hof, AAAHK, aam10}, but this perturbative condition does not imply the classical Carleson perturbation condition  \cite{FKP, AA, Hof-May-Mour}  that $\sup_{y\in B(x,t/2)}\frac{|A(y)-A_0(y)|^2}{t}\,dx\,dt$ is the density of a Carleson measure (and cannot imply it unless $A\equiv A_0$!).	
\end{remark}

We now make a few historical remarks concerning our results.  

Firstly, we should emphasize that all the results in this paper, as well as the ones in the first part \cite{bhlmp}, concern $t$-independent operators. While solvability results are available for $t$-dependent coefficients (see for instance \cite{Hof-May-Mour, dhp22, mpt22}) we will not concern ourselves with them at all here.

Even in this restricted setting of elliptic equations with $t$-independent coefficients the literature is vast, and we make no attempt at a complete historical account; we refer the reader to the first part of this series of papers, \cite{bhlmp}, for a more thorough overview of the area. We will now restrict attention to the works very closely related to our solvability and perturbation results.

The case of Hermitian $A_0$ is treated in \cite{Aus-Axel-McI}; in fact here the authors treat a perturbation theory analogous to our own (by a different method, and only for second order perturbations) for all three classes of $A_0$. We also note that the solvability for \emph{real} symmetric second-order equations was already known from the work in \cite{JK1} for the Dirichlet problem and \cite{KP1}, \cite{KP2} for the Neumann and regularity problems; moreover, in this setting of real equations, the above works also obtain solvability for the problems in $L^p$ for $2-\varepsilon<2\leq \infty$ in the case of Dirichlet, and $1<p<2+\varepsilon$ for Neumann and regularity. The issue of solvability by layer potentials goes back to \cite{V} for the case of the Laplacian, but the technique of using a Rellich identity to obtain the invertibility also works in the case of real symmetric \cite{KP1} or Hermitian \cite{Aus-Axel-McI} matrices.

For $A_0$ of block form, it was remarked by Kenig in \cite{kbook}, solvability is equivalent to the Kato conjecture, in the case of the regularity problem, and to boundedness of a Riesz Transform associated to the elliptic operator $\L_0$ for the Neumann problem\footnote{The Dirichlet problem is a consequence of semigroup theory, since the double layer potential in this setting is a constant multiple of the Poisson semigroup associated to $\L_0$.}. On the other hand, special results were known before, see \cite{bhlmp} and \cite{AAAHK} for a more detailed account of this. $L^p$ solvability results (via layer potentials) for perturbations of $t$-independent symmetric coefficients were obtained in \cite{hmm}.

Finally, when $A_0$ is a constant matrix, an explicit Poisson kernel is constructed in \cite{adn}, while the Dirichlet problem is solved in \cite{FJK}, and the Neumann and regularity problem in \cite{AAAHK}. See also \cite{mmm}.  
 
The works cited above dealt mostly with equations of pure second-order. The literature in the setting with lower order terms present (that is, not all of $b_1,b_2, V$ are identically $0$) is much more sparse, but has in recent years garnered a lot of attention, at least when the lower order terms are real. In \cite{hoflew}, \emph{parabolic} operators with singular drift terms $b_2$ were studied, and their results would later be applied toward $\Di_p$ for elliptic operators with singular drift terms $b_2$ in \cite{KP3} and \cite{DPP}. When $A\equiv I$, $b_1\equiv b_2\equiv0$ and $V>0$ satisfies certain conditions, Shen proved the solvability of $\Ne_p$ and $\Reg_p$ on Lipschitz domains in \cite{shenn}.    More recently,  Morris and Turner \cite{mt} proved the $L^2$ well-posedness of the Neumann and regularity problems in the  half-space setting for the Schr\"odinger operator $-\operatorname{div} A\nabla+V$ with $t-$independent Hermitian $A$ and $t-$independent potential $V$ in the reverse H\"older class $RH^{\frac n2}$.  The problems $\Di_2$ and $\Reg_2$ for equations with lower order terms have been considered by Sakellaris in \cite{sak} in bounded Lipschitz domains, under some continuity and sign assumptions on the coefficients. Solvability results and foundational estimates for solutions for the variational Dirichlet problem of equations with lower order terms on unbounded domains have been obtained by Mourgoglou in \cite{mourg}, with very lax assumptions on the (real) lower order terms. The fundamental solution for operators with lower order terms has been studied recently by \cite{bd22}, and the Neumann Green's function on Lipschitz domains for operators with lower order terms is considered in \cite{ks23}.

The paper is organized as follows. In Section 2, we review relevant preliminaries and definitions, including properties of the layer potentials and the theory of extrapolation of $A_p$ weights. In Section 3, we develop certain extrapolation theorems for both conical and vertical square functions, in the presence of sufficient off-diagonal decay. In Section 4, we use the general extrapolation results of Section 3 to obtain $L^p$ estimates for `slices' and for conical and vertical square functions of operators arising from the  layer potentials with enough transversal derivatives. In Section 5, we prove the non-tangential maximal function estimates, under the   background assumption of good square function bounds. In Section 6, we proceed to `travel down' on both the square and non-tangential maximal functions, to dispense of the hypothesis of good off-diagonal estimates.  In Sections 7 and 8, we show the existence and uniqueness, respectively, of solutions to the  boundary value problems $\Di_2$, $\Reg_2$ and $\Ne_2$, with representations of solutions via layer potentials. Finally, in Section 9, we prove the $L^p$ solvability of the Dirichlet, Neumann, and regularity problems, for $p\in(2-\ep_0,2+\ep_0)$, with $\ep_0$ small enough. 

\section{Notation and Preliminaries}\label{sec.prelim}

We recall  notation from  \cite{bhlmp}, as well as introduce  concepts to   be used throughout the article.

\begin{itemize}
\item Throughout, we assume that $n\geq 3$.   We write $(x,t)$ for the coordinates of $\bb R^{n+1}=\bb R^n\times\bb R$, where $x\in\bb R^n$ and $t\in\bb R$, and  $\reu:=\{ (x,t): x\in \rn, \, t>0\}$. The lower half-space will be denoted $\ree_-$. Similarly, for any $\tau\in \RR$ we write $\ree_\tau:=\{ (x,t)\in \ree: t>\tau\}$.
\item We always take $A=A(x)$ to be an $(n+1)\times(n+1)$ matrix of $L^{\infty}$, $t-$independent complex coefficients satisfying the ellipticity condition \eqref{elliptic}, while $B_1, B_2\in(L^n(\bb R^n))^{n+1}$ and $V\in L^{n/2}(\bb R^n)$ are complex-valued, $t-$independent (vector) functions  satisfying \eqref{eq.small}, with $\rho\ll1$.  Under these conditions,   the term $V$ can be ``hidden'' into  first-order terms $\widetilde B_1$, $\widetilde B_2$ (see \cite[Lemma 2.17]{bhlmp}); therefore, without loss of generality we will omit the zeroth order term $V$ from consideration.
\item For $-\infty\leq a<b\leq \infty$ we define the slab $\Sigma_a^b:= \{ (x,t)\in\bb R^{n+1}: a<t<b\}$. 
\item For a vector $\vec{v}=(v_1,\ldots,v_{n+1})\in \ree$, we write $\vec{v}_{\|}:= (v_1,\ldots, v_n)$, $ \vec{v}_\perp:=v_{n+1}$. For vector functions $B:\rn \to \CC^{n+1}$, we define $B_{\|}$ and $B_\perp$ analogously.
\item For a cube $Q\subset \rn$ we denote by $R_Q$ the \emph{Carleson region} above $Q$; that is, $R_Q:=Q\times(0,\ell(Q))$.
\item For $R>0$ we define $I_R:=(R,R)^{n+1}\subset \ree$, and $I_R^\pm:= I_R\cap \ree_\pm$. 
\item We denote by $\m$ the (uncentered) Hardy-Littlewood maximal function in $\rn$, and more generally for $r>0$ we define $\m_r(f):= \m(|f|^r)^{1/r}$. 
\item Given a cube $Q\subset \rn$ we denote by $Q^*$ a concentric dilate of $Q$ by a factor that depends only on $n$.
\item We denote by $I_1$ the  fractional integral of order 1 in $\rn$; that is, for nice enough $f$,
\[
I_1 f(x) = c_n \int_\rn \frac{f(y)}{|x-y|^{n-1}}\, dy.
\] 
\item We denote by $\DD$ the collection of all dyadic cubes in $\rn$, and for $t>0$ we define $\DD_t$ to be the cubes in $\DD$ which satisfy $\ell(Q)<t\leq 2\ell(Q)$. Similarly, for a cube $Q\subset \rn$ we denote by $\DD(Q)$ the collection of dyadic subcubes of $Q$.
\item For $(x,t)\in \reu$ we define the \emph{Whitney regions} $\cxt:= \big\{ (y,s)\in \reu: |x-y|<t/8, |t-s|<t/8 \big\}$. Given $x_0\in\bb R^n$, we denote by $\Gamma(x_0)$ the \emph{non-tangential cone} with vertex $x_0$, given by
\begin{equation}\label{eq.ntcone1}
\Gamma(x_0):=\{(x,t):|x-x_0|<t\}.
\end{equation}
\item Let $X$ be a topological space with Borel $\sigma-$algebra $\mathcal{B}$, and let $\mu$ be a non-negative $\sigma-$finite measure on $\mathcal{B}$.  If $p\in[1,\infty)$, we denote by $L^p(X,\mu)$ the Lebesgue space of $p-$th integrable (complex) functions $f$ on the measure space $(X,\mathcal{B}, \mu)$. We often write $L^p(\bb R^n)=L^p(\bb R^n,m_n)$, where $m_n$ is the $n-$dimensional Lebesgue measure. If $v\in L^1_{\loc}(\bb R^n)$, we write $L^p(\nu)=L^p(\bb R^n,\nu)$.
\item Given an open set $\Omega\subset\bb R^d$, $d\geq3$, we denote by $C_c^{\infty}(\Omega)$ the space  consisting of all compactly supported smooth complex-valued functions in $\Omega$. As usual, we denote $\n D=C_c^{\infty}(\bb R^{n+1})$, and we let $\n D'=\n D^*$ be the space of distributions on $\bb R^{n+1}$. The space $\n S$ consists of the Schwartz functions on $\bb R^{n+1}$, and $\n S'=\n S^*$ is the space of tempered distributions on $\bb R^{n+1}$.
\item We call a measurable function $\nu:\rn\to \RR$ a \emph{weight} if $\nu>0$ Lebesgue-a.e. on $\bb R^n$ and $\nu\in L^1_{\loc}(\rn)$. We say that $\nu$ is \emph{doubling} if the measure $\nu(x)dx $ is doubling; that is, if (with a slight abuse of notation) $\nu(2Q)\leq C_0 \nu(Q)$ for a constant $C_0>0$ and all cubes $Q\subset \rn$.
\item For $1<p<m$, the upper and lower Sobolev exponents of order 1 in $m$ dimensions are  respectively 
\[
p^*_m:= \tfrac{mp}{m-p},\qquad p_{*,m}:=\tfrac{mp}{m+p}.
\]
Sometimes, we drop $m$ from the subscript when the dimension is clear from the context. 
\item Given an open set $\Omega\subset\bb R^{n+1}$, for $p\in[1,\infty)$, we denote by $W^{1,p}(\Omega)$ the   Sobolev space of functions in $L^p(\Omega)$ whose weak gradients exist and lie in $(L^p(\Omega))^{n+1}$. We endow this space with the norm
\[
\Vert u\Vert_{W^{1,p}(\Omega)}:=\Vert u\Vert_{L^p(\Omega)} +\Vert\nabla u\Vert_{L^p(\Omega)}.
\]
We define $W_0^{1,p}(\Omega)$ as the completion of $C_c^{\infty}(\Omega)$ in the above norm. We shall have occasion to discuss the homogeneous Sobolev spaces as well: by $\dt W^{1,p}(\Omega)$ we denote the space of functions in $L^1_{\loc}(\Omega)$ whose weak gradients exist and lie in $L^p(\Omega)$. We equip this space with the seminorm $|u|_{\dt W^{1,p}(\Omega)}:=\Vert\nabla u\Vert_{L^p(\Omega)}$, and, if $\partial\Omega$ is sufficiently nice,  we point out that  $\dt W^{1,p}(\Omega)$ coincides with the completion of the quotient space  $C^{\infty}(\Omega)/\bb C$ in the $|\cdot|_{\dt W^{1,p}(\Omega)}$ (quotient) norm.
\end{itemize}

\begin{definition}\label{def-square-functions}
If $F:\reu\to \CC$, we define the \emph{conical square function} of $F$ as 
\begin{equation}\nonumber
\s F(x):= \Big( \dint_{\Gamma(x)} |F(y,t)|^2 \, \dfrac{dydt}{t^{n+1}} \Big)^{1/2},
\end{equation}
where $\Gamma(x):=\{ (y,t)\in \reu: |x-y|<t\}$. Similarly, we define the \emph{vertical square function} of $F$ as
\begin{equation}\nonumber
\v F (x):= \Big( \int_0^\infty |F(x,t)|^2\, \dfrac{dt}{t} \Big)^{1/2}.
\end{equation}
\end{definition}

\begin{remark} In the definition of $\s$, we could have chosen a different aperture; that is, for $\eta>0$, we can set
\begin{equation}\nonumber
\s_\eta F(x)= \Big(\dint_{|x-y|<\eta t} |F(y,t)|^2\, \dfrac{dydt}{t^{n+1}}\Big)^{1/2}.
\end{equation}
It is well-known that different apertures give rise to objects with equivalent $L^p$ norms  (see for instance {\cite[Proposition 4]{cms}} for the unweighted case and {\cite[Proposition 4.9]{chmpa}} for the weighted one).
\end{remark}

\noindent In contrast to the $L^2$ case, if $p\neq 2$ the ($L^p$ norms of) conical and vertical square functions are not equivalent.

\begin{proposition}[{\cite[Proposition 2.1]{ahm}}]\label{prop-prelim-comparability-square-functions}
Let $F:\reu \to \CC$ be measurable.
\begin{enumerate}[(i)]
	\item If $0<p\leq 2$ then $\| \v(F)\|_\lprn \lesssim_{n,p} \| \s(F)\|_\lprn$. 
	\item If $2\leq p<\infty$ then  $	\| \s(F)\|_\lprn \lesssim_{n,p} \| \v(F)\|_\lprn$. 
\end{enumerate}
\end{proposition}

\begin{definition}[Non-tangential Maximal Functions]\label{def-ntmax}
For  $F:\ree\to \CC$ and $q>0$, let 
\begin{equation}\label{eq.avg}
a_q(F)(x,t):= \Big(\,\fiint_{\cxt} |F(y,s)|^q\, dyds\Big)^{1/q}.
\end{equation}
We define the \emph{non-tangential maximal function} of $F$ as $$\mathcal{N}(F)(x):= \sup_{\Gamma(x)} |F|,$$ and $\mntmq(F)(x):= \mathcal{N}(a_q(F))(x)$. We also define the \emph{lifted modified non-tangential maximal function}, for $\varepsilon>0$, as $\mntmq^\varepsilon(F)(x):= \sup_{\substack{ |x-y|<t-\varepsilon\\ t>\varepsilon}} a_q(F)(y,t)$. Similarly, we define a truncated version of the non-tangential maximal function as $\mntmq^{(\varepsilon)}(F)(x):= \sup_{\substack{|x-y|<t\\ t>\varepsilon}} a_q(F)(y,t)$. Given a measurable function $g$ on $\bb R^n\times\{t=0\}$, we say that $F\longrightarrow g$ \emph{non-tangentially} if for almost every $x\in\bb R^n$, we have that
\begin{equation}\label{eq.ntconvavg}
\lim_{\substack{Y \to x\\ Y \in \Gamma(x)}}\widetilde{F}(Y) = g(x),
\end{equation}
where $\Gamma(x)$ is the non-tangential cone defined in (\ref{eq.ntcone1}), and $\widetilde F(z,t):=\fiint_{\mathcal{C}_{x,t}} F(y,s)\,dy\,ds$.
\end{definition}

We now prove a result on the boundary behavior of solutions, under the assumption that we have good control of a modified non-tangential maximal function.

\begin{proposition}\label{RegNTconverge.prop}
	Let $u\in W^{1,2}_{\loc}(\rn)$ solve $\L u=0$ in $\reu$. Then $u$ converges non-tangentially  at every $x\in\bb R^n$ where $\mntmo(\nb u)(x)<\infty$, in the sense that for any such $x\in\bb R^n$, the limit in (\ref{eq.ntconvavg}) exists and is finite.
\end{proposition}
\begin{proof}
	We   follow   \cite[Theorem 3.1(a)]{KP1}, with modifications due to   lack of pointwise estimates for $u$. 	Let $x \in \rn$ be such that $\mntmo(\nabla u)(x) <\infty$. 	Our goal to show that for $Y, Z \in \Gamma(x) \cap B(x,r)$ we have 
	\begin{equation}\label{ntkp1.eq}
		|\widetilde{u}(Y) - \widetilde{u}(Z)| \le Cr\mntmo(\nabla u)(x),
	\end{equation}
	from which we may easily establish (via the Cauchy criterion) that $\underset{\substack{Y \to x\\ Y \in \Gamma(x)}}{\lim}\widetilde{u}(Y)$	exists, and consequently define $g(x)$ to be the limit. Write $Y = (y, t_1)$ and $Z = (z,t_2)$. Then, to establish \eqref{ntkp1.eq}, it is enough that
	\begin{equation}\label{ntkp2.eq}
		\max\big\{|\widetilde{u}(Y) - \widetilde{u}(x,t_1)|~,~ |\widetilde{u}(Z) - \widetilde{u}(x,t_2)|\big\} \le Cr\mntmo(\nabla u)(x),
	\end{equation}
	and
	\begin{equation}\label{ntkp3.eq}
		|\widetilde{u}(x,t_2) -  \widetilde{u}(x,t_1)| \le Cr\mntmo(\nabla u)(x).
	\end{equation}
	To prove \eqref{ntkp2.eq} and \eqref{ntkp3.eq} we use the following fact. 
	\begin{claim}\label{ntkpcl1.cl}
		For $X \in \ree$ and $r > 0$ let $I(X,r) := \{W \in \ree: |X -W|  < r\}$ be the open cube with center $X$ and side length $2r$. Let $I_i = I(X_i,r_i)$, $i = 1,2$, $\om \subset \ree$ open with $I_i \subset \om$ $i = 1,2$ and $\varphi \in W^{1,2}(\om)$. If $\alpha \in [0,2)$ and $|X_1 - X_2| \le \alpha \min\{r_1,r_2\}$ then
		\begin{equation*}
			\Big|\fint_{I_1} \varphi \,dx \, dt - \fint_{I_2} \varphi \, dx\,dt \Big| \le C_\alpha \max\{r_1,r_2\} \big( \max\big\{\tfrac{r_1}{r_2},\tfrac{r_2}{r_1}\big\} \big)^\frac{n+1}{2}\max_{i = 1,2}\Big(\fint_{I_i} |\nabla \varphi|^2 \Big)^{1/2},
		\end{equation*}
		where $C_\alpha = C(n, \alpha)$. In particular, if $r_1 \approx r_2 \approx r$ then
		\begin{equation}\label{closecubepoincare.eq}
			\Big|\fint_{I_1} \varphi \, dx\,dt - \fint_{I_2} \varphi \, dx\,dt \Big| \le C_\alpha r\max_{i = 1,2}\Big(\fint_{I_i} |\nabla \varphi|^2 \Big)^{1/2},
		\end{equation}
		where $C_\alpha$ depends on the implicit constants in the expression $r_1 \approx r_2 \approx r$.
	\end{claim}
	\emph{Proof of Claim \ref{ntkpcl1.cl}.} Let $X_3 = \tfrac{X_1 + X_2}{2}$, then $|X_3 - X_i| < \frac{\alpha r_i}{2}$,  $i = 1,2$, 	and hence $I_3 = I(x_3, r) \subset I_i$, $i = 1,2$ for $r = (1- \alpha/2)\min\{r_1,r_2\}$. It follows from the triangle inequality and the Poincar\'e inequality that
		\begin{multline*}
			\Big|\fint_{I_1} \varphi - \fint_{I_2} \varphi \Big|  \le  \Big|\fint_{I_1} \varphi - \fint_{I_3} \varphi \Big| +  \Big|\fint_{I_1} \varphi - \fint_{I_3} \varphi \Big|
			 \le  2\max_{i = 1,2}\Big(\fint_{I_3} \Big|\varphi - \fint_{I_i} \varphi \Big|^2 \Big)^{1/2}
			\\  \le C_\alpha\big( \max\big\{\tfrac{r_1}{r_2},\tfrac{r_2}{r_1}\big\} \big)^\frac{n+1}{2}\max_{i = 1,2}\Big(\fint_{I_i} \Big|\varphi - \fint_{I_i} \varphi \Big|^2 \Big)^{1/2}
			  \le C_\alpha \max\{r_1,r_2\} \big( \max\big\{\tfrac{r_1}{r_2},\tfrac{r_2}{r_1}\big\} \big)^\frac{n+1}{2}\max_{i = 1,2}\Big(\fint_{I_i} |\nabla \varphi|^2 \Big)^{1/2}.
		\end{multline*}
	
	Now let us prove \eqref{ntkp2.eq} for the term with $Y$ (the proof for the term with $Z$ is identical). Note that $|x - y| \le t_1 < r$ since $Y \in \Gamma(x) \cap B(x,r)$. Let 
	$I_2 = I(z,t_1/2)$, for $z = (x + y)/2$, then $|z - x| = |z - y| \le t_1/2$. This allows us to apply \eqref{closecubepoincare.eq} with $\varphi = u$, $I_1 = I(x,t_1/2)$, and  $I_1 = I(y,t_1/2)$   to obtain \eqref{ntkp2.eq}.
	
	We turn our attention to \eqref{ntkp3.eq} and we assume, without loss of generality, that $t_1 \le t_2$. Let $a = 2/3$, $s_k = a^kt_2$ for $k = 0,1,\dots, K$, where $K = \max\{k: a^k t_2 \ge t_1\}$. Notice that for $k = 0, \dots K -1$, $|s_k - s_{k +1}| = \frac{s_{k +1}}{2} = \min\Big\{\frac{s_k}{2}, \frac{s_{k +1}}{2}\Big\}$. 	Defining $s_{K + 1} = t_1$, we see that the choice of $K$ guarantees that $|s_{K}- s_{K+1} | \le \min\Big\{\frac{s_{K}}{2}, \frac{t_1}{2}\Big\}$. 	Set $I_k:= I((x,s_k), \tfrac{s_k}{2})$, $k = 0,\dots K + 1$, then the previous two inequalities allow us to apply \eqref{closecubepoincare.eq} with $\varphi = u$ and the consecutive cubes $I_k$ and $I_{k +1}$, $k = 0, 1, \dots, K$ (in place of $I_1$ and $I_2$ therein). One then obtains
	\begin{multline*}
		|\widetilde{u}(x,t_1) - \widetilde{u}(x,t_2)| \le |\widetilde{u}(x,s_{K  + 1}) - \widetilde{u}(x,s_{K})| + \sum_{k = 0}^{K - 1}|\widetilde{u}(x,s_{k}) - \widetilde{u}(x,s_{k +1})|
		\\ \lesssim t_1 \mntmo(\nabla u)(x) + t_2 \sum_{k = 0}^{K - 1} a^k \mntmo(\nabla u)(x)
		 \lesssim r\mntmo(\nabla u)(x),
	\end{multline*}
	as desired (since $\sum_{k \ge 0} a^k = 3$).
\end{proof}

\begin{definition}[CLP Family]\label{def-clp-family}
We say that a family of convolution operators on $\ltrn$, $(\Q_s)_s$ is a \emph{CLP family} (Calder\'on-Littlewood-Paley family), if there exist $\sigma>0$ and $\psi\in L^1(\rn)$ satisfying $|\psi(x)|\lesssim (1+|x|)^{-n-\sigma}$  and $|\hat{\psi}(\xi)|\lesssim \min(|\xi|^\sigma,|\xi|^{-\sigma})$, such that the following conditions hold:
\begin{enumerate}[(i)]
	\item For $f\in C_c^{\infty}(\bb R^n)$, we have the representation $\Q_sf= \psi_s* f:= s^{-n}\psi(\cdot/s)*f$.
	\item For each $f\in C_c^{\infty}(\bb R^n)$, we have the bound $$\sup_{s>0}\| \Q_s f\|_\ltrn + \sup_{s>0}\| s\nb \Q_s f\|_\ltrn \lesssim \| f\|_\ltrn.$$
	\item For each $f\in C_c^{\infty}(\bb R^n)$, $\Q_s$ satisfies the square function estimate  $$\| \s(\Q_s f)\|_\ltrn \approx \| \v(\Q_s f)\|_\ltrn \lesssim \| f\|_\ltrn.$$ 
	\item The Calder\'on Reproducing Formula holds; that is, $$\int_0^\infty \Q_s^2\, \frac{ds}{s}=\operatorname{Id}_{\ltrn},$$ where the convergence of the integral is in the strong operator topology on the Banach space of linear bounded operators on $\ltrn$. 
\end{enumerate}
\end{definition}

\begin{definition}[Carleson measure] A non-negative measure $\mu$ on $\reu$ is a \emph{Carleson measure} if $$\| \mu\|_{\mathcal{C}} := \sup_{Q: Q \text{ is a cube in }\bb R^n} \frac{\mu(R_Q)}{|Q|} <\infty.$$
\end{definition}

\begin{lemma}[John-Nirenberg Lemma for Carleson Measures]\label{lem-john-nirenberg-carleson-measures}
Let $\mu$ be a non-negative measure on $\reu$. Suppose there exist $\eta\in (0,1)$ and $C_0>0$ such that for all cubes $Q\subset \rn$, there exists a disjoint collection  $(Q_j)_{j\in \NN} \subset \DD(Q)$ satisfying $\sum_{j\geq 1} |Q_j|<\eta |Q|$  and   $\mu(R_Q\setminus (\cup_j R_{Q_j})) \leq C_0|Q|$. Then $\mu$ is a Carleson measure.
\end{lemma}

\begin{remark}
We may replace the Lebesgue measure on $\rn$ by any other Radon measure. If we assume that the hypotheses only hold for dyadic cubes, then we require the measure to be doubling.
\end{remark}
 
\begin{lemma}\label{lem-prelim-john-nirenberg-local-square-functions}
Suppose that $F:\reu \to \RR$, $F\geq 0$ and define the local square function $A_{Q,F}:\bb R^n\ra\bb R$ by
\[ A_{Q,F}: =  \Big( \dint_{|x-y|<t<\ell(Q)}|F(y,t)|^2\, \frac{dydt}{t^{n+1}}\Big)^\frac{1}{2}.\]
If there exists $C_0>0$ with the property that for every cube $Q\subset \rn$, the estimate $\int_Q A_{Q,F}^2\, dx  \leq C_0 |Q|$ holds, then for every $p>1$, there exists a constant $C_1$ depending on $p, n$ and $C_0$ such that for every cube $Q$,   $$\int_Q A_{Q,F}^p\, dx \leq C_1 |Q|.$$
\end{lemma}

\noindent\emph{Proof.} When $p\leq2$, by Jensen's inequality, the result is trivially true with $C_1\leq C_0^{p/2}$. Now assume $p>2$. For ease of notation, we will write  $A_Q=A_{Q,F}$. Moreover, for $\alpha>0$ we define $$A_{Q, \alpha}(x):=\big(\int_0^{\ell(Q)} \int_{|x-y|<\alpha t} |F(y,t)|^2\, \frac{dydt}{t^{n+1}}\big)^{1/2}.$$ When $\alpha=1$, we may omit the subscript $\alpha$. We also set $K_p:= \sup_{Q\subset \rn} \fint_QA_Q^p$.  Note first that $K_{\alpha,p}\approx_{\alpha,p} K_{1,p}=: K_p$. We defer the proof of this fact to the end, and proceed with the proof of the lemma. 

Let us momentarily assume that $K_p<\infty$ a priori, and set $\alpha>0$ and $N\gg1$, both to be specified later. Consider the open set $\om_N:= \{ x\in Q: \aqa(x)>N\}$.  By the Chebyshev  inequality, we see that $|\om_N|\lesssim_\alpha C_0N^{-2}|Q|$.  In particular, given $\alpha>0$, we may choose $N\gtrsim_\alpha\sqrt{C_0}$ so that $\om_N\subsetneq Q$.
Observe that
\begin{equation}\nonumber
\int_Q A_Q^p= \int_{\om_N} A_Q^p + \int_{Q\backslash \om_N} A_Q^p =: I+II.
\end{equation}
By definition of $\om_N$, we have that $II\leq N^p |Q\backslash\om_N|$. On the other hand, if $(Q_j)_j$ is a Whitney decomposition of $\om_N$, we can write  (exploiting the convexity of  $s\mapsto s^{p/2}$)
\begin{equation}\nonumber
I\lesssim_p \sum_{j\geq 1} \int_{Q_j} A_{Q_j}(x)^p\, dx + \sum_{j\geq 1}\int_{Q_j} (A_{Q}(x)^2-A_{Q_j}(x)^2)^{p/2}\, dx.
\end{equation}
For the first term, we easily have that $$\sum_{j\geq 1} \int_{Q_j} A_{Q_j}(x)^p\, dx \leq \sum_{j\geq 1} K_p |Q_j| = K_p|\om_N|.$$ For the second term,  by definition of $A_Q$ and $A_{Q_j}$ we see that
\begin{equation}\nonumber
\begin{split}
A_Q(x)^2- A_{Q_j}(x)^2  =  \int_{\ell(Q_j)}^{\ell(Q)}\int_{|x-y|<t}|F(y,t)|^2\, \dfrac{dydt}{t^{n+1}}. 
\end{split}
\end{equation}
If $x\in Q_j$, then  there exists $x_*\in Q\backslash\om_N$ (recall $Q\backslash\om_N\neq \emptyset$) such that $|x-x_*|\approx \ell(Q_j)$ with implicit constants depending only on $n$. In particular, for some $\alpha=\alpha(n)>0$, we have the inclusion
\begin{equation}\nonumber
\big\{ (y,t)\in \reu: \, |x-y|<t, \quad \ell(Q_j)<t<\ell(Q)\big\} \subseteq \big\{ (y,t)\in \reu: \, |x_*-y|<\alpha t, \quad 0<t<\ell(Q)\big\},
\end{equation} 
so that  $A_Q(x)^2- A_{Q_j}(x)^2 \leq \aqa (x_*)^2\leq  N^2$, since $x_*\in Q\backslash \om_N$. Accordingly,
\begin{equation}\nonumber
\sum_{j\geq 1} \int_{Q_j} (A_Q (x)^2-A_{Q_j}(x)^2)^{p/2}\, dx \lesssim_\alpha N^p|\om_N|.
\end{equation}
Combining these previous estimates, we obtain that $I\lesssim_{p,n} K_p|\om_N| + N^p|\om_N|$, and so
\begin{equation}\nonumber
\int_Q A_Q(x)^p\, dx \lesssim_{p,n} K_p|\om_N| + N^p|Q| \leq C_0K_pN^{-2}|Q| + N^p |Q|. 
\end{equation}
Dividing by $|Q|$ and taking   supremum over cubes gives $K_p \lesssim_{p,n} C_0K_pN^{-2} + N^p$.  Choosing $N=M\sqrt C_0$ with $M\geq1$ large enough, we may hide the first term to the left-hand side, and thus obtain $K_p \lesssim_{p,n} C_0^{p/2}$.

Finally, to do away with the restriction  $K_p<\infty$, we fix $\eta>0$ and work with $F_\eta:= F{\bbm 1}_{\eta<|F|<1/\eta}{\bbm 1}_{\eta<t<1/\eta}$, for which $K_p<\infty$, and appeal to the monotone convergence theorem in the limit $\eta\to 0^+$. 

We now turn to the proof of $\kap \approx K_p$. Notice that we only used this in the case $p=2$, so we will only prove this special case. We will also work only with $\alpha>1$. By Fubini's theorem, if $\omega_n:=|B(0,1)|$ is the volume of the unit ball in $\rn$,  
\begin{multline}\nonumber
\int_Q \aqa(x)^2\, dx = \int_0^{\ell(Q)} \int_\rn\int_\rn {\bbm 1}_Q(x){\bbm 1}_{|x-y|<\alpha t} |F(y,t)|^2\,  dxdy  \dfrac{dt}{t^{n+1}}
\\= \alpha^n \omega_n\int_0^{\ell(Q)} \int_\rn \Big(\fint_{|x-y|<\alpha t} {\bbm1}_Q(x)\, dx \Big)  |F(y,t)|^2\, dy\dfrac{dt}{t}.
\end{multline}
We claim that\footnote{We remind the reader that the notation $CQ$ means the concentric dilate of $Q$ by a factor $C>0$.}, for some dimensional constants $c,c'$ and every $\beta>1$,
\begin{equation}\nonumber
c\beta^{-n}{\bbm1}_{Q}(y) \leq \fint_{|x-y|<\beta t}{\bbm1}_Q(x)\, dx \leq{\bbm 1}_{c'\beta Q}, \quad \textup{ whenever } 0<t<\ell(Q).
\end{equation}
This claim follows immediately by noting that $\fint_{|x-y|<\beta t}{\bbm1}_{Q}(x)\, dx = \frac{|Q\cap B(y,\beta t)|}{|B(y,\beta t)|}$. Using the second inequality with $\beta=\alpha$ and the first with $\beta=1$ and $c'\alpha Q$ in place of $Q$, we arrive at 
\begin{multline}\nonumber
\int_Q \aqa(x)^2 \, dx  \lesssim_{\alpha,n} \int_0^{\ell(Q)} \int_{c'\alpha Q} |F(y,t)|^2\, dy \dfrac{dt}{t} 
 = \int_0^{\ell(Q)} \int_\rn{\bbm1}_{c'\alpha Q} (y) |F(y,t)|^2\, dy \dfrac{dt}{t}\\
 \lesssim_{n} \int_0^{\ell(Q)} \int_\rn \Big(\fint_{|x-y|<t}{\bbm1}_{c'\alpha Q} (x)\, dx \Big) |F(y,t)|^2\, dy\dfrac{dt}{t} 
  \lesssim_{n,\alpha} \int_{c'\alpha Q} A_Q(x)^2\, dx
  \leq \int_{c'\alpha Q} A_{c'\alpha Q}(x)^2\, dx 
  \lesssim_{n,\alpha} K_p|Q|. 
\end{multline}
The result now follows from taking the supremum over all cubes.\hfill{$\square$}

\subsection{Weights and Extrapolation}

\begin{definition}[$A_p$ weights]\label{def-ap-weights}
Let $1<p<\infty$. A weight   $\nu\in L^1_{\loc}(\rn)$ is said to be an \emph{$A_p$ weight} if there exists a constant $C\geq1$ such that for every cube $Q\subset \rn$, the estimate $$ \Big( \fint_Q \nu \Big)\Big( \fint_Q\nu^{-p'/p}\Big)^{\frac{p}{p'}}\leq C$$ holds.   The infimum over all these constants is denoted $[\nu]_{A_p}$; we refer to it as the \emph{$A_p$ characteristic} of $\nu$. We say that $\nu\in A_1$ if $(\m\nu)(x)\leq C \nu(x)$ for a.e. $x\in \rn$. The infimum over such $C$ is denoted by $[\nu]_{A_1}$.
\end{definition}

Closely related to $A_p$ weights are the reverse H\"older classes.

\begin{definition}[Reverse H\"older class] Let $1<s<\infty$. A weight  $\nu$ is said to satisfy a \emph{reverse H\"older inequality with exponent $s$}, written $\nu\in RH_s$, if there exists  $C\geq1$ such that for every cube $Q\subset \rn$,    $$\Big(\fint_Q \nu^s\Big)^{1/s}\leq C \fint_Q \nu.$$
\end{definition}

Let us summarize most of the  basic  facts about $A_p$ weights which we will need.

\begin{proposition}[{\cite[Theorem 1.14, Lemma 2.2, Lemma 2.5, Theorem 2.6]{gcrdf85}}]\label{prop-properties-ap-weights}
	Let $1\leq p<q<\infty$. The following statements hold. 
	
	\begin{enumerate}[(i)]
		\item \emph{ (\cite[Ch. IV Theorem 1.14 (a)]{gcrdf85})} $A_p\subset A_q$.
		\item A weight $\nu$ belongs to $A_2$ if and only if $\nu^{-1}\in A_2$.
		\item \emph{ (\cite[Ch. IV Theorem 1.14 (b)]{gcrdf85})} If $\nu\in A_p$ then $\nu^\delta\in A_p$ for any $0<\delta<1$.
		\item \emph{ (\cite[Ch. IV Lemma 2.2]{gcrdf85})} If $\nu\in A_p$ then $\nu dx$ is a doubling measure, and the doubling constant depends on $\nu$ only through $[\nu]_{A_p}$ (and $p$).
		\item \emph{ (\cite[Ch. IV Lemma 2.5]{gcrdf85})} If $\nu\in A_p$ then $\nu\in RH_s$ for some $s$ that depends on the weight only through $[\nu]_{A_p}$ (and $p$).
		\item \emph{ (\cite[Ch. IV Theorem 2.6]{gcrdf85})} If $\nu\in A_q$ then $\nu\in A_{q-\varepsilon}$ for some $\varepsilon$ depending on $\nu$ only through $[\nu]_{A_q}$ (and $q$). 
		\item If $\nu\in A_q$ and $s>1$, then $\nu\in RH_s$ if and only if $\nu^s\in A_{s(q-1)+1}$.
		\item \emph{ (Coifman-Rochberg \cite[Proposition 2]{cr80}, \cite[Ch. II Theorem 3.4]{gcrdf85})} If $f:\rn\to \CC$ is such that $(\m f)(x)<\infty$ for a.e. $x\in\bb R^n$, then for every $0<\delta<1$ we have that $\nu_\delta:=(\m f)^\delta\in A_1$ and moreover $[\nu_\delta]_{A_1}\leq C_\delta$ depends only on $\delta$.
		\item \emph{ (Muckenhoupt's Theorem \cite[Theorem 2]{m72}, \cite[Ch. IV Theorem 2.8]{gcrdf85})} For any $1<p<\infty$, $\nu\in A_p$ and $f\in L^p(\nu)$, $\| \m f \|_{L^p(\nu)}\lesssim_{[\nu]_{A_p}} \| f\|_{L^p(\nu)}$.
		\item \emph{ (Coifman-Fefferman \cite[Theorem III]{cf74})} Let $T$ be a ``regular'' singular integral, as defined in \emph{\cite{cf74}}, and $T_*$ the associated maximal operator. Then, for every $\nu\in A_\infty$ and $f\in C_c^\infty(\rn)$, we have that $\| T_*f\|_{L^p(\nu)}\lesssim_{p,n} \| \m f \|_{L^p(\nu)}$.
		In particular, by Muckenhoupt's Theorem above, we have that $\| T_* f\|_{L^p(\nu)}\lesssim_{[\nu]_{A_p}} \| f\|_{L^p(\nu)}$.
	\end{enumerate} 
\end{proposition}

The following result was originally proved by Rubio de Francia in \cite{rdf83,rdf84}. We refer to \cite[Theorem 1.1]{cump} for a simple proof of this fact.

\begin{theorem}\label{thm-extrapolation-ap-weights} Let $1<p_0<\infty$ and let $T$ be an operator satisfying $\| Tf\|_{L^{p_0}(\nu)} \lesssim_{[\nu]_{A_{p_0}}} \| f\|_{L^{p_0}(\nu)}$, for all $\nu\in A_{p_0}$ and all $f\in L^{p_0}(\nu)$. Then, for every $p\in(1,\infty)$,  $\nu\in A_p$, and   $f\in L^p(\nu)$, we have $\| Tf\|_{L^p(\nu)} \lesssim_{[\nu]_{A_p}} \| f\|_{L^p(\nu)}$.
\end{theorem}

It is important for applications to note that the above theorem does not require any special structure on $T$; it does not need to be linear or sublinear. In fact, we have

\begin{theorem}[{\cite[Theorem 3.9]{cump}}]\label{thm-extrapolation-version-2}
Fix $p_0 \in (1,\infty)$ and $\mathcal{F}$ a collection of pairs of non-negative measurable functions $(f,g)$. Suppose that    $\| f\|_{L^{p_0}(\nu)} \lesssim_{[\nu]_{A_{p_0}}} \| g\|_{L^{p_0}(\nu)}$ for all $\nu\in A_{p_0}$ and all $(f,g)\in \mathcal{F}$. Then for every $p\in(1,\infty)$, $\nu\in A_p$, and   $(f,g)\in\mathcal{F}$, we have $\| f\|_{L^p(\nu)} \lesssim_{[\nu]_{A_p}} \| g\|_{L^p(\nu)}$.
\end{theorem}

In practice, the collection $\mathcal{F}$ often takes the form $(|S_1 h|, |S_2 h|)$ for some operators $S_i$ and $h$ in some nice class of functions. A corollary of the previous theorem and this observation is the following.

\begin{corollary}[{\cite[Corollary 3.14]{cump}}]\label{cor-weighted-l2-implies-lp-one-sided-restriction}
Let $r\in(1,2)$, and suppose that $T$ is an operator satisfying $\| Tf\|_\ltnu \lesssim_{[\nu]_{A_{2/r}}} \| f\|_\ltnu$, for  each $f\in C_c^\infty(\rn)$ and all $\nu\in A_{2/r}$. Then $\| Tf\|_{L^q(\mathbb{R}^n)}\lesssim_q \| f\|_{L^q(\mathbb{R}^n)}$  for all $q>r$.
\end{corollary}
To prove the corollary, one defines $S_1 f := |Tf|^r$, $S_2f := |f|^r$. Then, by hypothesis, $\| S_1 f\|_{L^{2/r}(\nu)} \lesssim_{[\nu]_{A_{2/r}}} \| S_2 f\|_{L^{2/r}(\nu)}$, and hence by the previous theorem, $\| S_1 f\|_{L^p(\nu)} \lesssim_{[\nu]_{A_p}} \| S_2 f\|_{L^p(\nu)}$ for $p\in(1,\infty)$. Setting $\nu\equiv 1$ and $p = q/r$ gives the desired result.

\begin{theorem}\label{thm-weighted-lp}
Let $(\Q_s)_s$ be a CLP family (see Definition \ref{def-clp-family}) and let $\nu\in A_2$. It holds that
\begin{equation}\nonumber
\int_\rn \int_0^\infty |(\Q_t f)(x)|^2\, \dfrac{dt}{t}\nu(x)\, dx \lesssim_{n,[\nu]_{A_2}} \int_\rn |f(x)|^2\nu(x)\, dx.
\end{equation}
\end{theorem}

\begin{remark}
By Theorem \ref{thm-extrapolation-ap-weights}, we obtain that the vertical square function associated to $(\Q_s)_s$ is bounded on $L^p(\nu)$ for every $\nu\in A_p$ and $1<p<\infty$; that is, $\| \v(\Q_s f)\|_{L^p(\nu)} \lesssim \| f\|_{L^p(\nu)}$ for every $\nu\in A_p$.
\end{remark}

\noindent\emph{Proof of Theorem \ref{thm-weighted-lp}}. The idea is to use the method in \cite[Theorem B]{drdf86}, to interpolate a ``good'' bound with a plain uniform bound in order to obtain another ``good'' bound in between. We will combine this with interpolation with change of measures as in \cite[Theorem 2.11]{sw58}, exploiting the self-improvement property of $A_p$ weights. Since this idea will be used quite often throughout the paper we write out this portion of the the argument in full here, and refer back to it when applicable.

We first claim that it is enough to prove the following estimate:
\begin{equation}\label{eq-quasi-orth-weighted-lp}
\int_{\rn} |\Q_s \widetilde{\Q}_t^2 f|^2\nu \lesssim_{[\nu]_{A_2}} \min\Big( \dfrac{t}{s}, \dfrac{s}{t}\Big)^\alpha \int_\rn |\widetilde{\Q}_tf|^2\nu, \qquad \text{for each }s,t>0, 
\end{equation}
for some $\alpha>0$ and some CLP family $(\widetilde{\Q}_t)_t$. Indeed, once this is shown, the desired result follows from a familiar quasi-orthogonality argument (see for instance the proof of \cite[Theorem 4.6.3]{grafakosmfa}).

To prove \eqref{eq-quasi-orth-weighted-lp}, we   claim that it is enough to prove the following estimates.
\begin{enumerate}[(i)]
	\item (Unweighted quasi-orthogonality) There exists $\beta>0$ such that for any $s,t>0$, we have the estimate $$\int_\rn |\Q_s\widetilde{\Q}_t^2 f|^2 \leq C_1\Big(\frac{s}{t},\frac{t}{s}\Big)^\beta \int_\rn |\widetilde{\Q}_tf|^2.$$
 	\item (Uniform weighted estimate) For any $s>0$ and $\nu\in A_2$, we have the estimate $$\int_\rn |\Q_s\widetilde{\Q}_t^2f|^2\nu \leq C_2([\nu]_{A_2}) \int_\rn |\widetilde{\Q}_tf|^2\nu.$$
\end{enumerate}

Assume that these hold for the moment and fix $\nu\in A_2$. By properties of $A_2$ weights, there exist  $\delta,C>0$ such that $\nu^{1+\delta}\in A_2$ with $[\nu^{1+\delta}]_{A_2}\leq C$. In particular, the uniform weighted estimate holds with $\nu^{1+\delta}$ in place of $\nu$, with the implicit constants depending only on $[\nu]_{A_2}$. Therefore, if we define the measures $d\mu_\tau:=\nu^{(1+\delta)\tau}\, dx $, interpolation with change of measure (see \cite[Theorem 2.11]{sw58}\footnote{Strictly speaking, the statement of \cite[Theorem 2.11]{sw58} explicitly excludes the case under consideration (indeed the proof given does not apply in this case); however as is mentioned immediately after the statement of said Theorem, we may run an argument similar to the standard proof of the Riesz-Thorin Theorem, employing instead the three line lemma for sub-harmonic functions as in \cite{cz56}.}) gives
\begin{equation}\nonumber
\int_\rn |\Q_s\widetilde{\Q}_t^2 f|^2\, d\mu_\tau  \leq C_1^{1-\tau}C_2^{\tau} \Big(\dfrac{s}{t},\dfrac{t}{s}\Big)^{\beta (1-\tau)} \int_\rn |\widetilde{\Q}_tf|^2 d\mu_\tau.
\end{equation}
The desired estimate \eqref{eq-quasi-orth-weighted-lp} is exactly the case $\tau=1/(1+\delta)$ with $\alpha=\beta \delta/(1+\delta)$. This completes the proof, modulo the above pair of estimates.

The first estimate, unweighted quasi-orthogonality, is a consequence of classical Littlewood-Paley theory. On the other hand, the weighted estimate follows from both the fact that $|\Q_s f|, |\widetilde{\Q}_tf|\lesssim \m f$ pointwise in $\rn$ and Muckenhoupt's theorem on the $L^2(\nu)$ boundedness of $\m$ for $\nu\in A_2$ (see Proposition \ref{prop-properties-ap-weights}).  \hfill{$\square$}

\begin{lemma}[$L^p$ inequalities from weighted $L^2$ bounds]\label{lem-restricted-extrapolation}
Suppose that $T:\ltrn\to\ltrn$ is a bounded (not necessarily linear) operator; that is, $\| Tf\|_\ltrn \lesssim \| f\|_\ltrn$.
\begin{enumerate}[(i)]
	\item Suppose that there exists $M>1$ such that for all $\nu\in A_1$ with the property that $\nu^M\in A_1$ it holds that $\| Tf\|_\ltnu \lesssim_{[\nu^M]_{A_1}} \| f\|_\ltnu$, for every $f\in C_c^\infty(\rn)$. Then for every $p\in (2,2+1/M)$, it holds that $\| Tf\|_\lprn \lesssim_p \| f\|_\lprn$.
	\item Suppose that there exists $M>1$ such that for all $\nu$ with the property that $\nu^{-M}\in A_1$ it holds that $\| Tf\|_\ltnu \lesssim_{[\nu^{-M}]_{A_1}} \| f\|_\ltnu$	for every $f\in C_c^\infty(\rn)$. Then for every $p\in (2-1/M,2)$, we have that $\| Tf\|_\lprn \lesssim_p \| f\|_\lprn$.
	\item Suppose that there exists $M>1$ such that for all $\nu\in A_2$ with the property that $\nu^M\in A_2$, it holds that  $\| Tf\|_{\ltnu} \lesssim_{[\nu^M]_{A_2}}\| f\|_\ltnu$ 	for every $f\in C_c^\infty(\rn)$. Then for every $p\in(2-1/M,2+1/M)$, we have $\| Tf\|_\lprn \lesssim_p \| f\|_\lprn$.
\end{enumerate}
\end{lemma}

\noindent\emph{Proof.} This lemma and its proof are contained in \cite[Corollary 3.37]{cump} for the much more general setting of restricted extrapolation of $A_p$ weights. However, since we will later on need to modify the arguments used in the proof a little to fit our needs, it seems appropriate to write the proof down for future reference. The key fact that we will use is the Coifman-Rochberg theorem (see Proposition \ref{prop-properties-ap-weights}).
 
We start with (i). Fix $p>2$ with $M<1/(p-2)$ and $f\in C_c^\infty(\rn)$. Note that $\nu:=(\m(|Tf|))^{p-2}\in A_1$, and 
\begin{multline}\nonumber
\int_\rn |Tf|^p  \leq \int_\rn |Tf|^2 (\m (|Tf|))^{p-2}  \lesssim_{[\nu^M]_{A_1}} \int_\rn |f|^2(\m(|Tf|))^{p-2} \\
  \leq \Big( \int_\rn |f|^p \Big)^{2/p} \Big( \int_\rn (\m(|Tf|))^p\Big)^{(p-2)/p}  \lesssim_p \Big( \int_\rn |f|^p\Big)^{2/p}\Big(\int_\rn  |Tf|^{p} \Big)^{1-\frac2p}.
\end{multline}
If we first assume that $\| Tf\|_\lprn<\infty$, then the result follows. To get rid of this assumption, we instead consider the sequence of operators  $S_kf(x):=(Tf)(x){\bbm 1}_{|Tf|\leq k}(x)$ on $L^2(\bb R^n)$. Then $\{S_k\}_k$ is uniformly bounded on $\ltrn$, and they satisfy the same hypotheses as $T$ with constants independent of $k$. Then, for $f\in C_c^\infty(\bb R^n)$,  we have that $S_k f\in \ltrn \cap L^\infty(\rn)$, and so by our argument above, $\| S_k f\|_\lprn \lesssim_p \| f\|_\lprn$. We now let $k\to \infty$ and use the Monotone Convergence Theorem.  

We   turn to (ii). Fix $p<2$ with $\frac1{2-p}<M$ and $f\in C_c^\infty(\rn)$ not identically $0$. Note that $\nu:=(\m(|Tf|+|f|))^{p-2}$ satisfies $\nu^{-1}\in A_1\subset A_2$, and hence $\nu\in A_2$. We estimate
\begin{multline}\nonumber
\int_\rn |Tf|^p  \leq \int_\rn (\m(|Tf|+|f|))^p = \int_\rn (\m(|Tf|+|f|))^2(\m(|Tf|+|f|))^{p-2} \\
  \lesssim_{[\nu]_{A_2}} \int_\rn (|Tf|^2+|f|^2)\nu \lesssim_{[\nu^{-M}]_{A_1}} \int_\rn |f|^2 \nu \leq \int_\rn |f|^2(\m f)^{p-2}  \leq \int_\rn (\m f)^p ,
\end{multline}  
where we have used Muckenhoupt's theorem, yielding the desired result.

The third statement follows from the first two and Jones's factorization theorem of $A_2$ weights (see \cite{JonesFactor}) as quotients of $A_1$ weights. \hfill{$\square$}

Sometimes we will not be able to conclude boundedness on all weights $\nu^M\in A_2$, but rather only on weights whose characteristic is uniformly bounded. An inspection of the proof of the above lemma, together with  Proposition \ref{prop-properties-ap-weights}, reveals that this is enough to conclude the unweighted $L^p$ estimates. 

\begin{corollary}\label{cor-unweighted-lp-extrapolation}
Let $M\geq 1$, $0<\delta<1$ and $T$ be an operator satisfying, for every $\nu\in A_2$ with $[\nu^M]\leq C_\delta$ (where $C_\delta$ is as in Proposition \ref{prop-properties-ap-weights}), that $\| Tf\|_\ltnu \lesssim_{[\nu^M]_{A_2}} \| f\|_\ltnu$.  Then, for every $p\in (2-\delta/M, 2+\delta/M)$, $\| Tf \|_\lprn \lesssim_p \| f\|_\lprn$. Analogous statements for the one-sided versions of the estimates also hold.
\end{corollary}

\begin{lemma}[Weighted Carleson's Lemma]\label{lem-prelim-weighted-carleson}
Suppose that $\mu$ is a measure in $\reu$ and that $\nu \in L^{1}_{\loc}(\rn)$ is a doubling weight. Assume further that for every cube $Q\subset \rn$,  it holds that $\mu(R_Q)\lesssim \nu (Q)$. Then, for every measurable function $F:\reu \to \CC$ and every $p>0$, we have that $\displaystyle\dint_{\reu} |F|^p\, d\mu \lesssim_{n,\operatorname{doub}} \int_{\rn} (\mathcal N F)^p\, \nu$.
\end{lemma}

The proof   is exactly the same as the usual one  when $\nu\equiv 1$, and thus omitted.  Next, we will need  a version of  Carleson's Lemma that uses the modified non-tangential maximal function $\mntm$ in place of $\mathcal{N}$; its proof is straightforward and thus omitted.

\begin{lemma}\label{lem-prelim-modified-carleson}
Let $d\mu(x,t)=m(x,t)\, dxdt$ be a non-negative measure on $\reu$ and $\nu$ is a doubling weight. For every $(x,t)\in \reu$,   suppose that $d\widetilde{\mu}(x,t):=  ( \sup_{(y,s)\in \cc_{x,t}} m(y,s) ) \, dxdt$ satisfies   $\widetilde{\mu}(R_Q)\leq C_0 \nu(Q)$ for every cube $Q\subset \rn$. Then, for every $q>0$,  $\displaystyle\dint_{\reu} |F|^q\, d\mu\lesssim_{\operatorname{doub}} C_0 \int_\rn (\mntmq F)^q \, \nu$.
\end{lemma}

\begin{definition}[$A_{p,q}$ classes]\label{def-apq-classes}
Let $1<p\leq q<\infty$. We say that a weight $\nu\in A_{p,q}=A_{p,q}(\rn)$ if there exists a constant $C>0$ such that for every cube $Q\subset \rn$,  
\[
\Big( \fint_Q \nu^q\, dx \Big)^{1/q}\Big( \fint_Q \nu^{-1/p'}\, dx\Big)^{1/p'}\leq C.
\]
The infimum over all such $C$ is written $[\nu]_{A_{p,q}}$.
\end{definition}

\begin{theorem}[{\cite[Theorem 4]{mw74}}]\label{thm-mapping-properties-i1}
Let $1<p<n$ and set $1/q:=1/p-1/n$. Then $\nu\in A_{p,q}$ if and only if  $\| I_1 f\|_{L^q(\nu^q)} \lesssim_{[\nu]_{A_{p,q}}} \| f\|_{L^p(\nu^p)}$.
\end{theorem}

Throughout, there will be instances where multiplication by an $\lnrn$ function is acting as, or rather in place of, a (spatial) gradient. The following proposition should be interpreted as stating that, at least in $L^p$ spaces, the two operations are not far from each other. We remind the reader that we assume $n\geq 3$.

\begin{proposition}\label{prop-i1b-bounded} Let $B\in \lnrn$ and $f\in C_c^\infty(\rn)$. Then, for every $\nu\in A_2$, we have $$\| I_1(B\cdot f)\|_\ltnu \lesssim_{[\nu]_{A_2}}\| B\|_\lnrn\| f\|_\ltnu.$$ In particular, for every $1<p<\infty$, it holds that $\| I_1(B\cdot f)\|_\lprn \lesssim_p \| f\|_\lprn$, where the implicit constants depend on $\| B\|_\lnrn$, $p$, and $n$. If in addition we have that $\nu^{2^*/2}\in A_2$ with $2^*=2^*_n$,  then $$\| B\cdot  I_1 f\|_\ltnu \lesssim_{[\nu^{2^*}]_{A_2}} \| B\|_\lnrn \| f\|_\ltnu.$$ Accordingly, $\| B\cdot I_1 f\|_\lprn \lesssim_p \| f\|_\lprn$, for $1+2/n <p<3-2/n$. 
\end{proposition}

\noindent\emph{Proof.} Let $\nu\in A_2$ and set $\omega:= \nu^{1/2}$. We claim that $\omega\in A_{2_*,2}$. Assuming the claim, we have 
\begin{equation}\nonumber
\begin{split}
\| I_1(B\cdot f)\|_{\ltnu}= \| I_1(B\cdot f)\|_{L^2(\omega^2)} \lesssim_{[\omega]_{A_{2_*,2}}} \| B\cdot f\|_{L^{2_*}(\omega^{2_*})} \leq \| B\|_\lnrn \| f\|_{L^2(\omega^2)},
\end{split}
\end{equation}
where we used H\"older's inequality in the last step. To prove the claim, we use Jensen's inequality to see that $\big(\fint_Q \omega^{-2_*} \big)^{2/2_*} \leq \fint_Q \omega^{-2}$. Using this estimate in the definition of $\omega^2\in A_2$, we deduce that $[\omega]_{A_{2_*,2}}\leq [\nu]_{A_2}^{1/2}$. This completes the proof of the first part. The second part follows the same lines, using instead that
\begin{equation}\nonumber
\Big( \fint_Q \omega^{2^*} \Big) \Big( \fint_Q \omega^{-2}\Big)^{2^*/2}\leq \Big( \fint_Q\omega^{2^*}\Big)\Big( \fint_Q \omega^{-2^*}\Big)\leq [\nu^{2^*/2}]_{A_2}, 
\end{equation}
so that $[\omega]_{A_{2,2^*}} \leq [\nu^{2^*/2}]_{A_2}^{1/2^*}$.  The $L^p$ estimate finally follows from restricted extrapolation (see Lemma \ref{lem-restricted-extrapolation}), using the fact that $2/2^*= 1-2/n$.\hfill{$\square$}

\begin{proposition}\label{prop-square-function-bounds-i-pt} Let $P_t$ be an approximate identity with  smooth, even, compactly supported kernel. Then, for every $\nu\in A_2$ and $f\in C_c^\infty(\rn)$, it holds that
\begin{equation}\nonumber
\| \v( t^{-1}(1-P_t) f)\|_\ltnu^2 = \int_0^\infty \int_\rn \Big| \dfrac{I-P_t}{t} f\Big|^2\, \dfrac{\nu(x)\,dxdt}{t} \lesssim_{[\nu]_{A_2}} \| \nabla_{\|}f\|_\ltnu^2.
\end{equation}
\end{proposition}

\noindent\emph{Proof.} Recall that $I_1$ denotes the fractional integral of order 1; hence $\nabla_{\|}I_1 = I_1\nabla_{\|}= R$, where $R$ is the vector-valued Riesz-transform (with symbol $\xi/|\xi|$). In particular, $\| R f\|_\ltnu \approx \| f\|_\ltnu$ for all $\nu\in A_2$, allowing us to reduce matters to the estimate
\begin{equation}\nonumber
\int_0^\infty \int_\rn \Big| I_1\dfrac{I-P_t}{t} f\Big|^2\, \dfrac{\nu(x)dt}{t} \lesssim_{[\nu]_{A_2}} \|  f\|_\ltnu,\qquad f\in C_c^{\infty}(\bb R^n).
\end{equation}
We now use a quasi-orthogonality argument, with a change of measure interpolation (see the proof of Theorem \ref{thm-weighted-lp}), to reduce matters to the pair of estimates: If we denote   $T_t:= I_1(1-P_t)/t$, then for some CLP family $(\Q_s)_s$ (see Definition \ref{def-clp-family}),
\begin{equation}\label{eq-1-i-pt}
\| T_t \Q_s^2f\|_\ltrn \lesssim \Big( \dfrac{s}{t}, \dfrac{t}{s}\Big)^{\alpha} \| \Q_s f\|_\ltrn,
\end{equation}
for some $\alpha>0$, and 
\begin{equation}\label{eq-2-i-pt}
\| T_t \Q_s^2 f\|_\ltnu\lesssim_{[\nu]_{A_2}} \| \Q_s f\|_\ltnu.
\end{equation}
Indeed, with \eqref{eq-1-i-pt} and \eqref{eq-2-i-pt} in hand, we may follow the proof of Theorem \ref{thm-weighted-lp}.

For \eqref{eq-1-i-pt}, we compute, via the Fourier transform and Plancherel's theorem, and using $\varphi_t$ and $\psi_s$ for the kernels of $P_t$ and $\Q_s$ respectively, 
\begin{equation}\nonumber
\| T_t \Q_s h\|_\ltrn^2 = c_n\int_\rn \Big| |\xi|^{-1}\dfrac{1-\hat{\varphi}(t|\xi|)}{t} \hat{\psi}(s|\xi|)\hat h(\xi)\Big|^2\, d\xi, 
\end{equation}
where as usual we have abused notation and written  $\varphi$, $\psi$ for the one-dimensional functions representing them. Consider first the case $t<s$,
\begin{equation}\nonumber
\int_\rn\Big| |\xi|^{-1}\dfrac{1-\hat{\varphi}(t|\xi|)}{t} \hat{\psi}(s|\xi|)\hat h(\xi)\Big|^2d\xi  =  \Big(\dfrac{t}{s}\Big)^2 \int_\rn \dfrac{|1-\hat{\varphi}(t|\xi|)|^2}{|t\xi|^4} |s\xi|^2|\hat{\psi}(s|\xi|)|^2 |\hat h(\xi)|^2 \, d\xi\lesssim \Big(\dfrac{t}{s}\Big)^2 \| h\|_\ltrn^2,
\end{equation}
where we used the properties of the CLP family and the fact that $|1-\hat{\varphi}(\tau)|\lesssim  \tau^2$  for $\tau$ near $0$, since $\varphi$ is even.
For the case $s<t$, we use instead the Fundamental Theorem of Calculus to obtain 
\begin{equation}\nonumber
\int_\rn \Big| |\xi|^{-1}\dfrac{1-\hat{\varphi}(t|\xi|)}{t} \hat{\psi}(s|\xi|)\hat h(\xi)\Big|^2 d\xi  = \Big( \dfrac{s}{t}\Big)^2 \int_\rn \Big|\int_0^{t|\xi|} \hat{\varphi}'(\tau)\, d\tau\Big|^2 \dfrac{|\hat{\psi}(s|\xi|)|^2}{|s\xi|^2} |\hat h(\xi)|^2\, d\xi\lesssim \Big( \dfrac{s}{t} \Big)^2 \| h\|_\ltrn^2,
\end{equation}
where we used that $\hat{\varphi} \in L^1(0,\infty)$ and $\hat{\psi}(\tau)/\tau\in L^\infty(0,\infty)$. Using now that $h=\Q_s f$ gives \eqref{eq-1-i-pt}.

The weighted estimate \eqref{eq-2-i-pt} follows from the pointwise inequality
\begin{equation}\nonumber
|T_t f(x)|=|t^{-1}(1-P_t)I_1 f(x)|\lesssim \m(R f)(x),
\end{equation}
where $R= I_1\nbp$ is as before. We sketch the argument: Write  $1-P_t= (1-E_t)+(E_t-P_t)$, where $E_t$ is the dyadic averaging operator; that is, $E_t f(x)=\fint_{Q_{x,t}} f$, where $Q_{x,t}$ is the unique dyadic cube $Q_{x,t}\in \DD_t$ containing $x$. Writing $g=I_1f$, we have that 
\begin{multline}\nonumber
|(E_t-P_t)g(x)| \approx  \Big| \fint_{Q_{x,t}} \fint_{|x-y|<Ct} \varphi(\tfrac{x-y}t)(g(y)-g(z))\, dydz \Big|\\
 \lesssim \fint_{B(x,Ct)}\fint_{B(x,Ct)} |g(y)-g(z)|\, dy dz  \lesssim t\fint_{B(x,Ct)} |\nabla_{\|} g(y)|\, dy   \leq t\m(\nabla_{\|} g)(x),
\end{multline}
where we used Poincare's inequality in the second to last step. Since $I_1\nabla_{\|} f= Rf$, we have the right bound for this term. To handle the  term $1-E_t$, we telescope 
\begin{equation}\nonumber
(1-E_t)g(x)= \sum_{j=0}^\infty (E_{2^{-j-1}t}-E_{2^{-j}t})g(x)=: \sum_{j=0}^\infty (E_{t_{j+1}}-E_{t_j}) g(x),
\end{equation}
and we compute that
\begin{equation}\nonumber
|(E_{t_{j+1}}-E_{t_j}) g(x)|  =\Big| \fint_{Q_{x,t_{j+1}}} (E_{t_j} g(x)-g(y))\, dy\Big| \lesssim \fint_{Q_{x,t_j}} |E_{t_j} g(x) - g(y)|\, dy 
  \lesssim  t_j \fint_{Q_{x,t_j}}|\nabla_{\|} g|\leq  2^{-j}t \m(\nabla_{\|} g)(x).
\end{equation}
The result now follows by summing over $j$.\hfill{$\square$}

We will  need for the following properties of the heat semigroup associated to the Laplacian $\Delta$ in $\rn$.

\begin{proposition}\label{prop-prelim-properties-heat-semigroup}
Let $P_t:=e^{-t^2\Delta}$ and $\Q_t:=t\partial_t P_t$. We define the measure 
\begin{equation}\nonumber
d\mu(x,t):= \dfrac{|\Q_t \nu(x)|^2}{|P_t \nu(x)|^2}P_t\nu(x) \, \dfrac{dxdt}{t}.
\end{equation}
This object satisfies the following properties
\begin{enumerate}[(i)]
	\item For any weight $\nu\in RH_s$ for some $s>1$ it holds that $|P_t \nu(x)|\lesssim \fint_{|x-y|<t} \nu(y)\, dy$, with constants depending on the $RH_s$ and doubling constants of $\nu$.
	\item The measure $d\mu$ satisfies the hypotheses of the modified Carleson's Lemma \ref{lem-prelim-modified-carleson}, provided $\nu\in RH_2$.
\end{enumerate}
\end{proposition}

\noindent\emph{Proof.} The proof of (i) is a simple computation: the kernel of $P_t$ is given by 
\begin{equation}\nonumber
\varphi_t(x-y)= c_nt^{-n}e^{-(|x-y|/2t)^2}, \qquad x,y\in \rn,\, \quad t>0,
\end{equation}
and we can write 
\begin{equation}\nonumber
P_t \nu(x)= \int_{|x-y|<t} \varphi_t(x-y)\nu(y)\, dy + \sum_{j\geq 0} \int_{2^jt\leq |x-y|<2^{j+1}t} \varphi_t(x-y)\nu(y)\, dy.
\end{equation}
Clearly, the first term satisfies the desired estimate; it remains to control the tail. For this, we set $\Delta_j:=\{ y: 2^jt\leq |x-y|<2^{j+1}t\}$ and  employ H\"older's inequality to obtain 
\begin{equation}\nonumber
\int_{\Delta_j} \varphi_t(x-y)f(y)\, dy \leq \Big(\int_{\Delta_j} \varphi_t(x-y)^{s'}\, dy \Big)^{1/s'} \Big( \int_{\Delta_j} \nu(y)^s\, dy \Big)^{1/s}.
\end{equation}
Now we see, using that $\nu\in RH_s$,
\begin{equation}\nonumber
\Big(\int_{\Delta_j}\nu^s \Big)^{1/s} \lesssim (2^{j}t)^{n(\frac1s-1)} \int_{|x-y|<2^{j+1}t} \nu(y) dy \lesssim (C_{doub}2^{n(1/s-1)})^j t^{\frac ns} \fint_{|x-y|<t} \nu(y)dy.
\end{equation}
On the other hand, for $y\in \Delta_j$ we have that $\varphi_t(x-y)\lesssim t^{-n}\exp(-2^{j})$, so that 
\begin{equation}\nonumber
\Big( \int_{\Delta_j}\varphi_t(x-y)^{s'}\, dy \Big)^{1/s'} \lesssim |\Delta_j|^{1/s'} t^{-n}\exp(-2^j) \lesssim t^{-n/s}2^{jn/s'}\exp(-2^j).
\end{equation}
Combining these estimates, (i) follows.

The proof of (ii) is somewhat more involved. We ought to show that
\[
\sup_{\cxt} \Big( \frac{|\Q_s\nu(y)|^2}{|P_s\nu(y)|^2}P_s\nu(y)\frac1s\Big)dxdt\approx   \Big(\sup_{\cxt} \dfrac{|\Q_s\nu(y)|^2}{|P_s\nu(y)|^2}P_s\nu(y)\Big) \dfrac{dxdt}{t}=:d\widetilde\mu(x,t)
\]
is a Carleson Measure. For this purpose, first note that, using (i) and  the doubling property of $\nu$, it is not hard to show that $P_s\nu(y) \approx P_t\nu(x)$ for all $(y,s)\in \cxt$.   In other words,
\begin{equation}\nonumber
d\tmu(x,y)\approx \Big(\sup_\cxt |\Q_s\nu(y)|^2\Big) \dfrac{1}{P_t\nu(x)} \dfrac{dxdt}{t}.
\end{equation}
Now we let $a>0$ be small enough so that $s^2-a^2t^2\approx t^2$ whenever $|t-s|<t/8$ and write 
\begin{equation}\nonumber
\Q_s\nu(y)  = s\pd_s (e^{(a^2t^2-s^2)\Delta}e^{-a^2t^2\Delta}\nu(y))  = -2s^2\Delta e^{(a^2t^2-s^2)\Delta}e^{-a^2t^2\Delta}\nu(y)  = -\dfrac{s^2}{a^2t^2} e^{-(s^2-a^2t^2)\Delta} \Q_{at}\nu(y).
\end{equation}
Therefore,  there exists a universal constant $c>0$ such that $|\Q_s\nu(y)|\lesssim P_{ct}|\Q_{at}\nu|(y)$, for all $|s-t|<t/8$. Setting $g_t(z):= |\Q_{at}\nu(z)|$, we see that
\begin{equation}\nonumber
P_{ct} g_t(y) = c_n \int_\rn (ct)^{-n}e^{-\frac{|y-z|^2}{4(ct)^2}} g_t(z)\, dz 
 = \int_{|x-z|<t} (ct)^{-n}e^{-\frac{|y-z|^2}{4(ct)^2}} g_t(z)\, dz + \sum_{j\geq 0} \int_{\Delta_j} (ct)^{-n}e^{-\frac{|y-z|^2}{4(ct)^2}} g_t(z)\, dz  =: I+II.
\end{equation}
For $I$, we simply note that  $I\lesssim \fint_{|x-z|<t} g_t(z)\, dz \lesssim P_t g_t(x)$. For the tails, we use that $|y-z|\geq\frac78|x-z|$ for any $z\in\Delta_j$, to obtain the bound $II\lesssim P_{c't}g_t(x)$. We conclude that $|\Q_s\nu(y)|\lesssim P_{ct}|\Q_{at}\nu|(x)$ for $(y,s)\in \cxt$.

We have thus reduced matters to proving a (weighted) Carleson Measure estimate for $d\mu'(x,t):= \frac{(P_{ct}|\Q_{at}\nu|(x))^2}{P_t\nu(x)} \, \frac{dxdt}{t}$; that is, we want to show that $\mu'(R_Q)\lesssim \nu(Q)$, for all $Q\subset \rn$. So fix a cube $Q\subset \rn$. We run a stopping time argument to obtain a collection of maximal (dyadic) subcubes $(Q_j)_{j\geq 1}$ of $Q$ with respect to the properties $$\text{either}\quad\text{(a)}\quad \fint_{Q_j} \nu  \geq A \fint_{Q} \nu,\qquad\text{or}\quad\text{(b)}\quad \fint_{Q_j} \nu  \leq A^{-1} \fint_{Q} \nu,$$  for some $A>1$   large. We call $\mathcal{F}_1$ the collection of $Q_j$ satisfying the   property (a), and $\mathcal{F}_2$ the collection of $Q_j$ satisfying (b).

Note that, by construction, if $Q_j\in \mathcal{F}_1$ we have  $|Q_j|\leq \frac{|Q|}{A\nu(Q)} \int_{Q_j} \nu$, so, after summing over $j$, $\sum_{Q_j\in \mathcal{F}_1} |Q_j|\leq \frac{|Q|}{A}$. By the $A_\infty$ property of $\nu$, if $A$ is large enough, we may write $\sum_{Q_j\in \mathcal{F}_1} \nu(Q_j) \leq \frac{\nu(Q)}{4}$. On the other hand, if $Q_j\in \mathcal{F}_2$ we obtain directly  that $\nu(Q_j)\leq A^{-1}|Q_j|\fint_Q \nu$, so that $\sum_{Q_j\in \mathcal{F}_2} \nu(Q_j) \leq \frac{\nu(Q)}{4}$,  if we choose $A>1/4$.

By Lemma \ref{lem-john-nirenberg-carleson-measures} it is enough obtain   $\mu'(R_Q\backslash (\cup_{\mathcal F_1\cup\mathcal F_2} R_{Q_j}))\lesssim\nu(Q)$.  Moreover, notice that for $(x,t)\in R_Q\backslash(\cup_{\mathcal F_1\cup\mathcal F_2} R_{Q_j})$, we have  that $\fint_Q \nu \lesssim P_t\nu(x)$, by construction of the $Q_j$ . Accordingly, it is enough to show 
\begin{equation}\nonumber
\dint_{R_Q} (P_{ct}|\Q_{at}\nu|)(x)^2\, \dfrac{dxdt}{t}\lesssim \dfrac{\nu(Q)^2}{|Q|}.
\end{equation}
To do this, we use Minkowski's inequality to write
\begin{equation}\nonumber
\Big( \dint_{R_Q}( P_{ct}|\Q_{at}\nu|)(x)^2 \dfrac{dxdt}{t}\Big)^{\frac12}\leq \sum_{j=0}^\infty \Big( \dint_{R_Q} (P_{ct}|\Q_{at}({\bbm 1}_{R_j(Q)}\nu)|)(x)^2 \dfrac{dxdt}{t}\Big)^{\frac12} :=\sum_{j=0}^\infty T_j,
\end{equation}
where we denote $R_0=2Q$ and $R_j= 2^{j+1}Q\backslash2^jQ$ for $j\geq 1$. For the first term $T_0$, we employ the fact that $P_t$ is uniformly $\ltrn-$bounded and that $\Q_t$ satisfies an $L^2$ square function estimate to obtain 
\begin{equation}\nonumber
T_0^2\leq \int_0^\infty \int_\rn (P_{ct}|\Q_{at}({\bbm 1}_{R_0}\nu)|)(x)^2\, \dfrac{dxdt}{t} \lesssim \int_0^\infty\int_\rn |\Q_{at}({\bbm 1}_{R_0}\nu)(x)|^2\,\dfrac{dxdt}{t} \lesssim \int_{2Q} \nu^2.
\end{equation}
We now use the reverse H\"older property   of $\nu$ to see that $\int_{2Q}\nu^2 \lesssim |Q|^{-1}\nu(Q)^2$, which gives the desired estimate for $T_0$. For the others, we use the kernel representations; first recall that if $\varphi_t$ is the kernel for $P_t$ and $\psi_t$ the one for $\Q_t$ then  $|\varphi_t(z)|, |\psi_t(z)|\leq c_1t^{-n}e^{-c_2|z|^2/t^2}$. Calling $\nu_j={\bbm 1}_{R_j}\nu$, we compute
\begin{multline}\nonumber
\int_Q(P_{ct}|\Q_{at}\nu_j|)(x)^2\, dx   = \int_Q\Big(\int_\rn \Big|\int_\rn \varphi_{ct}(x-y)\psi_{at}(y-z)\nu_j(z)\, dz\Big| \, dy\Big)^2 \, dx\\
  \lesssim \int_Q \Big( \int_{R_j} \nu(z) \int_\rn t^{-2n}e^{-c\frac{|x-y|^2+|y-z|^2}{t^2}}\, dy\,  dz \Big)^2\, dx. 
\end{multline}
It is easy to verify that $|x-y|^2+|y-z|^2\geq(|x-y|^2+ |x-z|^2)/4$, and hence 
\begin{equation}\nonumber
\int_Q(P_{ct}|\Q_{at}\nu_j|)(x)^2\, dx   \lesssim \int_Q\Big( \int_{R_j} \nu(z)t^{-n}e^{-c\frac{|x-z|^2}{t^2}} \, dz \Big)^2 dx \lesssim e^{-c\frac{(2^j\ell(Q))^2}{2t^2}}\int_Q(\widetilde{P}_t\nu)^2,  
\end{equation}
where we define $\widetilde{P}_t$ the convolution operator with kernel $t^{-n}e^{-c|z|^2/(2t^2)}$. Now we see, from the proof of part (i), that $\widetilde{P}_t\nu(x)\lesssim \fint_{|x-y|<t} \nu(y)\, dy \lesssim \nu(Q)/t^n$. Therefore, 
\begin{equation}\nonumber
\int_Q(P_{ct}|\Q_{at}\nu_j|)(x)^2\, dx \lesssim e^{-c'\frac{(2^j\ell(Q))^2}{t^2}}\frac{\nu(Q)^2|Q|}{t^{2n}}.
\end{equation}
The desired estimate for $T_j$ now follows by integrating in $t$ over $(0,\ell(Q))$.\hfill{$\square$}

\subsection{$L^r-L^q$ Off-diagonal estimates}

Throughout this section we denote by $T_t$, with $t\neq 0$, an operator mapping functions $C_c^\infty(\rn; \CC^{d_1})$ to measurable functions in $\rn$ with values in $\CC^{d_2}$ for some integers $d_1, d_2$.

\begin{definition}[$L^r\to L^q$ Off-diagonal estimates]\label{def-lr-to-lq-off-diagonal-estimates}
Let $1\leq r\leq q\leq \infty$. We say that a family of operators $(T_t)_{t\neq 0}$ satisfies \emph{$L^r-L^q$ off-diagonal estimates} if there exist $C_0>0$ and numbers  $\gamma_1\in \RR$, $\gamma_2>0$ such that for every cube $Q\subset \rn$, the following estimates hold with $\gamma:=\gamma_1+\gamma_2$.
\begin{enumerate}[(i)]
	\item If $|t|\approx \ell(Q)$, then $$\| T_t(f{\bbm 1}_{R_0(Q)})\|_{L^q(Q)} \leq C_0 |Q|^{1/q-1/r} \| f\|_{L^r(Q)}.$$
	\item If $t\in \RR$ and we set $R_j(Q):= 2^{j+1}Q\backslash 2^jQ$ for $j\geq 1$, then
	\begin{equation}\nonumber
	\| T_t(f{\bbm 1}_{R_j(Q)})\|_{L^q(Q)} \leq C_0 2^{-nj\gamma_1} \Big(\dfrac{|t|}{2^j\ell(Q)}\Big)^{nj\gamma_2} |Q|^{1/q-1/r} \| f\|_{L^r(R_j(Q))}.
	\end{equation}
	\item If $t\in \RR$ and $\supp f\subset Q$ then
	\begin{equation}\nonumber
	\| T_t(f)\|_{L^q(R_j(Q))} \leq C_0 2^{-nj\gamma_1}\Big( \dfrac{|t|}{2^j\ell(Q)}\Big)^{nj\gamma_2}|Q|^{1/q-1/r} \| f\|_{L^r(Q)}.
	\end{equation}
\end{enumerate}

\end{definition}

\begin{proposition}[Weighted estimates from off-diagonal decay]\label{prop-averaged-bounds-slices-via-off-diagonal-decay}
Suppose $(T_t)_{t>0}$ is sublinear and satisfies $L^r-L^2$ off diagonal estimates for some $1<r<2$ and $\gamma>1/r$. Then for all $\nu\in A_{2/r}$, and every $t>0$,
\begin{equation}\nonumber
\Big\| \Big(\fint_{|x-y|<t} |T_t f(y)|^2\, dy \Big)^{1/2}\Big\|_{L^2(\nu dx)} \lesssim_{[\nu]_{A_{2/r}}} \| f\|_\ltnu.
\end{equation} 
\end{proposition}

\noindent\emph{Proof. } This proposition is contained within \cite{gdlhh17}, but we provide the proof for completeness. We first decompose $\rn$ into cubes in the dyadic grid $\DDt$ of sidelength $\approx t$ to obtain  
\begin{multline}\nonumber
\Big\| \Big(\fint_{|x-y|<t} |T_t f(y)|^2\, dy\Big)^{1/2} \Big\|_{L^2(\nu dx)}  = \Big(\sum_{Q\in \DDt} \int_Q \fint_{|x-y|<t} |T_t f(y)|^2\, dy \, \nu(x)dx \Big)^{1/2}\\
  \lesssim  \Big(\sum_{Q\in \DDt} \fint_{Q^*} \int_Q |T_t f(y)|^2\,dy\,  \nu(x) dx\Big)^{1/2}
  \lesssim \sum_{j\geq 0} \Big(\sum_{Q\in \DDt}\fint_{Q^*} \int_Q |T_t({\bbm 1}_{R_j(Q)} f(y))|^2\, dy \, \nu(x)dx \Big)^{1/2}   =:I, 
\end{multline}
where as usual we define $R_0(Q):= 2Q$ and $R_j(Q):= 2^{j+1}Q\backslash 2^jQ$ for $j\geq 1$, and we used Minkowski's inequality in the last line. We now exploit the off-diagonal decay of $T_t$ to get,
\begin{equation}\nonumber
\begin{split}
\Big(\int_Q |T_t({\bbm1}_{R_j(Q)} f(y))|^2\, dy\Big)^{1/2} \lesssim 2^{-jn\gamma} t^{n(1/2-1/r)} \Big(\int_{R_j(Q)} | f(y)|^r\, dy\Big)^{1/r}.
\end{split}
\end{equation}
Going back to $I$, we see that
\begin{multline}\nonumber
I  \lesssim \sum_{j\geq 0} 2^{-jn\gamma}t^{n(1/2-1/r)} \Big(\sum_{Q\in \DDt} \fint_{Q^*} \Big(\int_{R_j(Q)} | f(y)|^r\, dy \Big)^{2/r} \, \nu(x) dx \Big)^{1/2}\\
  \lesssim \sum_{j\geq 0} 2^{-jn\gamma} t^{n(1/2-1/r)}(2^jt)^{n/r} \Big( \sum_{Q\in \DDt} \fint_{Q^*} \Big(\fint_{R_j(Q)} | f(y)|^r\, dy\Big)^{2/r} \, \nu(x)dx\Big)^{1/2}\\
  \lesssim \sum_{j\geq 0} 2^{-jn(\gamma-1/r)}t^{n/2}\Big(\sum_{Q\in \DDt}\fint_{Q^*} \m_r(f)(x)^2\, \nu(x)dx \Big)^{1/2}
  \lesssim \sum_{j\geq 0} 2^{-jn(\gamma-1/r)} \| \m_r( f)\|_\ltnu  \lesssim \|\m_r( f)\|_\ltnu,
\end{multline}	
since $\gamma>1/r$. Since $r<2$ and $A_{2/r}\subset A_2$, we have $I\lesssim_{[\nu]_{A_{2/r}}}   \| \m_r( f)\|_\ltnu \lesssim \|  f\|_\ltnu$.\hfill{$\square$}

\subsection{Properties of Solutions and Layer Potentials}\label{sec.prop} 

\subsubsection{Functional-analytic setup} First we  recall our definitions of    layer potentials. Formally, for instance, the single layer potential is given by $\mathcal{S}^{\cL}=(\Tr_0\circ(\cL^*)^{-1})^*f$, but we need to give a precise functional-analytic context for the operator $\cL$ to be able to talk about the traces of its inverse adjoint operator.

Let $\Omega\subseteq\bb R^{n+1}$ be an open set with Lipschitz   boundary, and fix $f\in L^1_{\loc}(\om)$, $F\in L^1_{\loc}(\om,\cee)$, and $u\in W^{1,2}_{\loc}(\om)$. We say that $u$ solves the equation $\cL u=f-\div F$ in $\om$ \emph{in the weak sense} if, for every $\varphi\in C_c^\infty(\om)$, we have that
\begin{equation}\label{eq.weak}
	\dint_{\ree} \Big( (A\nabla u + B_1 u)\cdot\overline{\nabla \varphi} + B_2\cdot \nabla u \overline{\varphi}\Big) = \dint_{\ree} \big( f\overline{\varphi} + F\cdot \overline{\nabla \varphi}\big).
\end{equation} 
  
For $p\in(1,n+1)$, we define the space $Y^{1,p}(\om)$ as
\begin{equation}\label{eq.y1p}
Y^{1,p}(\om) : = \Big\{ u \in   L^{\frac{(n+1)p}{n+1-p}}(\om): \nabla u \in L^p(\om)\Big\}.
\end{equation}
We equip this space with the norm $\lVert u \rVert_{Y^{1,p}(\om)} := \lVert u \rVert_{L^{\frac{(n+1)p}{n+1-p}}(\om)} + \lVert \nabla u \rVert_{L^p(\om)}$.  

Define the sesquilinear form $B_{\cL}:C_c^{\infty}(\bb R^{n+1})\times C_c^{\infty}(\bb R^{n+1})\ra\bb C$ via
\begin{equation*} 
	B_{\cL}[u,v]:=\dint_{\bb R^{n+1}}\Big[A\nabla u\cdot\overline{\nabla v}+uB_1\cdot\overline{\nabla v}+\overline{v}B_2\cdot\nabla u\Big],\qquad u,v\in C_c^{\infty}(\bb R^{n+1}),
\end{equation*}
and the associated operator $\cL:\n D\ra\n D'$ via the identity
\[
\langle\cL u,v\rangle=B_{\cL}[u,v],\qquad   u, v\in C_c^{\infty}(\bb R^{n+1}).
\]
In fact, when \eqref{eq.small} holds and $\rho\ll1$,  the form $B_{\cL}$ extends to a bounded, coercive form on  $Y^{1,2}(\bb R^{n+1})\times Y^{1,2}(\bb R^{n+1})$, and the operator $\cL$ extends to  an isomorphism  $Y^{1,2}(\bb R^{n+1})\ra(Y^{1,2}(\bb R^{n+1}))^*$ \cite[Proposition 2.19]{bhlmp}. Associated to $\cL$ we also have its dual $\cL^*:Y^{1,2}(\bb R^{n+1})\ra(Y^{1,2}(\bb R^{n+1}))^*$, defined by the relation
\begin{equation*}
	\langle \cL u, v\rangle = \langle u, \cL^* v\rangle,
\end{equation*}
and it is a matter of algebra to check that the identity
\begin{equation*}
	\cL^* v = -\div( A^* \nabla v + \overline{B}_2v) + \overline{B}_1\cdot \nabla v
\end{equation*}
holds in the weak sense \eqref{eq.weak} for any $v\in Y^{1,2}(\bb R^{n+1})$. In particular, $\cL^*$ is an operator of the same type as $\cL$ and if  $\rho\ll1$ so that $\cL^{-1}$ is defined, then  $(\cL^*)^{-1}$ is well defined, bounded, and satisfies $(\cL^*)^{-1}=(\cL^{-1})^*$.

We turn to the mapping properties of traces. For fixed $t\in\bb R$   we define the \emph{trace operator} $\Tr_t:C_c^{\infty}(\bb R^{n+1})\ra C_c^{\infty}(\bb R^n)$ by
\begin{equation}\label{eq.traceop}
	\Tr_tu=u(\cdot,t).
\end{equation}
Let $\mathcal{F}:L^2(\bb R^n)\ra L^2(\bb R^n)$ be the Fourier transform and write $\hat u=\mathcal{F}u$. Define $\Hf$ as  the completion of $C_c^{\infty}(\bb R^n)$ under the norm $|u|_{\dt H^{\frac12}(\bb R^n)}=\int_{\bb R^n}|\xi||\hat u(\xi)|^2\,d\xi$. Then \cite[Lemma 2.8]{bhlmp} gives that $\Tr_t$ extends uniquely to a bounded linear operator
\begin{equation}\label{eq.traceext}
	\Tr_t:Y^{1,2}(\bb R^{n+1})\ra\Hf.
\end{equation}
We write $\Hfm:=(\Hf)^*$, and we note that we are departing from notation used elsewhere in the literature, since our $\Hfm$ does not coincide with the usual (inhomogeneous) fractional Sobolev space of order $-1/2$  (for more on this, see the remarks before Proposition 2.5 in \cite{bhlmp}).

\begin{definition}\label{def.sl}
Let $\gamma\in H^{-1/2}(\rn)$. We define the \emph{single layer potential} $\sl \gamma\in \yot$ as 
\begin{equation}\label{eq.slp}
\sl \gamma= (\tr_0 \circ (\L^*)^{-1})^*\gamma.
\end{equation}
For fixed $t\in \RR$ we denote $\sl_t\gamma:= \tr_t\sl\gamma$.
\end{definition}

\begin{definition}\label{def.dl}
Let $\varphi\in H^{1/2}_0(\rn)$. We define the \emph{double layer potential} $\dlp \varphi \in \yotp$ as 
\begin{equation}\label{eq.dlp}
\dlp \varphi:= -\Phi|_{\reu} + \L^{-1} (\n F_{\Phi}^+)|_{\reu},
\end{equation}
where $\Phi\in \yot$ is any extension of $\varphi$, and $\n F_{\Phi}^+\in(Y^{1,2}(\bb R^{n+1}))^*$ is given by
\begin{equation}\nonumber
\langle\n F_\Phi^+, G\rangle:= \dint_\reu \Big(A\nb\Phi\cdot\overline{\nb G} + B_1\Phi \cdot \overline{\nb G} + B_2\cdot (\nb\Phi) \overline{G}\Big), \qquad\text{for each } G\in \yot.
\end{equation}
\end{definition}

The layer potentials are studied in detail in Section 4 of \cite{bhlmp}.

\subsubsection{Properties of weak solutions}

\begin{proposition}[Caccioppoli  Inequality in $L^p$.  {\cite[Proposition 3.9]{bhlmp}}]\label{prop-caccioppoli-inequality-lp}
There exists an open interval $I$ containing $2$ such that for every $p\in I$,  every weak solution $u\in W^{1,2}_{\loc}(\reu)$  of $\L u=0$ in $\bb R^{n+1}_+$, and every $(n+1)-$dimensional ball $B$ satisfying that $\overline{2B}\subset \reu$,  it holds that
\begin{equation}\label{eq.cacc1}
\Big(\dfint_B |\nabla u|^p\,dX\Big)^{1/p} \lesssim \frac{1}{r(B)} \Big(\dfint_{2B} |u|^p\,dX\Big)^{1/p},
\end{equation} 
with implicit constants that depend only on $n,p$, ellipticity of $\L$, and $\rho$.
\end{proposition}

\begin{proposition}[Caccioppoli  Inequality on Slices.  {\cite[Lemma 3.20]{bhlmp}}]\label{prop-caccioppoli-on-slices}
For $p, I, u$ as in Proposition \ref{prop-caccioppoli-inequality-lp}, every cube $Q\subset \rn$ and $\alpha>0$, we have that
\begin{equation}\label{eq.cacc2}
\Big(\fint_{Q}|\nabla u(x,t)|^p\, dx \Big)^{1/p} \leq \dfrac{C}{\ell(Q)}\Big(\fint_{t-\alpha\ell(Q)}^{t+\alpha\ell(Q)} \fint_{Q^*} |u(x,s)|^p\, dx ds\Big)^{1/p},
\end{equation}
whenever  $Q^*\times (t-\alpha\ell(Q),t+\alpha\ell(Q))\subset \reu$, with  $C$ depending only on $n,p, \alpha$, ellipticity of $\L$, and $\rho$.
\end{proposition}

\begin{definition}\label{def-twoplustwominus}
We define the interval $(2_-,2_+)$ as the largest open interval, symmetric around $2$ with the following two properties:
\begin{enumerate} 
	\item $2n/(n+1)=2_\#<2_-<2<2_+<2^\#:=2n/(n-1)$.
	\item If $p\in (2_-,2_+)$, then for every weak solution $u\in W_{\loc}^{1,2}(\bb R^{n+1}_+)$ of $\cL u=0$ in $\bb R^{n+1}_+$, the  estimates \eqref{eq.cacc1} and \eqref{eq.cacc2} hold, with  constants depending only on $n$, $p$, $\alpha$, ellipticity of $\cL$, and $\rho$.
\end{enumerate}
\end{definition}

\begin{proposition}\label{prop-prelim-inhomog-reverse-holder}
Let $u\in W^{1,2}_{\loc}(\Omega)$ be a solution to $\L u=\div F$ in $\Omega\subset \ree$, with $F\in L^2_{\loc}(\Omega)$. Let $B$ be an $(n+1)$-dimensional ball in $\ree$ with $2B\subset\Omega$. Then, for any $q\geq 1$ we have that
\begin{equation}\label{eq.cacc}
\Big( \dfint_B | u|^{2^*_{n+1}}\Big)^{1/2^*_{n+1}} \lesssim  \Big(\dfint_{2B} |u|^{q}\Big)^{1/q} + r(B)\Big( \dfint_{2B} |F|^2\Big)^{1/2},
\end{equation}
with implicit constants depending only on $q, n$ and ellipticity, and where we define
\begin{equation}\nonumber
\dfrac{1}{2^*_{n+1}}= \dfrac{1}{2}-\dfrac{1}{n+1}, \quad \dfrac{1}{2^*_{n+1}}+\dfrac{1}{2_{*,n+1}}=1.
\end{equation} 
\end{proposition}

\begin{proof} We first prove the result for a ball $B$ with $r(B)=1$. To simplify notation, during this proof we will write $2^*=2^{*}_{n+1}$. Fix $1\leq t<s\leq 2$, then, from the proof of the Caccioppoli inequality (see \cite[Proposition 3.1]{bhlmp}, and note that $f=0$ for us), we have
\begin{equation}\nonumber 
\| \nabla u\|_{L^2(B_t)}   \lesssim \dfrac{1}{s-t}\| u\|_{L^2(B_s)} + \| F\|_{L^2(B_s)} 
 \leq \dfrac{1}{s-t}\| u\|_{L^2(B_s)} +\| F\|_{L^2(2B)}, 
\end{equation}
where $B_t$ denotes the concentric dilate of $B$ by a factor of $t$.  On the other hand, if as usual we denote by $u_B$ the average of $u$ over $B$, then by the Poincar\'e-Sobolev inequality we have that
\begin{equation}\nonumber 
\|  u \|_{L^{2^*}(B_t)}   \lesssim t\cdot t^{(n+1)(1/2^*-1/2)} \| \nabla u\|_{L^2(B_t)} + t^{(n+1)/2^*}|u_{B_t}| 
  \lesssim \| \nabla u\|_{L^2(B_t)} +\| u\|_{L^1(2B)}, 
\end{equation}
where we use that $t\geq 1/2$. Combining these two inequalities, we obtain 
\begin{equation}\nonumber 
\| u\|_{L^{2^*}(B_t)}  \lesssim \dfrac{1}{s-t}\| u\|_{L^2(B_s)} + \| F\|_{L^2(2B)}  
 \quad + \| u\|_{L^1(2B)}. 
\end{equation}
Note that if here we set $t=1$ and $s=2$, the desired estimate (\ref{eq.cacc}) follows for $q\geq 2$. It thus remains to treat the case $q<2$. Recall, from interpolation of $L^p$ norms (here we use $q<2$) and the Cauchy inequality with a parameter, 
\begin{equation}\nonumber
\| u\|_{L^2}\leq \| u\|_{L^{2^*}}^{1-\theta}\| u\|_{L^q}^{\theta} \lesssim_\theta \eta^{1/(1-\theta)} \| u\|_{L^{2^*}} + \dfrac{1}{\eta^{1/\theta}} \| u\|_{L^{2_*}},
\end{equation}
valid for any $\eta>0$, and where $\theta$ satisfies $\frac{1}{2}= \frac{1-\theta}{2^*} + \frac{\theta}{q}$. Choosing $\eta^{1/(1-\theta)}\approx t-s$ and setting $T:=\| F\|_{L^2(2B)}+ \| u\|_{L^1(2B)}$, we arrive at  the estimate
\begin{equation}\nonumber
\| u\|_{L^{2^*}(B_t)} \leq \dfrac{1}{2}\| u\|_{L^{2^*}(B_s)} + \dfrac{C}{(s-t)^{(1-\theta)/\theta}}\| u\|_{L^{q}(B_s)} + T,
\end{equation}
for any $1\leq t<s\leq 2$. We are now in a position to apply the result in  \cite[Lemma 4.3]{Han-Lin} and conclude that
\begin{equation}\nonumber
\| u\|_{L^{2^*}(B_t)} \lesssim  \dfrac{1}{(s-t)^{(1-\theta)/\theta}}\| u\|_{L^{q}} +T.
\end{equation}
Setting now $t=1$ and $s=2$ we obtain that
\begin{equation}\nonumber
\| u\|_{L^{2^*}(B)} \lesssim \| u\|_{L^q(2B)} + \| u\|_{L^1(2B)} + \| F\|_{L^2(2B)}.
\end{equation}
This is the desired inequality, since $q\geq 1$ and $r(B)=1$. 

To obtain the result for a general $B$, we simply note that for $r>0$, $u_r(X)=u(rX)$ solves $\L_r u_r =\div F_r$ in $r\Omega$, where the coefficients of $\L_r$ are given by $A_r(X)= A(rX)$, $B_{i,r}(X)= rB_i(rX)$, $F_r(X)=rF(rX)$. It can be checked that these coefficients satisfy the same conditions as the originals, with the same relevant norms, except for $F_r$ which satisfies $\| F_r\|_{L^2(1/rB)}= r^{(-n+1)/2}\| F\|_{L^2(B)}$. The estimate (\ref{eq.cacc}) follows.
\end{proof}

\begin{proposition}[Off-diagonal Estimates. Part 1]\label{prop-vertical-off-diagonal-decay} Let $\thtm$ denote either of the following operators:
\begin{equation}\nonumber
t^m\partial_t^m \nabla\sl_t, \qquad t^m\partial_t^{m-1}\nb (\sl_t\nabla), \qquad (t^m\partial_t^{m}\nabla \sl_t 1) \cdot P_t, 
\end{equation}
where $P_t$ is an approximate identity with smooth, even, compactly supported kernel.

Let $2_-< q\leq p<2_+$ and $Q\subset \rn$ a cube. For every $h\in L^q(\rn)$, $|t|\approx \ell(Q)$, and $k\geq 0$ it holds  that
\begin{equation}\nonumber
\| \thtm (h\bbm 1_{Q})\|_{L^p(2Q)}\lesssim_{n,p,q}  |Q|^{\frac{1}{p}-\frac{1}{q}}\| h\|_{L^q(Q)}.
\end{equation}

Moreover, for any $j\geq 1$ and $t\in \RR$, if we  write $R_j(Q):= 2^{j+1}Q\backslash 2^j Q$, then
\begin{equation}\label{eq-off-diagonal-thtm}
\| \thtm g\|_{L^p(R_j(Q))} \lesssim_{n,p,q,m} 2^{-n\gamma_1} \Big(\dfrac{|t|}{2^j\ell(Q)}\Big)^m |Q|^{\frac{1}{p}-\frac{1}{q}} \| g\|_{L^q(Q)},
\end{equation}
and 
\begin{equation}\label{eq2-off-diagonal-thtm}
\| \thtm g\|_{L^p(Q)} \lesssim_{n,p,q,m} 2^{-n\gamma_1} \Big(\dfrac{|t|}{2^j\ell(Q)}\Big)^m |Q|^{\frac{1}{p}-\frac{1}{q}} \| g\|_{L^q(R_j(Q))},
\end{equation}
where $\gamma_1=1/2^\# - 1/p$ if $\thtm=t^m\partial_t^m\nabla\sl_t$ (recall that $2^\#$ is given in Definition \ref{def-twoplustwominus}), and $\gamma_1= 1/2^\#-1/p-1$ in the case that $\thtm=t^m\partial_t^{m-1}\nb(\sl_t \nb)$.	
\end{proposition}

\begin{proof}	
The proof of this result is essentially contained in \cite[Proposition 4.28]{bhlmp}. We sketch some of the modifications needed.

We start with the case   $\thtm= t^m\pd_t^m \nb \sl_t$. Here, estimate \eqref{eq-off-diagonal-thtm} was obtained in  \cite[Proposition 4.28]{bhlmp}. By duality, estimate \eqref{eq2-off-diagonal-thtm}   is equivalent to \eqref{eq-off-diagonal-thtm} for $\thtm=t^m\pd_t^m(\sl_t\nb)$; let us thus prove estimate (\ref{eq-off-diagonal-thtm}) for $\thtm=t^m\pd_t^m(\sl_t\nb)$.

Let $\rjt:= (3/2)R_j \times (t-2^j\ell(Q), t+2^j\ell(Q))$ and suppose that $g\in C_c^\infty(Q)$. Then  $(\sl_t\nbp g)= \sl_t \divp g$ is a solution in $\rjt$. By a careful application of Caccioppoli's inequality on slices, followed by the standard Caccioppoli inequality in $L^p$ $(m-1)$ times, we obtain 
\begin{equation}\nonumber 
\| t^m\pd_t^m(\sl_t\nbp)\cdot g\|_{L^p(R_j(Q))}   = |t|^m \| \pd_t^m \sl_t \divp g \|_{L^p(R_j(Q))}\lesssim \Big( \dfrac{|t|}{2^j\ell(Q)}\Big)^m \Big( \dfrac{1}{2^j\ell(Q)} \dint_{\rjt} |\sl_s \divp g(x)|^p \, dxds\Big)^{1/p}. 
\end{equation}
By duality again, it is enough to prove that  $\| \nbp \sl_s h\|_{L^{q'}(Q)} \lesssim 2^{-n\gamma_1} |Q|^{1/p-1/q} \| h\|_{L^{p'}(3/2R_j)}$, uniformly for $s\in \RR$. For this, if we define $\widetilde{Q}:= 3/2Q\times(s-\ell(Q), s+\ell(Q))$, by Caccioppoli's inequality we have that
\begin{multline}\nonumber 
\| \nbp \sl_s h\|_{L^{q'}(Q)}   \lesssim \dfrac{1}{\ell(Q)} \Big(\dfrac{1}{\ell(Q)} \dint_{\widetilde{Q}} |\sl_\tau h(x)|^{q'}\, dxd\tau\Big)^{1/q'}\\
 \lesssim |Q|^{1/q'-1/2^\#-1/n} \Big( \dfrac{1}{\ell(Q)} \int_{s-\ell(Q)}^{s+\ell(Q)} |R_j|^{q'(1/2_\#-1/p')}\| h \|_{L^{p'}(3/2R_j)}^{q'} \, d\tau \Big)^{1/q'}\lesssim 2^{n(1/p-1/2^\#)}|Q|^{1/q'-1/p'}\| h\|_{L^{p'}(3/2R_j)},  
\end{multline}
where we used the mapping property $\sl_s: L^{2_\#}(\rn)\to L^{2^\#}(\rn)$ uniformly in $s\in \RR$, and H\"older's inequality.

The above proof works, with straightforward modifications, in the case $\thtm=t^m\pd_t^{m-1}\nb(\sl_t\nb)$. The case of $\thtm=t^m\pd_t^{m}\nb\sl_t1 \cdot P_t$ is handled with the previous estimates and \cite[Lemma 3.11]{AAAHK}. 
\end{proof}

The following proposition follows the same lines as the above, the appropriate modifications being outlined in the proof of \cite[Proposition 4.37]{bhlmp}.

\begin{proposition}[Off-diagonal Estimates. Part 2]\label{prop-off-diagonal-decay-for-operators-with-B}
Let $B\in L^n(\rn;\CC^{n+1})$ and set  $\thtmb:= t^m\pd_t^m\sl_t B \cdot $ acting on functions $C_c^\infty(\rn; \CC^{n+1})$. Then, for $2_-<r<2<q<2_+$ $\thtmb$ satisfies the $L^r-L^q$ off-diagonal estimates of Definition \ref{def-lr-to-lq-off-diagonal-estimates} with $\gamma= m/n - \alpha$ for some $\alpha>0$ depending only on dimension, $r$ and $q$.
\end{proposition}

We shall also need the following quasi-orthogonality result.

\begin{proposition}[Quasi-orthogonality]\label{prop-control-for-J3}
Let $\thtm:=t^m\pd_t^m(\sl_t\nb)$, $B\in L^n(\rn;\CC^n)$, and let $\Q_s$ be a standard Littlewood-Paley family. There exists $m_0$ such that if $m\geq m_0$ and $\nu\in A_{2/r}$ (here $r$ is as in the $L^r-L^2$ off-diagonal estimates for $\thtm$ in Proposition \ref{prop-vertical-off-diagonal-decay}), then the estimate
\begin{equation}\nonumber
\Big\| \Big(\fint_{|x-y|<t} |\thtm BI_1 \Q_s^2 g(y)|^2\, dy \Big)^{1/2}\Big\|_\ltnu \lesssim_{[\nu]_{A_{2/r}}} \Big(\dfrac{s}{t}\Big)^{\beta}\| \Q_s g\|_\ltnu, \quad s<t,
\end{equation}
holds for some $\beta>0$ (possibly depending on $\nu$ only through $[\nu]_{A_{2/r}}$).
\end{proposition}

\begin{proof}
The unweighted case is proved in Lemma 4.30 in \cite{bhlmp}. The idea is to use interpolation with change of measure to reduce matters to proving a {\it uniform} weighted bound of the form
\begin{equation}\nonumber
\Big\| \Big(\fint_{|x-y|<t} |\thtm BI_1 \Q_s^2 g(y)|^2\, dy \Big)^{1/2}\Big\|_\ltnu \lesssim_{[\nu]_{A_{2/r}}} \| g\|_\ltnu.
\end{equation}
This in turn follows from the $L^r-L^2$ off-diagonal estimates of $\thtm$ in Proposition \ref{prop-vertical-off-diagonal-decay} and Proposition \ref{prop-averaged-bounds-slices-via-off-diagonal-decay}, together with the bounds for $I_1B$ from Proposition \ref{prop-i1b-bounded}; we omit the details.
\end{proof}

The first author would like to thank Moritz Egert and Olli Saari for showing him the simple computation that yields the following bound for the vertical maximal function (\cite{bes}).
\begin{proposition}\label{prop-prelim-vertical-max}
Let $m\geq 1$ and $\thtm$ be either $t^m\partial_t^m \nabla (\sl_t\nb)$, or $t^m\pd_t^m \nb \dl_t$. Then, for almost every $x\in \rn$, we have the estimate
\begin{equation}\nonumber
\sup_{t>0} |\thtm f(x)| \lesssim_m \v(\thtm f)(x) + \v(\thtmo f)(x) + |\Theta_{1,m} f(x)|.
\end{equation}
\end{proposition}

\begin{proof}
First we observe that, owing to \cite[Lemma 2.3]{bhlmp}, the function $t\mapsto\thtm f(x)=:g_t(x)$ is absolutely continuous for a.e. $x\in \rn$. Therefore, by the fundamental theorem of calculus, for such an $x\in \rn$ and every $0<s<t$, 
\begin{equation}\nonumber
|g_t(x)|^2= |g_s(x)|^2 +2\int_s^t \textup{Re}(g_t(x)\overline{\pd_\tau g_\tau (x)}) \, d\tau.
\end{equation}
Notice that
\begin{equation}\nonumber
\Big| \int_s^t \textup{Re}(g_t(x)\overline{\pd_\tau g_\tau(x)})\, d\tau \Big| \leq \int_s^t |g_t(x)||\tau \pd_\tau g_\tau(x)|\, \dfrac{d\tau}{\tau} \leq \v(g_t)(x)\v(\tau \pd_\tau g_\tau)(x),
\end{equation}
by the Cauchy-Schwarz inequality. The result now follows by setting $s=1$ and using Cauchy's inequality with a parameter.
\end{proof}

We record here also a weighted version of the Riesz transform estimates for $\Lp$ and, more importantly for us, estimates for the Hodge decomposition associated to $\Lp$.

\begin{theorem}[{\cite[Proposition 9.1]{cumr18}}]\label{thm-weighted-hodge-decomposition}
Let $\Lp := -\divp \ap \nbp$. Then there exists $M>0$ (depending only on dimension and the ellipticity of $\ap$) such that if $\nu\in A_2$ satisfies that $\nu^M\in A_2$, then 
\begin{equation}\nonumber
\max\big\{\|\nbp \Lp^{-1/2} \|_{\ltnu\to\ltnu}~,~ \| \Lp^{-1/2}\divp\|_{\ltnu\to\ltnu}\big\} \lesssim_{[\nu^M]_{A_2}} 1.
\end{equation}
In particular, if for $f\in L^2(\rn;\CC^n)$ we write Hodge Decomposition  $f=\ap\nbp F +H$ with $F\in \dt{W}^{1,2}(\rn)$ and $\divp H=0$, then for $\nu$ as above, we have that  $\|\nbp F\|_\ltnu \lesssim_{[\nu^M]_{A_2}} \| f\|_\ltnu$.  
\end{theorem}

We end this subsection with an identity characterizing the double layer in terms of operators involving only the single layer. This will allow us to focus, as far as the square and non-tangential maximal function estimates are concerned, on operators involving only the single layer.

\begin{lemma}[Double Layer Duality for $L^2$ functions]\label{lem-double-layer} 
Denote by $\vec{N}$ the outward unit normal vector of the upper-half space. The following formula holds for each $f \in C_c^\infty(\rn)$:
$$\cD_t^{\cL,+} f(x)= (\mathcal{S}_t^{\cL}\nabla)(A\vec{N}f)(x) + (\mathcal{S}_t^{\cL}\overline{B}_2)(\vec{N}f)(x).$$
\end{lemma}
\begin{proof}
We have, by Proposition 4.18 (ii) in \cite{bhlmp}, that $(\cD_t^{\cL,+} f, g) = (f, \partial_{\nu^{\cL^*,+}_{-t}}\mathcal{S}^{\cL^*}g)$  for $f, g \in C_c^\infty(\rn)$. On the other hand, since $\mathcal{S}^{\cL^*}g$ is in $\yot$, we may use the $L^2$ realization of the conormal (see Lemma 4.11 (i) in \cite{bhlmp}) from which it follows that 
\begin{multline*}
(f, \partial_{\nu^{\cL^*,+}_{-t}}\mathcal{S}^{\cL^*}g) = \langle f, \partial_{\nu^{\cL^*,+}_{-t}}\mathcal{S}^{\cL^*}g \rangle_{L^2(\rn)}  = \langle f, \vec{N} \cdot [A^*\nabla \mathcal{S}_{-t}^{\cL^*}g + \overline{B}_2\mathcal{S}_{-t}^{\cL^*}g] \rangle_{L^2(\rn)}
\\  = \langle \vec{N} f,  [A^*\nabla \mathcal{S}_{-t}^{\cL^*}g + \overline{B}_2\mathcal{S}_{-t}^{\cL^*}g] \rangle_{L^2(\rn)^n}  = \langle (\mathcal{S}_t^{\cL}\nabla)(A\vec{N}f) + (\mathcal{S}^{\cL}_tB_2)(\vec{N}f), g \rangle_{L^2(\rn)},
\end{multline*}
where we used the properties of the operator $(\mathcal{S}_t^{\cL}\nabla)$ (see Proposition 4.2 (viii) in \cite{bhlmp}) for the last line. This gives the desired identity for $f\in C_c^\infty(\rn)$.
\end{proof}

\subsection{Good Classes of Solutions}

\begin{definition}[Slice Spaces]\label{def.slicespace}For $n\geq 3$ and $p\in(1,\infty)$, we define  
\begin{equation*}
D^p_+:= \big\{ v\in C_0\big((0,\infty); L^p(\rn)\big): \| u\|_{D^p_+} <\infty\big\},
\end{equation*}
with norm given by $\| v\|_{D^p_+}:= \sup_{t>0} \| v(t)\|_{L^p(\bb R^n)}$. We also define
\begin{equation*}
\begin{split}
	S^p_+ & :=\big\{ u\in C^2_0\big((0,\infty); Y^{1,p}(\bb R^n)\big) : u'(t) \in C_0\big((0,\infty); L^p(\bb R^n)\big), \, \| u\|_{S^p_+}<\infty\big\},
\end{split}
\end{equation*}
with norm given by 
\begin{equation*}
\| u\|_{S^p_+}    := \sup_{t>0} \| u(t)\|_{Y^{1,p}(\bb R^n)}  +\sup_{t>0} \| u'(t)\|_{L^p(\bb R^n)}  + \sup_{t>0} \| tu'(t)\|_{Y^{1,p}(\bb R^n)} + \sup_{t>0}\| t^2 u''(t)\|_{Y^{1,p}(\bb R^n)}.
\end{equation*}
In particular, both $D^p_+$ and $S^p_+$ are Banach spaces. Similarly, with obvious modifications, we can define the slice spaces $D^p_-$ and $S^p_-$ in the negative half line $(-\infty, 0)$. For the rest of the article, except for Section \ref{sec.lp}, we will consider only the case $p=2$, which corresponds to the case of the problems $\Di_2$, $\Ne_2$, and $\Reg_2$.
\end{definition}

\begin{definition}[Good $\cD$ Solutions]\label{def-good-d-solutions}
We say that $u\in W^{1,2}_{\loc}(\reu)$ is a good $\cD$ solution if $\L u=0$ in $\reu$ in the weak sense, $u\in D^2_+$, and $u_\tau:=u(\cdot, \cdot+\tau)\in Y^{1,2}(\reu)$ for any $\tau>0$.
\end{definition}

\begin{definition}[Good $\mathcal{N}/\mathcal{R}$ Solutions]\label{def-good-nr-solutions}
We say that $u\in W^{1,2}_{\loc}(\reu)$ is a good $\mathcal{N}/\mathcal{R}$ solution if $\L u=0$ in $\reu$ in the weak sense, $u\in S^2_+$, and $\partial_t u_\tau\in Y^{1,2}(\reu)$ for every $\tau>0$.
\end{definition}

The following result is a companion to \cite[Corollary 6.20]{bhlmp}. Together they will imply that our uniqueness statement holds among the two most commonly used classes of solutions (those with either square or non-tangential maximal function estimates).

\begin{lemma}
Let $u\in W^{1,2}_{\loc}(\reu)$ be a solution of $\L u=0$ in $\reu$. The following holds
\begin{enumerate}[(i)]
	\item\label{item.dsol} If $\mntmt(u)\in \ltrn$, then $u$ is a good $\cD$ solution (see Definition \ref{def-good-d-solutions}).
	\item\label{item.nsol} If $\mntmt(\nb u)\in\ltrn$ then either $u$ is a good $\mathcal{N}/\mathcal{R}$ solution (see Definition \ref{def-good-nr-solutions}) if $\L 1\neq 0$, or there exists a constant $c\in \CC$ such that $u-c$ is a good $\mathcal{N}/\mathcal{R}$ solution if $\L 1=0$. 
\end{enumerate}
\end{lemma}	

\begin{proof}
As will be seen from the proof, \ref{item.dsol}  will follow the same outline as \ref{item.nsol}, and is a bit easier. We first prove that  $\sup_{t>0}\| \nb u(\cdot,t)\|_\ltrn \lesssim \| \mntmt(\nb u)\|_\ltrn$. Fix $t>0$ and let $\psi:\bb R\ra\bb R$ be a nonnegative Lipschitz cutoff function such that $\psi(t)=1$, $\psi(3t/4)=0$, and $|\psi'(s)|\leq4/t$ for each $s\in\bb R$. We make the computation
\begin{multline*}
\Vert\nabla u(\cdot,t)\Vert_2^2=\int_{\bb R^n}|\nabla u(\cdot,t)|^2\psi(t)=\int_{\bb R^n}|\nabla u(\cdot,t)|^2\psi(t)-\int_{\bb R^n}|\nabla u(\cdot,3t/4)|^2\psi(3t/4)\\=\int_{\bb R^n}\int_{3t/4}^t\partial_s\Big[|\nabla u(x,s)|^2\psi(s)\Big]\,ds\,dx \leq\int_{\bb R^n}\int_{3t/4}^t\Big[2|\nabla u(x,s)||\nabla\partial_su(x,s)|\psi(s)+|\nabla u(x,s)|^2|\psi'(s)|\Big]\,ds\,dx\\  \leq2\int_{\bb R^n}\fint_{3t/4}^t|\nabla u(x,s)|^2\,ds\,dx~+~\frac{t^2}{16}\int_{\bb R^n}\fint_{3t/4}^t|\nabla\partial_s u(x,s)|^2\,ds\,dx=:I+II,
\end{multline*}
where in the third equality we used the fundamental theorem of calculus and in the last line we used the Cauchy inequality with $\ep>0$. We now use Fubini's theorem to see that
\begin{multline}\nonumber
I=2\int_{\bb R^n}\fint_{|y-x_0|<t}\,\fint_{3t/4}^t|\nabla u(y,s)|^2\,ds\,dx_0\,dy\leq8\int_{\bb R^n}~\fiint_{\tiny\begin{matrix}|x_0-y|<t\\|s-t|<t/2\end{matrix}}|\nabla u(y,s)|^2\,ds\,dy\,dx_0\\ \leq 8 \int_{\bb R^n}\sup_{(x,\tau)\in\gamma(x_0)}\Big(~\fiint_{\tiny\begin{matrix}|x-y|<\tau\\|s-\tau|<\tau/2\end{matrix}}|\nabla u(y,s)|^2\,ds\,dy\Big)\,dx_0 =8\Vert \mntmt (\nabla u)\Vert_2^2.
\end{multline}
It remains to control $II$; for this we will use the Caccioppoli inequality as follows:
\begin{equation}\nonumber
II\leq\frac{t^2}{16}\int_{\bb R^n}\fint_{|y-x_0|<t/2}\,\fint_{3t/4}^t|\nabla\partial_s u(y,s)|^2\,ds\,dx_0\,dy\lesssim \int_{\bb R^n}\fint_{|x_0-y|<t}\,\fint_{t/2}^{5t/4}|\partial_s u(x,s)|^2\,ds\,dy\,dx_0,
\end{equation}
and thus it is clear that we may handle $II$ as above. We have obtained that for each $t>0$, $\Vert\nabla u(\cdot,t)\Vert_2\lesssim\Vert \mntmt (\nabla u)\Vert_2$ which yields the desired result.

We now improve this to  $\lim_{t\to\infty}\| \nb u(\cdot, t)\|_\ltrn=0$, where $\nb=(\nbp, \pd_t)$ is the full gradient in $n+1$ variables. This follows from the above estimate on slices: Notice that the proof actually gives that 
\begin{equation}\nonumber
\| \nb u(\cdot,t)\|_\ltrn \lesssim \| \mntm_2^{(t)}(\nb u)\|_\ltrn,
\end{equation}
where we use the truncated non-tangential maximal function (see Definition \ref{def-ntmax}) on the right hand side.  We claim now that $\mntmt^{(t)}(\nb u)(x)\to 0$ for every $x\in \rn$ as $t\to \infty$. To see this, assume to the contrary that $\limsup_{t\to\infty} \mntmt^{(t)}(\nb u)(x)>\eta>0$, for some $x\in \rn$. This means there exists a sequence $t_k\to \infty$ and points $x_k$ with $|x-x_k|<t$ such that 
\begin{equation}\nonumber
\dfint_{\mathcal{C}_{x_k,t_k}} |\nb u(y,s)|^2\, dyds >\eta^2. 
\end{equation}
By the definition of the non-tangential maximal function we then have 
\begin{equation}\nonumber
\mntmt(\nb u)(z)^2 \geq \dfint_{\mathcal{C}_{x_k,t_k}} |\nb u(y,s)|^2\, dyds >\eta^2,
\end{equation}
for every $z\in \rn$ such that $|z-x_k|<t_k$. Integrating over this set gives
\begin{equation}\nonumber
\| \mntmt (\nb u)\|_\ltrn^2 \geq \int_{|z-x_k|<t_k} \mntmt (\nb u)(z)^2\, dz \geq c_n \eta^2t_k^n. 
\end{equation}
Since $t_k\to \infty$, this contradicts our assumption that $\mntmt (\nb u)\in \ltrn.$ With the claim now proved, and since $\mntmt^{(t)}(\nb u)\leq \mntmt (\nb u)$ by definition, the dominated convergence theorem gives 
\begin{equation}\nonumber
\| \nb u (\cdot, t)\|_\ltrn \lesssim \| \mntmt^{(t)} (\nb u)\|_\ltrn \to 0,\qquad\text{as }t\to\infty.
\end{equation} 

Appealing to Caccioppoli's inequality and the above, together with \cite[Proposition 6.14]{bhlmp}, we see that $u\in S^2_+$ when $\L 1\neq 0$. If $\L 1 =0$, we proceed as follows: First, by the sup on slices estimate above and Caccioppoli's inequality on slices we see that $\pd_t u(\cdot, t)\in W^{1,2}(\rn)$ for every $t>0$; in particular 
\begin{equation}\nonumber
\int_s^t \pd_\tau u(\cdot, \tau)\, d\tau \in W^{1,2}(\rn)\subset \yotn, \quad \text{for all } 0<s<t<\infty.
\end{equation}
On the other hand, again by the sup on slices and \cite[Lemma 2.1]{bhlmp}, we have that for every $t>0$ there exists a constant $c_t\in \CC$ such that $u(\cdot,t)-c_t\in \yotn$. Therefore, by the fundamental theorem of calculus, for any $0<s<t<\infty$,
\begin{equation}\nonumber
\int_s^t\pd_\tau u(\cdot,\tau)\, d\tau -(c_t-c_s)= u(\cdot, t)-c_t-[u(\cdot,s)-c_s] \in \yotn.
\end{equation}
We conclude $c_t=c_s=c$ as desired, and so $u-c\in S^2_+$.

Finally we show  $\pd_t u_\tau:=\pd_t u(\cdot,\cdot+\tau)\in Y^{1,2}(\reu)$ for every $\tau>0$. For this we simply compute, decomposing $\rn$ into cubes in $\DD_s$ and using Caccioppoli's inequality on slices together with Fubini's Theorem 
\begin{equation}\nonumber
\dint_\reu |\nb \pd_tu (y,s+\tau)|^2\, dyds = \int_\tau^\infty \int_\rn |\nb \pd_t u(y,s)|^2\, dyds\lesssim  \sup_{t>0}\| \nb u(\cdot,t)\|_\ltrn \int_\tau^\infty\dfrac{ds}{s^2} <\infty.
\end{equation}

For \ref{item.dsol}, we run the same argument with $u$ in place of $\nb u$.
\end{proof}

\section{Two General Extrapolation Results}

In this section we prove two extrapolation theorems for conical and vertical square functions. The takeaway from these considerations is that conical square functions have good estimates in the range $(r,\infty)$ in the presence of $L^r-L^2$ off-diagonal estimates plus an $L^2$ square function bound. The vertical square function on the other hand requires (for our argument) that the operator satisfies a reverse H\"older inequality (and in fact, in this case we see that the vertical square function is controlled by the conical square function on an interval around $p=2$; this should be compared with Proposition \ref{prop-prelim-comparability-square-functions} which is optimal for general functions $F$, see \cite[Proposition 2.1 (c)]{ahm}).

\begin{lemma}[Extrapolation for Conical Square Functions]\label{lem-extrapolation-conical}
Suppose $T_t$ is an operator satisfying, for $q=2$ and some  $r<2$, the off-diagonal estimates\footnote{In fact we will only need the first and second estimates in Definition \ref{def-lr-to-lq-off-diagonal-estimates} for $T_t$, in the range $|t|\approx \ell(Q)$.} in Definition \ref{def-lr-to-lq-off-diagonal-estimates} with $\gamma>1/r$. (Notice this allows us to define $T_t 1$ as an element of $L^2_{\loc}$.) Set $R_t:=T_t- T_t 1\cdot P_t$ for a given approximate identity $P_t$ with compactly supported kernel of the form $P_t = \widetilde P_t \widetilde P_t$ for another approximate identity $\widetilde P_t$. Finally assume that for every $f\in L^2(\rn)$, $\| \s(T_t f)\|_\ltrn \lesssim \| f\|_\ltrn$, and 
\begin{equation}\label{eq-1-quasi-orthogonality-assumption}
\| R_t \Q_s^2 f\|_{\ltrn}\lesssim  \Big( \dfrac{s}{t}\Big)^{\beta} \| \Q_s f\|_\ltrn, \quad\text{for } s\leq t,
\end{equation}
for some (and therefore any) CLP family $\Q_s$ (see Definition \ref{def-clp-family}) and some $\beta>0$. Then  
\begin{equation}\label{eq-weithed-estimate-conical-extrapolation}
\| \s(T_t f)\|_\ltnu \lesssim \| f\|_\ltnu,\quad  \text{for each } \nu\in A_{2/r}.
\end{equation}
In particular, $\| \s(T_t f)\|_\lprn \lesssim \| f\|_\lprn$,  for each  $p \in(r,\infty)$. 
\end{lemma}

The above lemma can be thought of as a Calder\'on-Zygmund-type theorem. In this case the off-diagonal decay plays the role of the usual size condition while the quasi-orthogonality estimate for $R_t$ plays the role of H\"older continuity of the kernel. Note also that the case $t\leq s$ in the quasi-orthogonality estimate \eqref{eq-1-quasi-orthogonality-assumption} is a consequence of the off-diagonal decay of $R_t$ and \cite[Lemma 3.5]{AAAHK}. Therefore, with the off-diagonal decay of $R_t$ as a background assumption, \eqref{eq-1-quasi-orthogonality-assumption} is equivalent to 
\begin{equation}\nonumber
\| R_t \Q_s^2 f\|_\ltrn \lesssim \min \Big(\dfrac{t}{s}, \dfrac{s}{t}\Big)^{\beta} \| \Q_s f\|_\ltrn.
\end{equation}

\begin{proof} Let $f\in C_c^\infty(\rn)$. We begin by writing 
\begin{equation}\label{coifmey.eq}
T_t f(x)= R_t f(x)+ [T_t 1(x)]\cdot P_t f(x),
\end{equation}
where $R_t$ and $P_t$ are as in the hypotheses. To handle the first term we use interpolation with change of measure (see the proof of Theorem \ref{thm-weighted-lp}) to reduce the weighted estimate of $\s(R_t)$ to the pair of estimates 
\begin{equation}\label{eq-1-conical-extrapolation}
\Big\| \Big(\fint_{|x-y|<t} |R_t \Q_s^2f(y)|^2\, dy\Big)^{1/2} \Big\|_\ltrn \lesssim  \min\Big(\dfrac{s}{t}, \dfrac{t}{s}\Big)^{\beta} \| \Q_s f\|_\ltrn,
\end{equation}
for some $\beta>0$, and 
\begin{equation}\label{eq-2-conical-extrapolation}
\Big\| \Big(\fint_{|x-y|<t} |R_t\Q_s^2 f(y)|^2\, dy\Big)^{1/2}\Big\|_\ltnu\lesssim_{[\nu]_{A_{2/r}}} \| \Q_s f\|_\ltnu,
\end{equation}
for $r$ as in the statement of the lemma.

The unweighted quasi-orthogonality estimate \eqref{eq-1-conical-extrapolation} follows from Fubini's Theorem and the good off-diagonal decay. 

The uniform weighted estimate follows from Proposition \ref{prop-averaged-bounds-slices-via-off-diagonal-decay} and the fact that $|\Q_s h(x)|\lesssim \m h(x)$ and $\m$ is bounded on $\ltnu$ (because $A_{2/r}\subset A_2$). This shows the desired weighted estimate, and so by interpolation with change of measure, 
\begin{equation}\nonumber
\Big\| \Big(\fint_{|x-y|<t} |R_t \Q_s^2 f(y)|^2\, dy \Big)^{1/2}\Big\|_\ltnu \lesssim \min\Big(\dfrac{s}{t},\dfrac{t}{s}\Big)^{\beta} \|\Q_s f\|_\ltnu,
\end{equation} 
(for a possibly smaller $\beta$ than the one for \eqref{eq-1-conical-extrapolation}). The estimate $\| \s(R_t f)\|_\ltnu \lesssim \| f\|_\ltnu$ now follows from a standard quasi-orthogonality argument, once one realizes that $\s(R_t)=\v(\widetilde{R}_t)$ if 
\begin{equation}\nonumber
\widetilde{R}_t h(x):= \Big(\fint_{|x-y|<t} |R_t h(y)|^2\, dy\Big)^{1/2}.
\end{equation} 
Now it remains to establish the square function bound for $T_t 1(x)\cdot P_t$. For this we first claim that the measure
\begin{equation}\nonumber
d\mu(x,t):= \Big(\fint_{|x-y|<t}|T_t 1(y)|^2\, dy\Big)\, d\nu(x) \dfrac{dxdt}{t},
\end{equation}
is a $\nu$-Carleson measure, i.e. that for every cube $Q$, $\mu(R_Q)\lesssim \nu(Q)$, where $R_Q:=Q\times (0,\ell(Q))$. Let us assume the claim for a moment. By a weighted version of Carleson's lemma (Lemma \ref{lem-prelim-weighted-carleson})  and  the fact that  $|P_t f(y)|\lesssim \widetilde{P}_t(\m f)(x)$ whenever $|x-y|<t$ (since $P_t=\widetilde{P}_t\widetilde{P}_t$),  we obtain that
\begin{multline}\nonumber 
\int_\rn \dint_{\Gamma(x)} |(T_t 1(y))|^2|P_tf(y)|^2 \, \dfrac{dydt}{t^{n+1}} \, \nu(x)dx  \lesssim \int_\rn \int_0^\infty |\widetilde{P}_t(\m f)(x)|^2\, d\mu(x,t) \\
  \lesssim \int_\rn \mathcal{N}(\widetilde{P}_t(\m f))(x)^2\, \nu(x)dx 
  \lesssim \int_\rn \m(\m f)(x)^2\, \nu(x)dx 
  \lesssim \int_\rn |f(x)|^2 \, \nu(x)dx, 
\end{multline}
where we used the fact that $\m: L^2(\nu)\to L^2(\nu)$ since $r>1$. This accounts for the contribution of the second term in \eqref{coifmey.eq}, using Theorem \ref{thm-weighted-lp}.

To prove the claim we invoke  Lemma \ref{lem-prelim-john-nirenberg-local-square-functions} and the reverse H\"older inequality for $A_{2/r}$ weights (see Proposition \ref{prop-properties-ap-weights}) in the following way: For a fixed  cube $Q\subset \rn$, using H\"older's inequality
\begin{multline}\nonumber 
\mu(R_Q)   = \int_Q \int_0^{\ell(Q)} \fint_{|x-y|<t<\ell(Q)} |T_t1(y)|^2\, \dfrac{dydt}{t} \, \nu(x)dx  = \int_Q \Big( \dint_{|x-y|<t<\ell(Q)} |T_t 1(y)|^2 \, \dfrac{dydt}{t^{n+1}} \Big) \nu(x)\, dx   \\
  =: \int_Q A_Q^2(x) \, \nu(x)dx  \leq \Big( \int_Q A_Q^{2(1+\delta_1)}\, dx \Big)^{1/(1+\delta_1)}\Big( \int_Q \nu^{1+\delta_2}\, dx \Big)^{1/(1+\delta_2)}, 
\end{multline}
where $\delta_1=1/\delta_2$ and $1+\delta_2$ is the exponent corresponding to the reverse H\"older inequality for $\nu$ so that
\begin{equation}\nonumber 
\mu(R_Q)  \lesssim \Big( \int_Q A_Q^{2(1+\delta_1)}\, dx \Big)^{1/(1+\delta_1)} |Q|^{-1+1/(1+\delta_2)} \nu(Q)
  \lesssim |Q|^{1/(1+\delta_2)} |Q|^{-1+ 1/(1+\delta_1)} \nu(Q) 
  = \nu(Q), 
\end{equation}
where we used  Lemma \ref{lem-prelim-john-nirenberg-local-square-functions} in the second to last line. We should remark here that the implicit constant depends on $\delta_1$ and the constant in the reverse H\"older inequality for $\nu$, but these in turn depend only on $[\nu]_{A_{2/r}}$ (see for instance \cite{s93}). This finishes the proof of the weighted estimate \eqref{eq-weithed-estimate-conical-extrapolation}. The unweighted result now follows from Corollary \ref{cor-weighted-l2-implies-lp-one-sided-restriction}. 
\end{proof}

We now proceed to the extrapolation result for vertical square functions. The idea will be the same, which is to reduce matters to a weighted $L^2$ estimate. However, notice that before we used crucially the properties of cones in both the weighted estimates for $R_t$ and the Carleson measure estimate for $T_t$ in Lemma \ref{lem-extrapolation-conical}. In order to handle this issue we will transform $\v(T_t)$ into $\s(\widetilde{T}_t)$ for an appropriate $\widetilde{T}$ involving the weight; this makes the analysis more involved than in Lemma \ref{lem-extrapolation-conical}.

\begin{lemma}\label{lem-extrapolation-vertical}
Let $T_t$ be an operator satisfying, for some $r<2<q$ and $\delta\in (0,1)$, the $L^r-L^q$ off-diagonal estimates in Definition \ref{def-lr-to-lq-off-diagonal-estimates} for some $\gamma>-1/n+2/r+\log_2(C_\delta)/n$ (here $C_\delta$ is as in (viii) of Proposition \ref{prop-properties-ap-weights}). We  also require that, for every cube $Q\subset \rn$,
\begin{equation}\label{revholdassump}
\Big( \dfint_{I(Q)} | T_t f(x) |^q\, dxdt \Big)^{1/q} \lesssim \Big( \dfint_{I(Q^*)} |S_t f(x)|^2\, dxdt\Big)^{1/2},
\end{equation}
where   $I(Q)=Q\times (\ell(Q)/2,\ell(Q))$ and $S_t$ is an operator satisfying $\| \s(S_t)\|_{\ltrn\to\ltrn}<\infty$. The assumption on $\gamma$ allows us to define $T_t 1$ as an element of $L^2_{\loc}$, and we set 
\begin{equation}\label{eq.rt}
R_t f(x):= [T_t - T_t 1(x)\cdot P_t](f)(x),
\end{equation}
for some approximate identity $P_t$ with compactly supported kernel. Suppose further that $R_t$ satisfies the quasi-orthogonality estimate $\| R_t\Q_s^2 f\|_\ltrn \lesssim(\frac{s}{t})^\beta \| f\|_\ltrn$, $s<t$, for all $f\in \ltrn$ and some $\beta>0$, and that $T_t$ satisfies the $L^2$ square function estimate $\| \v(T_t f)\|_\ltrn \lesssim \| f\|_\ltrn$. Then, if $\nu\in RH_M\cap A_1$ for $M>\max(2r/(r-2), (q/2)')$ and $[\nu]_{A_1}\leq C_\delta$, we have that
\begin{equation}\nonumber
\| \v(T_t f)\|_\ltnu\lesssim \| f\|_\ltnu.
\end{equation}
In particular, for any $p\in (2-\delta/M, 2+\delta/M)$, it holds that
\begin{equation}\nonumber
\| \v(T_t f)\|_\lprn \lesssim \| f\|_\lprn.
\end{equation}  
If $T_t1=0$, that is, if $T_t=R_t$, then we can dispense of (\ref{revholdassump}).
\end{lemma}
 
\begin{remark}\label{weakeningRH} As will be seen from the proof, we can weaken the reverse H\"older condition on $T_t$ to 
\begin{equation}\nonumber
\Big( \dfint_{I(Q)} |T_t f(x)|^{\bar{q}}\, dxdt \Big)^{1/\bar{q}} \lesssim \Big(\dfint_{I(Q^*)} |S_t f(x)|^{\bar{q}}\, dxdt\Big)^{1/\bar{q}},
\end{equation}
for every $r\leq \bar{q}\leq q$, and where the operator $S_t$ satisfies both $\| \s(S_t f)\|_\ltrn \lesssim \|f\|_\ltrn$  and a reverse H\"older inequality. In our intended application  where $T_t =t^m\pd_t^m \nb \sl_t$, we do not have a reverse H\"older inequality for $T_t$, but we do have such an estimate for solutions, $S_t=t^{m-1}\pd_t^m\sl_t$.
\end{remark}  

\begin{proof}
We  note that, by Proposition \ref{prop-prelim-comparability-square-functions}, in the range $p<2$ we have  that $\| \v(T_t f)\|_\lprn \lesssim \| \s(T_t f)\|_\lprn$, and for $r<p<2$, by Lemma  \ref{lem-extrapolation-conical} (recall that vertical and conical square functions coincide on $L^2$) and Corollary \ref{cor-weighted-l2-implies-lp-one-sided-restriction}, we have $\| \s (T_t f)\|_\lprn \lesssim \| f\|_\lprn$. Therefore, it is enough to consider the case $p>2$. We   proceed to rewrite our vertical square function into a conical square function by introducing an average adapted to $\nu$. For this purpose we set, for $x\in \rn$ and $t>0$ fixed,
\begin{equation}\nonumber
\nuxt:= \fint_{|x-y|<t} \nu(y)\, dy.
\end{equation}
We thus write, using Fubini's theorem,
\begin{multline}\nonumber 
\int_\rn |\v(T_t f)(x)|^2\, \nu(x) dx   = \int_\rn \int_0^\infty |T_t f(x)|^2\, \dfrac{dt}{t}\, \nu(x)dx
  = \int_\rn \int_0^\infty \fint_{|x-y|<t} \dfrac{\nu(y)}{\nuxt} |T_t f(x)|^2\, dy \dfrac{dt}{t^{n+1}}\, \nu(x)dx\\
  = \int_\rn \dint_{|x-y|<t} \dfrac{\nu(x)}{\nuxt}|T_t f(x)|^2\, \dfrac{dxdt}{t^{n+1}} \, \nu(y) dy
  =: \int_\rn \dint_{|x-y|<t} |\tit f(x)|^2\, \dfrac{dxdt}{t^{n+1}}\, \nu(y)dy
  = \| \s(\tit f)\|_\ltnu. 
\end{multline}
We are now in a position to try and mimic the proof of Lemma \ref{lem-extrapolation-conical}. Unfortunately the process is quite a bit more involved and, rather than proving a full weighted estimate, we will use the specific form of our weight $\nu$. To simplify notation we introduce the operators:
\begin{equation}\nonumber
\rtt f(x)= \sqrt{\dfrac{\nu(x)}{\nuxt}} R_t,
\end{equation}
where $R_t$ is as in (\ref{eq.rt}). It follows that $\tit f(x)= \rtt f(x)+ \tit 1(x)\cdot P_t f(x)$. 

As was done in the case of the conical square function, to handle the second term it is enough to show the $\nu$-Carleson measure estimate (see Lemma \ref{lem-prelim-weighted-carleson}) $
\mu(R_Q)\lesssim \nu(Q)$  for every cube $Q\subset \rn$, where $R_Q:=Q\times (0,\ell(Q))$ and the measure $\mu$ is defined as 
\begin{equation}\label{eq-carlesonmeasure}\nonumber
d\mu(x,t):= \Big( \fint_{|x-y|<t} |\tit 1(y)|^2\, dx \Big)\, \nu(x)\dfrac{dxdt}{t}.
\end{equation}

For this we reduce matters to an unweighted estimate via  Lemma \ref{lem-prelim-john-nirenberg-local-square-functions} as follows: For any $\bar{q}>1$, 
\begin{multline}\nonumber 
\mu(R_Q)   = \int_Q \Big( \dint_{|x-y|<t<\ell(Q)} |\tit 1(x)|^2\, \dfrac{dxdt}{t^{n+1}} \Big)\, \nu(y)\, dy 
  =: \int_Q A_Q^2(y)\, \nu(y)\,dy \\
  \leq \Big( \int_Q A_Q^{2\bar{q}}\, dy \Big)^{1/\bar{q}} \Big( \int_Q \nu(y)^{\bar{q}'}\, dy\Big)^{1/\bar{q}'} 
  \lesssim_{[\nu]_{RH_{\bar{q}'}}} \Big( \int_Q A_Q^{2\bar{q}}\Big)^{1/\bar{q}} |Q|^{-1/\bar{q}} \int_Q \nu(y)\, dy
  = \Big( \fint_Q A_Q^{2\bar{q}} \, dy \Big)^{1/\bar{q}}\nu(Q), 
\end{multline}
where as before the quantity $[\nu]_{RH_{\bar{q}'}}$ is admissible if say $M>\bar{q}'$(see Proposition \ref{prop-properties-ap-weights}). Therefore it is enough to show that  $( \fint_Q A_Q^{2\bar{q}}\, dy)^{1/\bar{q}} \lesssim 1$. Furthermore, by      Lemma \ref{lem-prelim-john-nirenberg-local-square-functions}, we reduce to proving that $\fint_Q A_Q^2\, dy \lesssim 1$. Using Fubini's theorem, we see that this last estimate is equivalent to the unweighted Carleson   estimate
\begin{equation}\label{eq-vertical-unweighted-carleson}
\int_0^{\ell(Q^*)}\int_{Q^*} |\tit 1(x)|^2\, \dfrac{dxdt}{t} \lesssim |Q^*|,
\end{equation}
where as before $Q^*=c_nQ$ is a dilate of $Q$. Since the above has to hold for every cube, we write $Q$ in place of $Q^*$ in what follows. Moreover, since the quantity $\nu/\nuxt$ is invariant under scalar multiplication of $\nu$ by a positive constant, for a fixed cube $Q$ we may assume that $\nu(Q)/|Q|=1$.

First we use a stopping time argument to deal with $\nuxt$: For a fixed constant $\Lambda^{-1}<1/4$, to be selected later, we let $\{ Q_j\}_{j\in \NN}$ be the collection of maximal dyadic sub-cubes of $Q$ with respect to the conditions 
\begin{equation}\nonumber
\fint_{Q_j} \nu(x)\, dx>\Lambda, \quad \text{ or } \quad \fint_{Q_j} \nu(x)\, dx <\Lambda^{-1}.
\end{equation}
We say $j\in I_1$ if the first condition holds, and $j\in I_2$ if the second does. By the first condition we have
\begin{equation}\nonumber
\sum_{j\in I_1} |Q_j|< \sum_{j\in I_1} \Lambda^{-1}\int_{Q_j}\nu(x)\, dx \leq\Lambda^{-1} \int_Q \nu(x)\, dx =\Lambda^{-1}|Q|,
\end{equation}
since $\nu(Q)/|Q|=1$. On the other hand  if $j\in I_2$,
\begin{equation}\nonumber
\int_{Q_j} \nu(x)\, dx <\Lambda^{-1}|Q_j|, \quad \text{ and } \quad \int_{Q_j^*} \nu(x)\, dx \geq\Lambda^{-1}|Q_j^*|,
\end{equation}
where $Q_j^*$ is the dyadic parent of $Q_j$. Therefore, $\sum_{j\in I_2} \nu(Q_j) \leq \sum_{j\in I_2}\Lambda^{-1}|Q_j| \leq\Lambda^{-1}|Q|=\Lambda^{-1}\nu(Q)$. By the $A_\infty$ property of $\nu$ we can choose $\Lambda$, depending only on the $A_1$ characteristic of $\nu$, small enough such that the above inequality implies that $| \cup_{j\in I_2} Q_j| < \frac{1}{2} |Q|$. Combining this with the corresponding estimate for $I_1$, and using the fact that the cubes $Q_j$ are pairwise disjoint, we see that $\sum_{j\geq 0} |Q_j| < B|Q|$, for $B=1/2 +\Lambda^{-1}<1$. By  Lemma \ref{lem-john-nirenberg-carleson-measures}, the above implies that it is enough to show that
\begin{equation}\label{eq-vertical-carleson-measure-sawtooth1}
\dint_{E_Q}  \vwt |T_t 1(x)|^2\, dx \dfrac{dxdt}{t} \leq C_0|Q|,
\end{equation}
where we define the sawtooth region $E_Q:= R_Q\backslash( \cup_{j\geq 0} R_{Q_j})$. To handle \eqref{eq-vertical-carleson-measure-sawtooth1} we first claim the following:
\begin{equation}\label{eq-vertical-nuxt-bounded-below}
\nuxt \gtrsim 1, \qquad \text{for each } (x,t)\in E_Q,
\end{equation}
with implicit constants depending only on the doubling constant of $\nu$. To see this fix $(x,t)\in E_Q$ and consider first the case $x\notin \cup_{j\geq 0} Q_j$ so that $\Lambda^{-1}\leq \fint_{Q'} \nu(y)\, dy \leq\Lambda$,  for any dyadic subcube $Q'\in \DD(Q)$ containing $x$. In particular, choosing $Q_t'\in \DD_t(Q)$ with this property, and using the doubling property of $\nu$ we see 
\begin{equation}\nonumber
\Lambda^{-1}\leq \fint_{Q_t'} \nu(y)\, dy \approx \fint_{|x-y|<t}\nu(y)\, dy= \nuxt.
\end{equation}
On the other hand if $x\in Q_j$ for some $j\geq 0$ we proceed as follows: If $t>4\ell(Q_j)$, it means that, if as before $Q_t'\in \DD_t(Q)$ is the unique dyadic subcube of $Q$ containing $x$, then $Q_t'$ is not in the collection $\{ Q_j\}_j$ so by definition $\Lambda^{-1}\leq \fint_{Q_t'} \nu(y)\, dy \leq\Lambda$, and we conclude as before since this average is comparable, by doubling of $\nu$, to $\nuxt$. If $\ell(Q_j)\leq t\leq 4\ell(Q_j)$ (the first inequality owing to the definition of $E_Q$) then by definition the dyadic parent $\widetilde{Q}_j$ of $Q_j$ satisfies $
\Lambda^{-1}\leq \fint_{\widetilde{Q}_j} \nu(y)\, dy \leq\Lambda$, so that, again by doubling of $\nu$, the claim follows.

We conclude, using \eqref{eq-vertical-carleson-measure-sawtooth1} and \eqref{eq-vertical-nuxt-bounded-below}, that it is enough to establish (recall $\nu(Q)=|Q|$)
\begin{equation}\label{eq-vertical-last-carleson-measure-estimate}
\dint_{E_Q} |T_t 1(x)|^2 \nu(x)\, \dfrac{dxdt}{t}\leq C_0|Q|.
\end{equation}
To show this we first fix $\psi=\psi_Q\in C_c^\infty(4Q)$ with the property that $\psi\equiv 1$ in $2Q$ and $0\leq \psi\leq 1$, so that 
\begin{equation}\nonumber 
\dint_{E_Q} |T_t 1(x)|^2\nu(x)\, \dfrac{dxdt}{t}  \lesssim \dint_{E_Q} |T_t \psi(x)|^2\nu(x)\, \dfrac{dxdt}{t}   + \dint_{E_Q} |T_t (1-\psi)(x)|^2\nu(x)\, \dfrac{dxdt}{t} =: II+III. 
\end{equation}
We first handle $III$: Using H\"older's inequality, with $q/2>1$ as in the hypotheses, we recall that we have chosen $M>(q/2)'$ so that $\nu\in RH_{(q/2)'}$ (see Proposition \ref{prop-properties-ap-weights}), 
\begin{multline}\nonumber 
III   \leq \int_0^{\ell(Q)}\int_Q |T_t (1-\psi)|^{2}\nu(x)\, \dfrac{dxdt}{t}  \leq \int_0^{\ell(Q)} \Big( \int_Q \nu^{(q/2)'}(x)\, dx\Big)^{1/(q/2)'}\Big( \int_Q |T_t(1-\psi)|^{q}\, dx \Big)^{2/q}\, \dfrac{dt}{t}\\
  \lesssim \int_0^{\ell(Q)}|Q|^{2/q}\nu(Q) \Big( \int_Q |T_t(1-\psi)|^{q}\, dx \Big)^{2/q}\, \dfrac{dt}{t}
  =\int_0^{\ell(Q)} |Q|^{1-2/q} \Big( \int_Q |T_t(1-\psi)|^{q}\, dx \Big)^{2/q}\, \dfrac{dt}{t}, 
\end{multline}
where we used the normalization $\nu(Q)/|Q|=1$ in the last line. Now, since $T_t$ satisfies $L^2-L^q$ off-diagonal estimates (see Definition \ref{def-lr-to-lq-off-diagonal-estimates}), using as usual $R_j=R_j(Q)=2^{j+1}Q\backslash 2^jQ$ for $j\geq 1$ and recalling that $1-\psi\equiv 1$ outside of $4Q$,
\begin{multline}\nonumber 
\Big( \int_Q |T_t(1-\psi)|^{q}\, dx \Big)^{1/q}   \leq \sum_{j\geq 1} \Big( \int_Q |T_t[(1-\psi)\bbm 1_{R_j}]|^{q}\, dx \Big)^{1/q}\\
  \leq \sum_{j\geq 1} 2^{-nj\gamma_1}\Big(\dfrac{t}{(2^j\ell(Q))}\Big)^{n\gamma_2}\ell(Q)^{n(1/q-1/2)} \Big( \int_{R_j} |1-\psi|^2\, dx \Big)^{1/2}\\
  \lesssim \sum_{j\geq 1} 2^{-nj\gamma}\Big(\dfrac{t}{\ell(Q)}\Big)^{n\gamma_2}\ell(Q)^{n(1/q-1/2)} |R_j|^{1/2}  \lesssim \sum_{j\geq 1} 2^{-nj(\gamma-1/2)}\Big( \dfrac{t}{\ell(Q)}\Big)^{n\gamma_2} |Q|^{1/q} \lesssim \Big( \dfrac{t}{\ell(Q)}\Big)^{n\gamma_2}|Q|^{1/q} 
\end{multline}
since $\gamma>1/2$. Plugging this into the estimate for $III$ above we see, since $\gamma_2>0$, 
\begin{equation}\nonumber
\begin{split}
III\lesssim \int_0^{\ell(Q)}|Q| \Big(\dfrac{t}{\ell(Q)}\Big)^{2n\gamma_2}\, \dfrac{dt}{t} \lesssim |Q|.
\end{split}
\end{equation}
This is the desired estimate for $III$.

To handle $II$ we first define, for $Q'\in \DD(Q)$, $I(Q'):= \{ (x,t) \in R_Q: x\in Q', \, \ell(Q')/2 < t\leq \ell(Q')\}$, the Whitney region in $\reu$ associated to $Q'$. We see that 
\begin{equation}\nonumber
\begin{split}
II & = \sum_{\substack{Q'\in \DD(Q) \\ Q'\cap Q_j=\emptyset, \, \forall j}} \dint_{I(Q')} |T_t \psi(x)|^2\nu(x)\, \dfrac{dxdt}{t}.
\end{split}
\end{equation}
We now use H\"older's Inequality with $q>2$ so that the $L^{q}$ reverse H\"older inequality for $T_t$ holds, again noting that we have chosen $M$ large enough to guarantee $\nu\in RH_{(q/2)'}$, to conclude that
\begin{equation}\nonumber 
II   \leq \sum_{\substack{Q'\in \DD(Q) \\ Q'\not\subset  Q_j, \, \forall j}}\Big(  \dint_{I(Q')} |T_t \psi|^{q}\, \dfrac{dxdt}{t} \Big)^{\frac2q} \Big( \dint_{I(Q')} \nu^{(\frac{q}2)'}\, \dfrac{dxdt}{t} \Big)^{1/(\frac{q}2)'}\\
  \lesssim  \sum_{\substack{Q'\in \DD(Q) \\ Q'\not\subset  Q_j, \, \forall j}} |Q'|^{1/(\frac{q}2)'}\Big(  \dint_{I(Q')} |T_t \psi|^{q}\, \dfrac{dxdt}{t} \Big)^{\frac2q}, 
\end{equation}
where we used that for $Q'$ satisfying $Q'\not\subset  Q_j $ for all $j$ (i.e. for $Q'$ not contained in any of the $Q_j$) we have, by construction of the $Q_j$, $\fint_{Q'} \nu(x)\, dx \approx 1$. We now use reverse H\"older assumption on $T_t$ to obtain 
\begin{equation}\nonumber
\Big(  \dfint_{I(Q')} |T_t \psi|^{q}\, \dfrac{dxdt}{t} \Big)^{2/q} \lesssim    \dfint_{I([Q']^*)} |S_t \psi|^{2}\, \dfrac{dxdt}{t}.
\end{equation}
Therefore, using this in the estimate for $II$,
\begin{equation}\nonumber 
II   \lesssim  \sum_{\substack{Q'\in \DD(Q) \\ Q'\not\subset  Q_j, \, \forall j}}|Q'|^{1/p'}|Q'|^{1/p} \dfint_{I([Q']^*)} |S_t \psi|^2\, dxdt
  \lesssim \sum_{\substack{Q'\in \DD(Q) \\ Q'\not\subset  Q_j, \, \forall j}} \dint_{I([Q']^*)} |S_t\psi|^2\, \dfrac{dxdt}{t}   \lesssim \dint_{R_{Q^*}} |S_t\psi |^2\, \dfrac{dxdt}{t}. 
\end{equation}
The desired estimate now follows from the fact that $T_t$ satisfies an $L^2(\rn)$ square function estimate and $\| \psi\|_\ltrn \lesssim |Q|^{\frac12}$ by construction. Combining the estimates for $II$ and $III$, \eqref{eq-vertical-last-carleson-measure-estimate} follows and thus, by our previous reductions, we have shown
\begin{equation}\nonumber
\| \s(T_t 1\cdot P_t f)\|_\ltnu \lesssim \| f\|_\ltnu.
\end{equation}

It remains to handle the contribution of $\rtt$. Notice that so far, we have only required that $\nu\in A_{2/r}$ and $\gamma>1/r$. The extra assumptions will be needed in order to handle $\rtt$. Again as in the proof of Lemma \ref{lem-extrapolation-conical} we will appeal to interpolation with change of measure (see Theorem \ref{thm-weighted-lp}). For this it is enough to prove the following pair of estimates:
\begin{equation}\label{eq-1-vertical}
\Big\| \Big(\fint_{|x-y|<t} |\rtt \Q_s^2f(y)|^2\, dy\Big)^{1/2} \Big\|_{\ltrn} \lesssim_{[\nu^M]_{A_1}} \min\Big(\dfrac{t}{s},\dfrac{s}{t}\Big)^\beta\| \Q_sf\|_\ltrn,
\end{equation}
valid for some (and therefore all) Littlewood-Paley family $(\Q_s)_s$ and some $\beta>0$; and 
\begin{equation}\label{eq-2-vertical}
\Big\| \Big(\fint_{|x-y|<t} |\rtt \Q_s^2 f(y)|^2\, dy\Big)^{1/2} \Big\|_\ltnu \lesssim_{[\nu^M]_{A_1}} \| \Q_s f\|_\ltnu.
\end{equation}
We remark that in the first quasi-orthogonality estimate \eqref{eq-1-vertical}, even though the estimate itself is unweighted, $\rtt$ still has a dependence on $\nu$. The uniform $L^2(\nu)$ estimate is handled the same way it was done for the conical; setting $h:=\Q_s^2 f$ we see 
\begin{multline}\nonumber 
\Big(\int_\rn \fint_{|x-y|<t} |\rtt h (y)|^2\, dy \nu(x) dx \Big)^{1/2}  \lesssim \Big(\sum_{Q\in \DDt} \int_Q \int_{|x-y|<t} |\rtt h(y)|^2\, dy\, \nu(x)dx \Big)^{1/2}\\
  \lesssim \Big(\sum_{Q\in \DDt}\fint_{Q^*} \int_Q \dfrac{\nu(y)}{\nuyt}|R_t h(y)|^2 \, dy \, \nu(x)dx\Big)^{1/2}. 
\end{multline}
By H\"older's Inequality with exponent $q/2>1$ and $M>(2/q)'$, we see 
\begin{equation}\nonumber 
\int_Q \dfrac{\nu(y)}{\nuyt}|R_t h(y)|^2\, dy \, \nu(x)dx   \leq \Big(\int_Q \Big| \dfrac{\nu(y)}{\nuyt}\Big|^{(\frac{q}2)'}\, dy \Big)^{1/(\frac{q}2)'}\Big( \int_Q |R_t h(y)|^{q}\, dy\Big)^{\frac2q}\\
  \lesssim_{[\nu^M]_{A_1}} |Q|^{1-\frac2q}\Big( \int_Q |R_t h(y)|^{q}\, dy \Big)^{\frac2q}. 
\end{equation}
Plugging this into the first estimate, we can now proceed as in the conical case (see Lemma \ref{lem-extrapolation-conical}), exploiting the $L^r-L^{q}$ off-diagonal decay in place of the $L^r-L^2$.

For the quasi-orthogonality estimate we proceed as follows: We exploit the off-diagonal decay that $\rtt$ inherits from $R_t$. More explicitly we have, for fixed $t,s>0$ using Fubini's Theorem and duality 
\begin{multline}\nonumber 
\int_\rn \fint_{|x-y|<t} |\rtt \Q_s^2 f(x)|^2\, dx \, dy  = \int_\rn \fint_{|x-y|<t} \vwt |R_t\Q_s^2 f(x)|^2\, dx \, dy \\
 = \int_\rn \vwt |R_t\Q_s^2 f(x)|^2\, dx   = \int_\rn  \vwt R_t\Q_s^2 f(x) \cdot \overline{R_t\Q_s^2f(x)} \, dx \, dy\\
  = \int_\rn R_t^*\Big( \vwt R_t\Q_s^2 f\Big)(x)\cdot \overline{\Q_s^2 f(x)}\, dx 
  \leq \| R_t^*( (\nu(x)/\nuxt R_tQ_s^2 f)\|_\ltrn \| Q_s^2 f\|_\ltrn, 
\end{multline}
where $R_t^*$ is the adjoint of $R_t$, for fixed $t>0$, in $\ltrn$. Since $\| \Q_s^2 f\|_\ltrn\lesssim \| \Q_s f\|_\ltrn$, we have reduced matters to showing 
\begin{equation}\label{eq-vertical-unweighted-quasi-orthogonality}
\int_\rn \Big| R_t^*\Big( \vwt R_t\Q_s^2 f\Big)(x)\Big|^2\, dx \lesssim \min\Big( \dfrac{s}{t}, \dfrac{t}{s}\Big)^\alpha \int_\rn |\Q_s f(x)|^2\, dx,
\end{equation}
for some $\alpha>0$. To save space we denote by $I$ the left-hand-side of this last inequality. Recall that we denote by $\DDt$ the collection of dyadic cubes of scale $2^{-k}$ where $t/2< 2^{-k}\leq t$. We compute, denoting by $Q^*= c_nQ$ for any cube $Q\subset \rn$ where $c_n$ is a dimensional constant,
\begin{multline}\label{eq-vertical-rt-duality-estimate} 
I^{1/2}  = \Big( \sum_{Q\in \DDt} \int_Q \Big| R_t^*\Big( \vwt R_t\Q_s f\Big)\Big|^2\, dx \Big)^{1/2} =\Big(\sum_{Q\in \DDt} \int_Q \fint_{|x-y|<t} \Big| R_t^*\Big( \vwt R_t\Q_s f\Big)\Big|^2\, dy\, dx  \Big)^{1/2}\\
  \lesssim \Big( \sum_{Q\in \DDt} \fint_{Q^*} \int_Q  \Big| R_t^*\Big( \vwt R_t\Q_s f\Big)\Big|^2\, dx\, dy \Big)^{1/2}  \lesssim \sum_{j\geq 0} \Big( \sum_{Q\in \DDt} \fint_{Q^*} \int_Q  \Big| R_t^*\Big(\bbm1_{R_j(Q)}\vwt R_t\Q_s f\Big)\Big|^2\, dx\, dy \Big)^{1/2}, 
\end{multline} 
where we define $R_0(Q):=2Q$ and for $j\geq 1$, $R_j(Q):=2^{j+1}Q\backslash 2^jQ$ and we used the triangle inequality in the last line, together with the $\ltrn$-boundedness of $R_t^*$. We now use the off-diagonal decay for $R_t^*$ to write, 
\begin{multline}\nonumber 
\int_Q \Big| R_t^*\Big(\bbm1_{R_j(Q)}\vwt R_t\Q_s f\Big)\Big|^2\, dx   \leq 2^{-nj\gamma}|Q|^{1-2/r}  \times \Big( \int_{R_j(Q)} \Big| \vwt R_t\Q_s^2 f(x)\Big|^r\, dx \Big)^{2/r}\\
  \approx 2^{-nj(\gamma-\frac2r)}|Q| \Big( \fint_{R_j(Q)} \Big| \vwt R_t\Q_s^2 f(x)\Big|^r\, dx \Big)^{\frac2r}  \leq 2^{-nj(\gamma-\frac2r)} |Q| \Big( \fint_{R_j(Q)} |R_t\Q_s^2 f|^2\Big)  \times\Big( \fint_{R_j(Q)} \Big| \vwt\Big|^{\widetilde{r}}\, dx \Big)^{\frac2{\widetilde{r}}},
\end{multline}
where $\widetilde{r}^{-1}= 1/r-1/2$ by H\"older's inequality. Plugging this estimate into \eqref{eq-vertical-rt-duality-estimate}, we see that 
\begin{equation}\label{eq-vertical-rt}
I^{1/2}\lesssim \sum_{j\geq 0} \Big(\sum_{Q\in \DDt} \fint_{Q^*}C^{2}_j|Q| \Big( \fint_{R_j(Q)} |R_t\Q_s^2 f(x)|^2\,dx \Big)\Big( \fint_{R_j(Q)} \Big| \vwt\Big|^{\widetilde{r}}\, dx \Big)^{2/\widetilde{r}}\, dy\Big)^{1/2},
\end{equation}
where we have defined $C_j:= 2^{-nj(\gamma-2/r)}$. 
Since $M>\widetilde{r}$, so that $\nu\in RH_{\widetilde{r}}$ (see Proposition \ref{prop-properties-ap-weights}). Moreover, using the doubling property of $\nu$ and denoting $C_{doub}$ to be the doubling constant, we have
\begin{equation}\label{eq-vertical-doubling-weight}
\nuxt \geq 2^j C_{doub}^{-j} \nu_{x,2^jt} \approx C_{doub}^{-j}2^j \nu(Q'_j),
\end{equation}
where $Q'$ is any cube with $\ell(Q'_j)\approx 2^jt$ containing $x$. Therefore decomposing $R_j(Q)$ into $N=N(n)$ cubes $Q'$ of sidelength $2^jt$ we compute 
\begin{multline}\nonumber 
\Big( \fint_{R_j(Q)} \Big| \vwt\Big|^{\widetilde{r}}\, dx \Big)^{1/\widetilde{r}}  \lesssim \Big(\sum_{Q'} \fint_{Q'} \Big| \vwt \Big|^{\widetilde{r}}\, dx \Big)^{1/\widetilde{r}}
  \lesssim_{[\nu^M]_{A_2}} \Big( \sum_{Q'} \Big(\fint_{Q'} \vwt \, dx\Big)^{\widetilde{r}} \Big)^{1/\widetilde{r}}\\
  \lesssim C^j_{doub}2^{-j}\Big( \sum_{Q'} \Big( \dfrac{\nu(Q')|Q'|}{\nu(Q')|Q'|}\Big)^{\widetilde{r}}\Big)^{1/\widetilde{r}} 
  \lesssim C^j_{doub}2^{-j}, 
\end{multline}
where we used \eqref{eq-vertical-doubling-weight} in the second to last line. In what follows, we absorb this constant into $C_j$, now writing $\widetilde C_j:= 2^{-nj(\gamma+1/n-2/r-\log_2(C_{doub})/n)}$. Plugging this into the estimate for $I$, appearing in \eqref{eq-vertical-rt}, and using Fubini's Theorem, we have that
\begin{multline}\nonumber 
I^{1/2}  \lesssim \sum_{j\geq 0}\Big( \sum_{Q\in \DDt} \fint_{Q^*} \widetilde C_j^{2}|Q| \fint_{R_j(Q)} |R_t\Q_s^2 f(x)|^2\, dx\, dy \Big)^{\frac12}  \approx \sum_{j\geq 0}\Big( \sum_{Q\in \DDt} \int_{Q^*} \widetilde C_j^{2} \fint_{|x-y|<2^{k+1}t} |R_t\Q_s^2 f(x)|^2\, dx\, dy \Big)^{\frac12}\\
  \approx \sum_{j\geq0} \Big(  \widetilde C_j^{2} \int_\rn \fint_{|x-y|<2^{(k+1)}t} |R_t\Q_s^2 f(x)|^2 \, dx \, dy \Big)^{\frac12}
  \lesssim \Big(\int_\rn |R_t\Q_s^2 f(x)|^2\, dx \Big)^{1/2} \Big( \sum_{j\geq0} \widetilde C_j^{2}\Big)^{\frac12} \\ \lesssim \Big(\int_\rn |R_t\Q_s^2 f(x)|^2\, dx \Big)^{\frac12},  
\end{multline}
where in the last step we used that $C_{doub}\lesssim_n [\nu]_{A_1}\leq C_\delta$.
This gives the desired estimate \eqref{eq-vertical-unweighted-quasi-orthogonality}, since we have good quasi-orthogonality estimates   (see the proof of Theorem  \ref{thm-conical-square-function-grad-st-part1}).   
\end{proof}

\section{Extrapolation of Square Function Estimates}

In this section, we obtain weighted and $L^p$ estimates for operators of the form $t^m\pd_t^m \nb(\sl_t\nb)$, for some $m\in \NN$ large. The main ingredients for these estimates are the $L^r-L^q$ off-diagonal diagonal decay estimates for our operators (see Propositions \ref{prop-vertical-off-diagonal-decay} and \ref{prop-off-diagonal-decay-for-operators-with-B}) for $r<2<q$, used implicitly through the extrapolation results of the previous sections.

At this stage we also mention the work \cite{pa19}, where the vertical and conical square functions for objects associated to the heat and Poisson semigroups of $\L$ (without lower order terms) are considered. We remark that our objects are a bit more technically involved to handle, in part due to the mild off-diagonal decay that they enjoy. Nevertheless, the basic idea of extrapolation and control of the vertical square function by a conical square function is the same.

In order to simplify the statement of our results, we make use of the following definition which encapsulates the assumptions that $\L$ must  satisfy.

\begin{definition}[Hypothesis A]\label{def-hypothesis-a}
We say that the operator $\L $ \emph{satisfies hypothesis A} if the following  hold.
\begin{enumerate}
	\item $\L$ has the form $\L= -\div(A\nabla + B_1)+ B_2\cdot\nabla$, 	for some $B_i\in L^n(\rn;\CC^{n+1})$ and a complex elliptic, $t$-independent, matrix $A$, that is, for some $\lambda>0$, a.e. $x\in \rn$ and every $\xi,\zeta\in \CC^{n+1}$ it holds that
	\begin{equation}\nonumber
	\lambda |\xi|^2\leq \textup{Re} (A(x)\xi \cdot \overline{\xi}), \qquad |A(x)\xi\cdot \overline{\zeta}|\leq\frac1{\lambda} |\xi||\zeta|.
	\end{equation}
	\item With $\widetilde{\rho}_1>0$ as in Theorem \ref{BHLMP1-Main.thrm}, we have 
	\begin{equation}\nonumber
	\max\{ \| B_1\|_\lnrn, \| B_2\|_\lnrn\} <\widetilde{\rho}_1.
	\end{equation}
\end{enumerate}
We will say a quantity depends on ellipticity if it depends only on $\lambda$ and $\widetilde{\rho}_1$.
\end{definition}

Next, we state the main result of this section. 

\begin{theorem}[$L^p$ extrapolation of square function estimates]\label{thm-general-extrapolation-thm}
Suppose that $\L$ satisfies Hypothesis A (see Definition \ref{def-hypothesis-a}), and let $\thtm$ be any  one of the   operators
\begin{equation}\nonumber
	t^m\pd_t^{m-1}\nb(\sl_t\nb), \quad t^m\pd_t^{m-1}\nb(\sl_t B_i), \quad t^m\pd_t^{m-1} B_i(\sl_t\nb), \qquad i=1,2.
\end{equation}
Then there exist $\varepsilon_0>0$, $m_0\in \NN$, and $\rho_0>0$ depending on dimension and ellipticity, such that for every $m\geq m_0$ and $p\in (2-\varepsilon_0, 2+\varepsilon_0)$, we have the estimate 
\begin{equation}\nonumber
\| \s(\thtm f)\|_\lprn  + \| \v(\thtm f)\|_\lprn \lesssim_p \| f\|_\lprn,
\end{equation}
provided that  $\max\{ \| B_1\|_\lnrn, \| B_2\|_\lnrn \} <\rho_0$. 
\end{theorem}

Let us give a quick roadmap to the location of the proofs of the various estimates summarized in the previous theorem. 

\begin{itemize}
\item The conical square function estimate  for $t^m\pd_t^{m-1}\nb (\sl_t\nb)$ is obtained in Theorem \ref{thm-conical-square-function-bounds-gradient-st-gradient}, while the vertical square function estimate is given  in Theorem \ref{thm-weithed-vertical-bounds-gradient-st-gradient}.
\item The estimates for $t^m\pd_t^{m-1}\nb(\sl_t B)$ are contained in Corollary \ref{cor-square-function-estimates-b-st-grad-and-dual}. 
\item The results for $t^m\pd_t^{m-1} B(\sl_t\nb)$ are obtained in Lemma \ref{lem-square-functions-with-b-outside}. There the results are obtained for the operator with the gradient replaced by a $t$ derivative. A careful inspection of the proof though shows that, as long as we have good estimates for the operator $t^m\pd_t^{m-1}\nb(\sl_t\nb)$, the same argument applies.
\item Estimates for the double layer potential are obtained in Theorem \ref{thm-sqaure-fctn-bounds-double-layer}.
\end{itemize}

\subsection{Estimates for $\nb\sl_t$}

In this subsection we prove the relevant estimates for operators of the form $t^m\pd_t^m \nb\sl_t$. These will follow immediately from the extrapolation results from the previous section; together with the off-diagonal estimates obtained in Propositions \ref{prop-vertical-off-diagonal-decay} and \ref{prop-off-diagonal-decay-for-operators-with-B}.

\begin{remark}
We would like to be able to apply Lemmas \ref{lem-extrapolation-conical} and \ref{lem-extrapolation-vertical} to $\thtm= t^m\pd_t^m (\sl_t \nbp)$ to handle the double layer potential; it is not a simple matter however to obtain the necessary quasi-orthogonality condition in those (one reason is that in the regime $s<t$ we need to ``add'' derivatives to $\thtm$, while taking them away from $\Q_s$; however adding derivatives to $\thtm$ is tricky since we already have a $\nbp$ in front. We will have to use the equation to circumvent this issue). We will treat this operator separately, in Section \ref{subsecdual}.
\end{remark}

\begin{theorem}\label{thm-conical-square-function-grad-st-part1}
Suppose that $\L$ satisfies Hypothesis A (see Definition \ref{def-hypothesis-a}). Let $\thtm= t^m\partial_t^m \nabla \sl_t$, then there exist  $\varepsilon_1>0$ and $m_1\in \NN$, depending on dimension and ellipticity, such that if $m\geq m_1$ and $2-\varepsilon_1<p<\infty$ then $\|\s(\thtm f)\|_\lprn \lesssim_{p,m} \| f\|_\lprn$.
\end{theorem}

\begin{proof}
This follows immediately from Lemma \ref{lem-extrapolation-conical}. The off-diagonal decay is contained in Proposition \ref{prop-vertical-off-diagonal-decay}, while the quasi-orthogonality estimate \eqref{eq-1-quasi-orthogonality-assumption} for $R_t$ is obtained in the proof of the $L^2$ square function bound for $\thtm$ (see \cite[Theorem 5.1]{bhlmp}); we sketch it here for completeness. Fix $0<s<t$ and we choose $\Q_s = s\div_{\|} \Qt_s$, so that 
\begin{equation}\nonumber
\thtm \Q_s h(x)= t^m\pd_t^m \nabla \sl_t(s\divp \Qt_s h)(x) = \dfrac{s}{t}t^{m+1}\pd_t^m \nabla(\sl_t \nbp)(\Qt_s h), 
\end{equation}  
and we appeal to \cite[Lemma 6.2]{bhlmp}, which shows that the operator $t^{m+1}\pd_t^m\nabla (\sl_t\nbp)$ is uniformly bounded in $\ltrn$, moreover so is $\Qt_s$. This takes care of the contribution of $\thtm$ to $R_t$. To handle the other term we further choose $P_t= \Pt_t\Pt_t$ for an approximate identity $\Pt_t$ and note that $|\thtm 1(x)|\Pt_t$ is uniformly bounded in $\ltrn$ while $\Pt_t\Q_s$ satisfies good quasi-orthogonality estimates when $s<t$. Finally, the $L^2$ square function bound is obtained in \cite[Theorem 5.1, Lemma 5.2]{bhlmp}.
\end{proof}

We now turn to the appropriate vertical square function bounds.

\begin{theorem}[$L^p$ Bounds for Vertical Square Function]
Suppose that $\L$ satisfies Hypothesis A (see Definition \ref{def-hypothesis-a}). Let $\thtm:=t^m\pd_t^m \nabla \sl_t$. There exists $\varepsilon_2>0$ and $m_2\in \NN$, depending on dimension and ellipticity, such that if $p\in (2-\varepsilon_2, 2+\varepsilon_2)$ and $m\geq m_2$ then  $\| \v(\thtm f )\|_\lprn \lesssim \| f\|_\lprn$. 
\end{theorem}

\begin{proof}
We use Remark \ref{weakeningRH}, with  $T_t =t^m\pd_t^m \nb \sl_t$ and  $S_t=t^{m-1}\pd_t^m\sl_t$.
Then the square function bound for $T_t$ follow from \cite[Theorem 5.1, Lemma 5.2]{bhlmp}, while the square function bound $S_t$ follows from \cite[Theorem 5.1]{bhlmp}. The comparability of $T$ and $S$, as in Remark \ref{weakeningRH}, follows from Caccioppoli's inequality (Proposition \ref{prop-caccioppoli-inequality-lp}), and the Reverse H\"older inequality for $S_t$ is contained in Proposition \ref{prop-prelim-inhomog-reverse-holder}, recalling that $S_tf(x)$ is a solution of $\L u=0$ in $\reu$ (see for instance \cite[Proposition 3.16]{bhlmp}). The necessary off-diagonal decay for both $S$ and $T$ is in Proposition \ref{prop-vertical-off-diagonal-decay}, choosing $m$ large enough. The conclusion now follows from Lemma \ref{lem-extrapolation-vertical}. 	
\end{proof}

While the extrapolation result in Lemma \ref{lem-extrapolation-vertical} is interesting on its own, it turns out that in our context, exploiting Caccioppoli's inequality, it is easy to get a much stronger bound (in fact the moral of the proof seems to be that, if $T_t$ enjoys a reverse H\"older inequality on slices, then we can always control the vertical square function by the conical in an interval {\it around} $p=2$). We state this in the following

\begin{theorem}[Weighted Bounds for Vertical Square Function]\label{thm-weighted-bounds-vertical-square-function-nb-st}
Suppose that $\L$ satisfies Hypothesis A (see Definition \ref{def-hypothesis-a}). Let $\thtm:=t^m\pd_t^m\nb \sl_t$. There exist $m_2'\in \NN$ and $M_2\geq 1$, depending on dimension and ellipticity, such that for every $m\geq m_2'$, $M \ge M_2$ and every $\nu\in A_2$ with the property $\nu^{M}\in A_2$ it holds
\begin{equation}\nonumber
\| \v(\thtm)\|_{\ltnu} \approx \|\s(\tithm f)\|_\ltnu \lesssim_{[\nu^M]_{A_2}} \| f\|_\ltnu,
\end{equation}
where we define 
\begin{equation}\nonumber
\tithm f(x):= \sqrt{\dfrac{\nu(x)}{\fint_{|x-y|<t}\nu(y)dy}} \thtm f(x)= \sqrt{\dfrac{\nu(x)}{\nuxt}}\thtm f(x).
\end{equation}
\end{theorem}

\begin{proof}
We note, from the beginning of the proof of Lemma \ref{lem-extrapolation-vertical}, that the comparability $\| \v(\thtm f)(x)\|_\ltnu \approx \|\s(\tithm f)\|_\ltnu$ holds for any weight $\nu$. Therefore it remains to estimate the conical square function associated to $\tithm$. First, we write
\begin{equation}\nonumber
\begin{split}
\| \s(\tithm f)\|_\ltnu^2  = \int_\rn \int_0^\infty \fint_{|x-y|<t} \dfrac{\nu(y)}{\nuyt} |\thtm f(y)|^2\, dy \dfrac{dt}{t} \, \nu(x)dx,
\end{split}
\end{equation}
and then, by the H\"older and Caccioppoli inequalities, 
\begin{multline}\nonumber 
\fint_{|x-y|<t} \dfrac{\nu(y)}{\nuyt}|\thtm f(y)|^2\, dy  \leq \Big(\fint_{|x-y|<t} \Big| \dfrac{\nu(y)}{\nuyt}\Big|^{q'}\, dy \Big)^{1/q'} \Big(\fint_{|x-y|<t} |\thtm f(y)|^{2q}\, dy \Big)^{1/q}\\
  \lesssim_{[\nu^M]_{A_2}} \Big(\fint_{|x-y|<t} |\thtm f(y)|^{2q}\, dy\Big)^{1/q} 
  \lesssim \Big(\fint_{t/2}^{3t/2}\fint_{|x-y|<2t} |\thsmt f(y)|^{2q}\, dyds\Big)^{1/q} 
\end{multline}
where we have defined $\thsmt:= t^{m-1}\pd_t^m \sl_t$, and chosen $q \in (1,2)$ such that our operators satisfy a $2q$ Caccioppoli Inequality on slices (see Proposition \ref{prop-caccioppoli-on-slices}) and then chosen $M>q'$. Now since $\thsmt$ satisfies a reverse H\"older Inequality (see Proposition \ref{prop-prelim-inhomog-reverse-holder}) we see that 
\begin{equation}\nonumber 
\Big(\fint_{\frac t2}^{\frac{3t}2}\fint_{|x-y|<2t} |\thsmt f(y)|^{2q}\, dyds\Big)^{\frac1q}   \lesssim \fint_{\frac t4}^{\frac{7t}4} \fint_{|x-y|<3t} |\thsmt f(y)|^2\, dyds \\
  \lesssim \fint_{t/4}^{7t/4} \fint_{|x-y|<4s} |\thsmt f(y)|^2\, dyds. 
\end{equation}
The desired result follows now from Fubini's theorem and the fact that conical square functions with different cone apertures are comparable (see the comments after Definition \ref{def-square-functions}).
\end{proof}

In what follows we   need square function estimates for the operators $t^m\pd_t^m B\sl_t$, where $B\in \lnrn$ is independent of the transversal variable. The $L^2$ case follows from the bounds for $t^m\pd_t^m \nb \sl_t$ and Sobolev's inequality, the case $p\neq 2$ requires a more involved argument both for vertical and conical square functions.

\begin{lemma}\label{lem-square-functions-with-b-outside}
Suppose that $\L$ satisfies Hypothesis A (see Definition \ref{def-hypothesis-a}). For a function $B\in \lnrn$, independent of the $t$ variable and $m\in \NN$ consider the operators $\thtmb f(x):=t^m\pd_t^m B\sl_t f(x)$, $\thtm f(x):= t^m\pd_t^m \nbp \sl_t f(x)$. For every $1<p<n$ and $f\in C_c^\infty(\rn)$ it holds 
\begin{equation}\nonumber
\| \v(\thtmb f)\|_\lprn \lesssim \| \v(\thtm f)\|_\lprn.
\end{equation}
Moreover, if  $\lthtm f:=t^m\pd_t^{m+1} \sl_tf$, then for any $1<p<\infty$,
\begin{equation}\nonumber
\| \s(\thtmb f)\|_\lprn \lesssim \| \s(\thtm f)\|_\lprn + \| \s(\thtmt f)\|_\lprn \lesssim \|\s(\thtmt f)\|_\lprn.
\end{equation}
\end{lemma}

\begin{proof}
We begin with the bound for the conical versions. Note that the second inequality follows from the fact that conical square functions  always ``travel up" by the $L^2$ Caccioppoli inequality. To handle the first inequality we note that for fixed $x\in \rn$ and $t>0$ we have, by H\"older's and Poincar\'e-Sobolev Inequalities,
\begin{multline}\nonumber 
\Big( \fint_{|x-y|<t} |\thtmb f(y)|^2 \, dy\Big)^{1/2}   \lesssim \dfrac{\| B\|_{\lnrn}}{t} \Big(\fint_{|x-y|<t} |t^m\pd_t^m\sl_t f(y)|^{2^*}\, dy\Big)^{1/2^*} \\
  \lesssim \| B\|_\lnrn \Big[\Big( \fint_{|x-y|<t} |\thtm f(y)|^2\, dy \Big)^{1/2} + |(\thtmt f)_{x,t}| \Big],  
\end{multline} 
where $(\thtmt f)_{x,t}$ denotes the average of $\thtmt f$ on the $n$-ball $|x-y|<t$. The result now follows from Jensen's inequality and the definition of $\s$.

The vertical square function is a bit more involved. The idea is to write 
\begin{equation}\nonumber
\thtmb f(x)= B I_1R \nbp t^m\pd_t^m\sl_t f(x)=B I_1 R \thtm f(x),
\end{equation}
where $I_1$ is the fractional integral of order 1 and $R$ is a vector valued Riesz Transform (note that the above makes sense in $L^2(\rn)$ owing to the slices estimates of \cite[Theorem 1.4]{bhlmp} and the mapping properties of $I_1$ and $R$). Therefore, by H\"older's Inequality 
\begin{equation}\nonumber
\| \v(\thtmb f)\|_\lprn \leq \| B\|_\lnrn \| \v(I_1R \thtm f)\|_{L^{p^*}(\rn)},
\end{equation}
where $1/p^*=1/p-1/n$ is the Sobolev exponent in dimension $n$. The desired result follows from the following estimate: Let $F:\reu \to \CC$, then for every $1<p<n$,
\begin{equation}\nonumber
\| \v(I_1R F)\|_{L^{p^*}(\rn)} \lesssim \| \v(F)\|_\lprn.
\end{equation}
To show this, first note that for every $1<p<\infty$ we have 
\begin{equation}\nonumber
\| \v(R F)\|_\lprn \lesssim \|\v(F)\|_\lprn.
\end{equation} 
This is a consequence of the weighted estimate 
\begin{equation}\nonumber
\int_\rn \int_0^\infty |RF(x,t)|^2\, \dfrac{dt}{t} \nu(x)dx \lesssim \int_\rn \int_0^\infty |F(x,t)|^2\, \dfrac{dt}{t} \nu(x)dx, \quad \nu\in A_2,
\end{equation}
and the extrapolation theorem for $A_p$ weights (see Theorem \ref{thm-extrapolation-ap-weights}). Therefore it is enough to prove the estimate for $I_1$ alone. For this we will need an off-diagonal extrapolation result (see \cite[Theorem 3.23]{cump}) to reduce matters to proving 
\begin{equation}\nonumber
\Big(\int_\rn \int_0^\infty |I_1 F(x,t)|^2 \, \dfrac{dt}{t}\, \nu^2(x)dx\Big)^{1/2} \lesssim \Big( \int_\rn\Big( \int_0^\infty |I_1F(x,t)|^2\,\dfrac{dt}{t}\Big)^{2_*/2}\, \nu(x)^{2_*} dx \Big)^{1/2_*},
\end{equation} 
where $1/2_*=1/2+1/n$, and the above holds for every $\nu\in A_{2_*,2}$ (see Definition \ref{def-apq-classes}). To prove the above inequality we appeal to Theorem \ref{thm-mapping-properties-i1} to obtain, for a weight $\nu$ as above, $\| I_1 g\|_{L^2(\nu^2)} \lesssim \| g\|_{L^{2_*}(\nu^{2_*})}$. Therefore 
\begin{equation}\nonumber
\Big(\int_\rn \int_0^\infty |I_1 F(x,t)|^2 \, \dfrac{dt}{t}\, \nu^2(x)dx\Big)^{1/2} \lesssim \Big(\int_0^\infty\Big( \int_\rn| F(x,t)|^{2_*}\, \nu(x)^{2_*}dx\Big)^{2/2_*}  \dfrac{dt}{t}\Big)^{1/2}.
\end{equation}
The desired bound now follows from Minkowski's inequality (in $L^{2/2_*}$).
\end{proof}

\begin{remark}
More generally, the proof above gives weighted inequalities and, in fact, shows the following: Weighted bounds $T:\ltnu\to \ltnu$ imply that $\| \v(T F)\|_\ltnu\lesssim \| \v(F)\|_\ltnu$. The same is true for the conical square function if $T$ in addition has good local estimates, we refer to \cite{apa17}. 
\end{remark}

\subsection{Estimates for $(\sl_t\nb)$}\label{subsecdual}

We will need the analogue of Lemma \ref{lem-square-functions-with-b-outside} for the dual (in $\ltrn$) operator. 

\begin{lemma}\label{lem-square-fuctions-with-B-inside}
Suppose that $\L$ satisfies Hypothesis A (see Definition \ref{def-hypothesis-a}). Let $B\in L^n(\rn;\CC^n)$ and set $\thtm:= t^m\pd_t^m(\sl_t\nbp)$ and $\thtmb:=t^m\pd_t^m\sl_t B$. then for any weight $\nu\in A_2$ we have 
\begin{equation}\nonumber
\| \s(\thtmb f)\|_\ltnu \lesssim_{[\nu]_{A_2}} \| B\|_\lnrn \| \s(\thtm) \|_{\ltnu\to \ltnu}\| f\|_\ltnu.
\end{equation}
In fact the constant can be shown to be at most a dimensional constant times $[\nu]_{A_2}^{1+\alpha}$ for some $\alpha<1$ (see for instance \cite{p08} and \cite{lmpt10})
\end{lemma}

\begin{proof}
We begin by writing, for $f\in C_c^\infty(\rn;\CC^n)$, $B\cdot f= \divp I_1 I_1\nbp (B\cdot f) = \divp I_1 R(B\cdot f)$, where $I_1$ is the fractional integral of order 1. Therefore $\thtmb f = \thtm(I_1 R(B\cdot f))$, and so 
\begin{equation}\nonumber
\| \s(\thtmb f)\|_\ltnu \lesssim \| \s(\thtm)\|_{\ltnu\to\ltnu} \| RI_1 (B\cdot f)\|_\ltnu.
\end{equation}
Since $R:\ltnu\to\ltnu$ by the Coifman-Fefferman maximal inequality (see Proposition \ref{prop-properties-ap-weights}), the result follows from Proposition \ref{prop-i1b-bounded}. 	
\end{proof}

\begin{theorem}[Square Function bounds for $(\sl_t\nb)$]\label{thm-weighted-square-function-bounds-double-layer}
Suppose that $\L$ satisfies Hypothesis A (see Definition \ref{def-hypothesis-a}). Let $\thtm:= t^m\pd_t^m(\sl_t\nb)$ and $\thtmp:=t^m\pd_t^m(\sl_t\nbp)$ and $\delta\in (0,1)$. Then there exist $M>0$, $m_3\in \NN$ (depending only on dimension, ellipticity, and for $m_3$ also $\delta$)  such that for every $m\geq m_3$ and if $\nu^M\in A_2$ is such that $[\nu^M]_{A_2}\leq C_\delta$ (with $C_\delta$ as in 7 of Proposition \ref{prop-properties-ap-weights}) then 
\begin{equation}\nonumber
\| \s(\thtmp f )\|_\ltnu \lesssim_{C_\delta,m} \|f\|_\ltnu,
\end{equation}
provided that\footnote{This is one of the few places where we may require additional smallness in addition to that imposed in \cite{bhlmp}, prior to discussing existence and uniqueness of the solutions to the boundary value problems.}  $\| B_{2\|}\|_\lnrn \leq \rho_3$, for some $\rho_3$ depending only on dimension, ellipticity of $\Lp$\footnote{More precisely, the dependence on the ellipticity of $\L_{\|}$ is through the constants appearing in Theorem \ref{thm-weighted-hodge-decomposition}.} and $C_\delta$.

In particular, for $p\in (2-1/2M,2+1/2M)$, it holds
\begin{equation}\nonumber
\| \s(\thtm f)\|_\lprn \lesssim \| f\|_\lprn.
\end{equation}
\end{theorem}

\begin{proof}
We   follow the same outline as in the proof the corresponding unweighted $L^2$ bound for this object (see \cite[Lemma 5.26]{bhlmp}, which in turn is based on the method in \cite{Hof-May-Mour}). Throughout we  fix $\Q_s$ a CLP family (see Definition \ref{def-clp-family}) with smooth compactly supported kernel, and set $P_t:=-\int_t^\infty \Q_s^2 \, \dfrac{ds}{s}$. By the Hodge decomposition and the weighted estimates in Theorem \ref{thm-weighted-hodge-decomposition}, we see that it is enough to show  
\begin{equation}\nonumber
\| \s(\thtm \ap \nbp F)\|_\ltnu \lesssim_{[\nu^M]_{A_2}} \| \nbp F\|_\ltnu.
\end{equation}
We start by writing, via the Coifman-Meyer technique \cite{CM}, 
\begin{multline}\nonumber  
\thtmp \ap \nbp F(x)   = (\thtmp (\ap \nbp F)(x) - [\thtmp\ap(x)]\cdot P_t\nbp F(x)) + [\thtmp \ap(x)]\cdot P_t \nbp F(x)\\
  =: R_t(\nbp F)(x)+[\thtmp \ap(x)]\cdot P_t \nbp F(x).  
\end{multline}

Since these objects already satisfy good (unweighted) $L^2$ estimates, the difficulties now shift to the ``error" term $R_t$; indeed, using the weighted version of Carleson's lemma (see Lemma \ref{lem-prelim-weighted-carleson}) to handle the second term it is enough to show that  $\mu(R_Q)\lesssim \nu(Q)$, for each  $Q\subset \rn$, where we have defined the measure $\mu$ as (recall that we are trying to control a {\it conical} square function)
\begin{equation}\nonumber
d\mu(x,t):= \Big(\fint_{|x-y|<t} |\thtmp\ap(y)|^2\, dy \Big) \dfrac{\nu(x)dxdt}{t}.
\end{equation}
To obtain this estimate, owing to the off-diagonal decay of $\thtmp$ in Proposition \ref{prop-vertical-off-diagonal-decay} and the fact that $d\mu$ is a Carelson measure when $\nu=1$  \cite[Lemma 5.26]{bhlmp}, we  mimic the argument  in the proof of Theorem \ref{lem-extrapolation-conical} involving Lemma \ref{lem-prelim-john-nirenberg-local-square-functions}; we omit the details. 

It now remains to show that $R_t$ has good square function bounds. This is the main part of the proof; we will follow almost verbatim   \cite[Lemma 5.26]{bhlmp}, replacing weighted bounds where appropriate. We begin by rewriting $R_t$ in the following way:
\begin{equation}\nonumber 
R_t   = \thtmp \ap - [\thtmp\ap]P_t
 = \Big( \thtmp\ap P_t - [\thtmp\ap]P_t\Big) + \thtmp\ap(I-P_t)=: \reto +\rett.  
\end{equation}
Since $\thtmp$ has good off-diagonal decay by Proposition \ref{prop-vertical-off-diagonal-decay} (see also \cite[Lemma 3.3]{AAAHK}), so does $R_t$ and satisfies the quasi orthogonality estimate \eqref{eq-1-quasi-orthogonality-assumption}, thanks to the presence of the $P_t$ term. We can then apply the extrapolation lemma for conical square functions (Lemma \ref{lem-extrapolation-conical}) to conclude that $\|\s(\reto\nbp F)\|_\ltnu \lesssim \|\nbp F\|_\ltnu$, for each $\nu\in A_{2/r}$ and some $1<r<2$. For the term $\rett$ we will use the equation in the form of the identities on slices (see \cite[Proposition 3.19]{bhlmp}). For notational convenience,  we will denote $Z_t:=(1-P_t)$, and also $\vb:=(A_{n+1,1},\ldots,A_{n+1,n})$, $ \va:=(A_{1,n+1},\ldots, A_{n+1,n+1})$. We write 
\begin{multline}\nonumber 
\rett(\nbp F)   = \thtmp\ap Z_t \nbp F= \thtmp\ap \nbp Z_t F
  = \pd_t\thtm \va Z_t F - \lthtm(\vb\nb Z_t F) + \thtm B_1Z_tF\\
  - t\thtmt (B_{2\|} \nbp Z_t F)+ \lthtm(B_{2\perp}Z_t F) 
  =: J_1+J_2+J_3+J_4+J_5, 
\end{multline}
where as usual we have defined $\lthtm:=t^m\pd_t^{m+1}\sl_t$. To handle $J_1$ we note that, owing to the $L^r-L^2$ off-diagonal decay of $\thtm$ (Proposition \ref{prop-vertical-off-diagonal-decay}) and the average weighted estimates on slices of Proposition \ref{prop-averaged-bounds-slices-via-off-diagonal-decay}, 
\begin{multline}\nonumber 
\| \s(J_1)\|_\ltnu^2  = \int_0^\infty\int_\rn \fint_{|x-y|<t} |\thtmo \va t^{-1}Z_t F(y)|^2\, dy \nu(x)dx\, \dfrac{dt}{t}\\
  \lesssim_{[\nu]_{A_{2/r}}} \int_0^\infty \int_\rn |\va t^{-1}Z_t F(x)|^2\, \nu(x)dx \dfrac{dt}{t} 
  \lesssim \| \nbp F\|_\ltnu, 
\end{multline}
where we have used Proposition \ref{prop-square-function-bounds-i-pt} to handle the square function associated to $Z_t$ in the last line. 

For $J_2$ we rewrite as follows:
\begin{equation}\nonumber 
J_2   = \lthtm (\vb \cdot \nbp F) + \Big( \lthtm(\vb \cdot P_t\nbp F)- [\lthtm \vb]\cdot P_t\nbp F\Big) + [\lthtm \vb]\cdot P_t \nbp F\\
 =: J_{2,1}+J_{2,2}+J_{2,3}. 
\end{equation}
Again appealing to  Lemma \ref{lem-prelim-john-nirenberg-local-square-functions} (see also the proof of Lemma \ref{lem-extrapolation-conical}) we see that the contribution of $J_{2,3}$ is under control by the weighted version of Carleson's lemma (Lemma \ref{lem-prelim-weighted-carleson}). The term $J_{2,2}$ we can handle the same way we did $\reto$; we omit the details. Finally, by Theorem \ref{thm-weighted-bounds-vertical-square-function-nb-st}, we have good weighted conical square function bounds for $\lthtm$ and $\vb\in L^\infty(\rn)$, so the contribution of  $J_{2,1}$ is also under control.

For $J_3$ we appeal to Proposition \ref{prop-control-for-J3}, which, for $s < t$, gives the bound (for $I_1g=F$) 
\begin{equation}\nonumber
\Big\| \Big(\fint_{|x-y|<t}|\thtm B_1I_1\Q_s^2 g(y)|^2\, dy \Big)^{1/2}\Big\|_\ltnu \lesssim \Big( \dfrac{s}{t}\Big)^\beta \| \Q_s g\|_\ltnu.
\end{equation}
Therefore 
\begin{multline}\nonumber 
\| \s(J_3)\|_\ltnu   = \int_0^\infty \int_\rn \fint_{|x-y|<t} |\thtm B_1(1-P_t)F(y)|^2\, dy \, \nu(x)dx \, \dfrac{dt}{t}\\
 = \int_0^\infty \int_\rn \fint_{|x-y<t}\Big|\thtm B_1 I_1 \int_0^t \Q_s^2 g(y) \, \dfrac{ds}{s} \Big|^2\ dy \, \nu(x)dx \, \dfrac{dt}{t}\\
  \lesssim_\beta \int_0^\infty \int_\rn \fint_{|x-y|<t}\int_0^t \Big( \dfrac{t}{s}\Big)^{\beta/2} |\thtm B_1 I_1\Q_s^2 g(y)|^2\, \dfrac{ds}{s}\, dy \, \nu(x)dx \, \dfrac{dt}{t}\\
   \lesssim \int_0^\infty \int_s^\infty \Big(\dfrac{s}{t}\Big)^{\beta/2} \| \Q_s g\|_\ltnu^2 \, \dfrac{st}{t}\, \dfrac{ds}{s} 
  \lesssim \| \v(\Q_s g)\|_\ltnu 
  \lesssim \| g\|_\ltnu, 
\end{multline}
where we invoked Theorem \ref{thm-weighted-lp} in the last line. To conclude we note that $I_1 g=F$ and so $R g= \nbp F$, where $R$ is the vector-valued Riesz transform (with symbol $\xi/|\xi|$) and we know $\| Rg\|_\ltnu \approx \| g\|_\ltnu$ for every $\nu\in A_2$; the desired bound follows from this since $A_{2/r}\subset A_2$.

To handle $J_4$ we write it as
\begin{equation}\nonumber
J_4= -t\thtmt B_{2\|}\nbp F + t\thtmt B_{2\|}\nbp P_t F=: J_{4,1}+ J_{4,2}.
\end{equation}
For $J_{4,1}$ we appeal to Lemma \ref{lem-square-fuctions-with-B-inside} to bound
\begin{equation}\nonumber
\| \s(J_{4,1})\|_\ltnu \lesssim_{[\nu]_{A_2}} \| B_{2\|}\|_\lnrn \|\nbp F\|_\ltnu \| \thtmp \|_{\ltnu\to\ltnu}.
\end{equation}
Therefore, if $\| B_2\|_\lnrn$ is small enough we may hide this term on the left hand side.

We rewrite $J_{4,2}$ in the following way: 
\begin{equation}\nonumber 
J_{4,2}   = \Big(t\thtmt(B_{2\|} \cdot P_t\nbp F)- [t\thtmt B_{2\|}] \cdot P_t \nbp F\Big) + [t\thtmt B_{2\|}] \cdot P_t \nbp F
  =: R_t^{[3]} + [t\thtmt B_{2\|}]\cdot P_t\nbp F. 
\end{equation}
The term $R_t^{[3]}$ may be handled the same way as $R_t^{[1]}$, using Proposition \ref{prop-off-diagonal-decay-for-operators-with-B} to obtain the right $L^r-L^2$ off-diagonal estimates. It remains to show, by an application of the weighted version of Carelson's lemma, a $\nu$-Carleson measure estimate for 
\begin{equation}\nonumber
d\mu(x,t):= \Big(\fint_{|x-y|<t}|t^m\pd_t^m\sl_t B_{2\|}(y)|^2\, dy \Big) \, \nu(x)\dfrac{dxdt}{t}.
\end{equation}
This follows, once again, by an application of Lemma \ref{lem-prelim-john-nirenberg-local-square-functions} (see also the proof of Lemma \ref{lem-extrapolation-conical})\footnote{Notice that, since we already have good unweighted $L^2$ square function estimates for $\thtm$, the John-Nirenberg lemma gives us that this object is under control; as opposed to the unweighted case, where we were forced to hide this term.}.

Finally, to handle $J_5$, we appeal again to the $L^r-L^2$ off-diagonal estimates of $t^{m+1}\pd_t^{m+1}\sl_t B_{2\perp}$ (Proposition \ref{prop-off-diagonal-decay-for-operators-with-B}) and Proposition \ref{prop-averaged-bounds-slices-via-off-diagonal-decay} (which give that $t^{m+1}\pd_t^{m+1}\sl_t B_{2\perp}$ satisfies good averaged weighted bounds on slices) to conclude 
\begin{equation}\nonumber
\| \s(J_5)\|_\ltnu \lesssim \int_0^\infty \int_\rn \Big|\dfrac{Z_t}{t} F\Big|^2\, \nu(x)\dfrac{dt}{t}.
\end{equation}
We conclude by the square function estimates of Proposition \ref{prop-square-function-bounds-i-pt}, the same as we did for $J_1$.

Combining all the above, we see that we have shown:
\begin{equation}\nonumber
\| \s(\thtmp)\|_{\ltnu\to \ltnu}\lesssim_{[\nu^M]_{A_2}} 1 + \| B_{2\|}\|_\lnrn \|\s(\thtmp)\|_{\ltnu\to\ltnu}.
\end{equation}
This gives the desired bound if the left hand side is finite and $\| B\|_\lnrn$ is small enough. To achieve the former we may work with the truncated square functions given by 
\begin{equation}\nonumber
\s_\eta(\thtmp f)(x):= \int_{\eta}^{1/\eta}\fint_{|x-y|<t} |\thtmp f(y)|^2\, dy \dfrac{dt}{t},
\end{equation}
which satisfy $\|\s_\eta(\thtmp)\|_{\ltnu\to\ltnu}<\infty$, owing to the estimates on slices of Proposition \ref{prop-averaged-bounds-slices-via-off-diagonal-decay}. Fix now $C_\delta$ as in the assumptions, i.e. $[\nu^M]\leq C_\delta$, then our estimates read (see also Theorem \ref{thm-weighted-hodge-decomposition})
\begin{equation}\nonumber
\| \s(\thtmp)\|_{\ltnu\to\ltnu} \lesssim_{C_0} 1+ \| B_{2\|} \|_\lnrn \|\s(\thtmp)\|_{\ltnu\to\ltnu}.
\end{equation}
Thus, choosing $\| B_{2\|}\|_\lnrn<\rho=\rho(C_\delta)$ we can hide the second term on the right hand side and conclude the result.		 The $L^p$ bounds are a consequence of this and Corollary \ref{cor-unweighted-lp-extrapolation} if we choose $\delta=1/2$, where recall $C_{1/2}$ is defined as in 7 of Proposition \ref{prop-properties-ap-weights}.
\end{proof}

\begin{theorem}\label{thm-conical-square-function-bounds-gradient-st-gradient}
Suppose that $\L$ satisfies Hypothesis A (see Definition \ref{def-hypothesis-a}). Let $\thtm:=t^{m}\pd_t^{m-1}\nb (\sl_t\nb)$ and $\delta\in (0,1)$. There exist $m_4\in \NN$ and $M>0$ (depending only on dimension and ellipticity, and for $m_4$ also on $\delta$) such that if $m\geq m_4$ and $\nu\in A_2$ is such that $[\nu^M]_{A_2}\leq C_\delta$, then 
\begin{equation}\nonumber
\| \s(\thtm f)\|_\ltnu \lesssim_{[\nu^M]_{A_2}} \| f\|_\ltnu,
\end{equation}
provided $\| B_{2\|}\|_\lnrn<\rho_4$, for some $\rho_4$ depending on dimension, ellipticity of $\Lp$, and $C_\delta$ only. 

In particular there exists $\varepsilon_4>0$ (depending on dimension and ellipticity) such that if $p\in (2-\varepsilon_4, 2+\varepsilon_4)$ and $m\geq m_4$ then 
\begin{equation}\nonumber
\| \s(\thtm f)\|_\lprn \lesssim_{p} \| f\|_\lprn.
\end{equation} 
\end{theorem}	

\begin{proof}
Notice that it is enough to consider $\nbp$ instead of $\nb$ in both instances; otherwise we are in the situation of Theorem \ref{thm-conical-square-function-grad-st-part1} or Theorem \ref{thm-weighted-square-function-bounds-double-layer}. Therefore, without loss of generality, $\thtm=t^m\pd_t^{m-1}\nbp(\sl_t\nbp)$. In this case, for $f\in C_c^\infty(\rn;\CC^n)$ we can write 
\begin{equation}\nonumber
\thtm f(x)= t^m\pd_t^{m-1}\nbp\sl_t(\divp f)(x)=: t^m\pd_t^{m-1}\nbp \sl_t g(x).
\end{equation}
By the Caccioppoli inequality on slices (Proposition \ref{prop-caccioppoli-on-slices}) we see that, for fixed $x\in \rn$ and $t>0$, 
\begin{multline}\nonumber 
\int_0^\infty \fint_{|x-y|<t} |t^m\pd_t^{m-1}\nbp \sl_t f(y)|^2\, dy \dfrac{dt}{t}   \lesssim \int_0^\infty \fint_{t/4}^{5t/4} \fint_{|x-y|<2t}|t^{m-1}\pd_t^{m-1}\sl_t g(y)|^2\, dy \dfrac{dt}{t}\\
  = \int_0^\infty \fint_{t/4}^{5t/4} \fint_{|x-y|<2t} |s^{m-1}\pd_s^{m-1}(\sl_s\nbp) f(y)|^2\, dy ds \dfrac{dt}{t}
  \lesssim \int_0^\infty \fint_{|x-y|<2t} |s^{m-1}\pd_s^{m-1}(\sl_s\nbp) f(y)|^2\, dy \dfrac{ds}{s}. 
\end{multline}
The result now follows from Theorem \ref{thm-weighted-square-function-bounds-double-layer} and the change of angle formula for square functions (see the comments after Definition \ref{def-square-functions}).
\end{proof}

The following is the weighted version of Theorem \ref{thm-general-extrapolation-thm} for the vertical square function.

\begin{theorem}\label{thm-weithed-vertical-bounds-gradient-st-gradient}
Suppose that $\L$ satisfies Hypothesis A (see Definition \ref{def-hypothesis-a}). Let $\thtm:= t^m\pd_t^{m-1}\nb(\sl_t\nb)$ and $\delta\in (0,1)$. There exists $m_5\in \NN$ and $M>0$ (depending only on dimension and ellipticity, and in the case of $m_5$ also on $\delta$) such that if $m\geq m_5$ and $\nu\in A_2$ satisfies $[\nu^M]_{A_2}\leq C_\delta$, then
\begin{equation}\nonumber
\| \v(\thtm f)\|_\ltnu \lesssim_{C_0} \| f\|_\ltnu,
\end{equation}
provided $\|B_{2\|}\|<\rho_5$, for some $\rho_5>0$ depending only on dimension, ellipticity of $\Lp$, and $C_\delta$.

In particular there exists $\varepsilon_5>0$ (depending on dimension and ellipticity) such that for $p\in (2-\varepsilon_5,2+\varepsilon_5)$
\begin{equation}\nonumber
\| \v(\thtm f)\|_\lprn\lesssim \| f\|_\lprn.
\end{equation}
\end{theorem}

\begin{proof}
By Theorem \ref{thm-weighted-bounds-vertical-square-function-nb-st}  and $t$-independence, it is enough to consider the operator $\thtm:= t^m\pd_t^{m-1} \nb(\sl_t\nbp)$. Now the idea is to repeat the proof of the weighted bound for $\v(t^m\pd_t^m\nb\sl_t g)$, with $g=\divp f$, in Theorem \ref{thm-weighted-bounds-vertical-square-function-nb-st}, using now instead Theorem \ref{thm-conical-square-function-bounds-gradient-st-gradient}; we omit the details.
\end{proof}

The next corollary is useful when dealing with square functions involving the double layer potential. 

\begin{corollary}\label{cor-square-function-estimates-b-st-grad-and-dual}
Suppose  that $\L$ satisfies Hypothesis A (see Definition \ref{def-hypothesis-a}). Let $\varepsilon_0>0$ and $M_0>0$ be as in Theorem \ref{thm-general-extrapolation-thm}. Let $\thtmb:= t^m\pd_t^{m-1}\nb(\sl_tB)$ for some $B\in L^n(\rn;\CC^{n+1})$. Then for every $f\in C_c^\infty(\rn;\CC^{n+1})$
\begin{equation}\nonumber
\| \s(\thtmb f)\|_\ltnu, \| \v(\thtmb f)\|_\ltnu \lesssim_{[\nu^m]_{A_2}} \| f\|_\ltnu.
\end{equation}
\end{corollary}

\begin{proof}
Write $B_{\|}\cdot f_{\|}= \divp \nbp I_1 I_1(B_{\|}\cdot f_{\|})= \divp R I_1(B_{\|}\cdot f_{\|}) = \divp g_{\|}$. Notice that, by the proof of Proposition \ref{prop-i1b-bounded}, we know $\| g_{\|}\|_\lprn \lesssim \| f_{\|}\|_\lprn$ for every $1<p<\infty$. On the other hand we can also write $B_\perp f_\perp= \divp RI_1(B_\perp f_\perp) =\divp g_\perp$ where  $\| g_\perp\|_\lprn\lesssim  \| f_\perp\|_\lprn$. Combining these two observations the result follows from  either Theorem \ref{thm-conical-square-function-bounds-gradient-st-gradient} for the conical version or Theorem \ref{thm-weithed-vertical-bounds-gradient-st-gradient} for the vertical.
\end{proof}

Recall that for $\vec{N}=-e_{n+1}$ the exterior normal to $\pd\reu$ we have the  representation formula for the double layer $\dl_t f= (\sl_t\nb)\cdot A\vec{N} f + (\sl_t\overline{B_2})\cdot \vec{N}f$, for $f\in C_c^\infty(\rn)$. As an immediate consequence of this and the previous results we have

\begin{theorem}[Square Function Bounds for $\dl_t$. Part 1]\label{thm-sqaure-fctn-bounds-double-layer}
Suppose that $\L$ satisfies Hypothesis A (see Definition \ref{def-hypothesis-a}). Let $\varepsilon_0>0$, $m_0\in \NN$, and $\rho_0>0$ as in Theorem \ref{thm-general-extrapolation-thm}. Suppose $\thtm= t^m\pd_t^{m}\nb\dl_t$. Then for $m\geq m_0$, $p\in (2-\varepsilon_0, 2+\varepsilon_0)$ and $f\in C_c^\infty(\rn)$
\begin{equation}\nonumber
\| \s(\thtm f)\|_\lprn, \| \v(\thtm f)\|_\lprn \lesssim \| f\|_\lprn.
\end{equation}
\end{theorem}

\begin{theorem}[Estimates on Slices for $\nb\sl_t$]\label{thm-sup-on-slices}
Suppose that $\L$ satisfies the hypotheses of Theorem \ref{thm-general-extrapolation-thm}. If $\varepsilon_0>0$ is as in Theorem \ref{thm-general-extrapolation-thm} and  $p\in (2-\varepsilon_0, 2+\varepsilon_0)$, then for $f\in C_c^\infty(\rn)$
\begin{equation}\nonumber
\| \sl_t f\|_{L^{p^*}(\rn)} + \| \nb\sl_t f\|_\lprn \lesssim_{m,p} \| f\|_\lprn, \qquad  t>0.
\end{equation}
\end{theorem}

\begin{proof}
The proof of this result is essentially contained in \cite[Theorem 1.4]{bhlmp}, using the $L^p$ square function estimates from the previous section instead. We omit the details.
\end{proof}

As a consequence  we obtain the necessary boundedness on slices of our operators.

\begin{corollary}\label{cor-lp-slices-estimates-grad-st-grad}
Suppose that$\L$ satisfies the hypotheses of Theorem \ref{thm-general-extrapolation-thm}. If $\varepsilon_0$ is as in Theorem \ref{thm-general-extrapolation-thm}, then for every $m\geq 1$, $p\in (2-\varepsilon_0, 2+\varepsilon_0)$ and $f\in C_c^\infty(\rn; \CC^{n+1})$
\begin{equation}\nonumber
\| t^m\pd_t^{m-1}(\sl_t\nb)\cdot f\|_{L^{p^*}(\rn)}+ \| t^m\pd_t^{m-1} \nb(\sl_t\nb)\cdot f\|_\lprn \lesssim_{m,p}\| f\|_\lprn, \qquad t>0.
\end{equation} 

If either of the gradients is replaced by $\pd_t$, then the above remains true for $m=0$.
\end{corollary}

\begin{proof}
It is enough to treat $\nb(\sl_t\nbp)$ by $t$-independence. The idea is to write $\nb(\sl_t\nbp f)= \nb\sl_t \divp f$ and apply the $L^p$ Caccioppoli's inequality on slices (Proposition \ref{prop-caccioppoli-on-slices}) once, and then use induction (recall that the off-diagonal decay already gives uniform $L^p$ bounds for $m$ large enough). Details can be found in \cite[Lemma 2.11]{AAAHK}.
\end{proof}

Similarly we have the result for $B_i$ in place of the gradient, the proof is a simple application of Sobolev's inequality for $m=0$, Caccioppoli's inequality on slices, and duality.

\begin{corollary}
Suppose that $\L$ satisfies the hypotheses of Theorem \ref{thm-general-extrapolation-thm}. If $\varepsilon_0>0$ is as in Theorem \ref{thm-general-extrapolation-thm}, $m\geq 0$, and $p\in (2-\varepsilon_0,2+\varepsilon_0)$, then for $f\in C_c^\infty(\rn)$ and $g\in C_c^\infty(\rn; \CC^{n+1})$
\begin{equation}\nonumber
\| t^m\pd_t^m B_i \sl_t f\|_\lprn \lesssim \| f\|_\lprn,\qquad\| t^m\pd_t^m (\sl_t B_i)\cdot g\|_\lprn \lesssim \| g\|_\lprn.
\end{equation} 
\end{corollary}

The following result is really a corollary of the above estimates; we state it on its own for future reference.

\begin{theorem}[Estimates on Slices for $\dl$]\label{thm-slices-double-layer}
Suppose $\L$ satisfies the hypotheses of Theorem \ref{thm-general-extrapolation-thm}. If $\varepsilon_0$ is as in Theorem \ref{thm-general-extrapolation-thm}, then for $m\geq 0$ and $p\in (2-\varepsilon_0, 2+\varepsilon_0)$ and $f\in C_c^\infty(\rn)$,
\begin{equation}\nonumber
\| t^m\pd_t^m \dl_t f\|_\lprn \lesssim \| f\|_\lprn, \qquad t>0.
\end{equation}
\end{theorem}

\begin{proof}
Again by Caccioppoli inequality on slices it is enough to treat the case $m=0$. The result is an immediate consequence of the following representation formula for the double layer given earlier: For $f\in C_c^\infty(\rn)$ we have $\dl_t f= (\sl_t\nb)\cdot(A\vec{N} f) + (\sl_t\overline{B}_2)\cdot \vec{N} f$, where $\vec{N}= -e_{n+1}$ is the exterior normal.
\end{proof}

The following result will be used in the proof of the non-tangential maximal function estimate. 
\begin{lemma}\label{lem-slices-on-Lipschitz-graph}
Suppose that $u\in W^{1,2}_{\loc}(\reu)$ is a solution of $\L u=0$ in $\reu$, given by either $u= \sl_t f$ or $u=\dl g$ for some $f,g\in C_c^\infty$. For any positive Lipschitz function $\varphi:\rn \to \RR$ with $\| \nb \varphi\|_{L^\infty(\rn)}\leq 1$, if we define the function $u_\varphi(x,t):= u(x,t+\varphi(x))$,  $(x,t)\in \reu$, then we have 
\begin{equation}\nonumber
\sup_{t>0} \|  u_\varphi(\cdot, t)\|_\lprn \lesssim \| \v(t\nabla u)\|_\lprn,
\end{equation}
for $p\in (2-\varepsilon_0, 2+\varepsilon_0)$ as in Theorem \ref{thm.verticalsqfctnbounds} (one has to keep in mind here that the vertical square function ``travels down" as long as we have good estimates on slices).
\end{lemma}

\begin{proof}
We sketch the proof: The function $u_\varphi$ solves an elliptic equation $\L_\varphi u_\varphi=0$ of the same type as $\L$, with the corresponding norms of the operator $\L_\varphi$ controlled in terms of those for $\L$ and $\|\nb\varphi\|_{L^\infty(\rn)}$. Next, we apply \cite[Theorem 6.12]{bhlmp} (or, more precisely, its proof of the uniform $Y^{1,2}(\bb R^n)$ estimate). 
\end{proof}

\section{Non-tangential Maximal Function Estimates}

\begin{proposition}\label{prop-improvement-ntmax}
Let $u\in W^{1,2}_{\loc}(\reu)$ be a solution of $\L u=0$ in $\reu$. For all $q\geq 1$ and $\varepsilon>0$ it holds that\footnote{See Definition \ref{def-ntmax} for the truncated maximal function $\mntmt^{(\varepsilon)}$.} 
\begin{equation}\nonumber
\mntmt^{(\varepsilon)}(\nabla u) \lesssim \m(\mntmq(\pd_t u)) + \m(\nabla_{\|} u(\cdot, \varepsilon)) + \m(u(\cdot,\varepsilon))\cdot \m_2(B_1),
\end{equation}
with implicit constants depending on dimension, ellipticity and $q$. Here we recall that we have defined, for any $r>0$ and $g\in L^r_{\loc}(\rn)$, $\m_r(g):=\m(|g|^r)^{1/r}$. In particular, if $1<p<n$ and $u(\cdot, \varepsilon)\in Y^{1,p}(\rn)$ for every $\varepsilon>0$, 
\begin{equation}\nonumber
\| \mntmt(\nabla u)\|_\lprn \lesssim \| \mntmo(\pd_t u)\|_\lprn + \sup_{\varepsilon>0} \| \nabla_{\|} u(\cdot,\varepsilon)\|_\lprn,
\end{equation}
whenever the right hand side is finite. Moreover, for $u$ we have that for any $p>1$,
\begin{equation}\nonumber
\| \mntmt(u)\|_\lprn \lesssim \|\mntmo(u)\|_\lprn.
\end{equation}  
\end{proposition}

\begin{proof}
The statement for $u$ follows from the reverse H\"older inequality for solutions (Proposition \ref{prop-prelim-inhomog-reverse-holder}) and the comparability of $\mntm$ defined with different parameters for $\cxt$.

Fix $\varepsilon > 0$ and set $u_\varepsilon(x) = u(x, \varepsilon)$.
Fix $z\in \rn$ and $(x,t)\in \Gamma(z)$. Recall that we defined the cylinders $\cxt$ in Section \ref{sec.prelim}. We denote by $\cxtd$ the concentric dilate $2\cxt$ and $a_1^*$ the   $L^1$ average built with $\cxtd$ instead of $\cxt$ in (\ref{eq.avg}). By the reverse H\"older inequality in Proposition \ref{prop-prelim-inhomog-reverse-holder}, for any $c\in \CC$
\begin{equation}\nonumber
a_2(\nabla u)(x,t) \lesssim \dfrac{1}{t}a_1^*(u-c)(x,t) + |c|a_2^*(B_1)(x,t)=: I+II.
\end{equation}
If we choose  $c:=\fint_{\Delta_{x,t}} u_\varepsilon\, dz$, then we immediately see, exploiting the $t$-independence of $B_1$, 
\begin{equation}\nonumber
II \leq \m(u_\varepsilon)(z)\cdot a_2^*(B_1)(x,t) \lesssim \m(u_\varepsilon)(z)\cdot \m_2(B_1)(z).
\end{equation}
It remains to estimate $I$. For this purpose we compute 
\begin{equation}\nonumber
a_1^*(u-c)(x,t) \leq a_1^*(u-u_\varepsilon)(x,t) + a_1^*(u_\varepsilon-c)(x,t).
\end{equation}
From the definition of $c$ we see that 
\begin{equation}\nonumber 
a_1^*(u_\varepsilon-c)(x,t)  = \dfint_{\cxtd} \Big|u_\varepsilon(y)-\fint_{\Delta_{x,t}} u_\varepsilon(w)\, dw \Big|\, dyds  \lesssim t \fint_{\Delta_{x,t}} |\nabla_{\|} u_\varepsilon(y)|\, dy 
  \leq t\m(\nabla_{\|} u_\varepsilon)(z),
\end{equation}
where we used the Poincar\'e inequality in $L^1$. On the other hand, by the fundamental theorem of calculus, and introducing an average in space, we have that
\begin{multline}\nonumber 
a_1^*(u-u_\varepsilon)(x,t)  = \dfint_\cxtd \Big| u(y,s)-u(y,\varepsilon)\Big|\, dyds 
  \leq  \dfint_\cxtd  \int_\varepsilon^s |\pd_\tau u(y,\tau)|\, d\tau\, dyds 
  \leq \dfint_\cxtd  \int_0^s |\pd_\tau u(y,\tau)| d\tau dyds\\
  = \fint_{|t-s|<t/8}\int_0^s \fint_{|x-y|<t/8}\fint_{|y-w|<\tau/8} |\pd_\tau u(y,\tau)|\, dwdy\, d\tau  ds\\
  \lesssim  \fint_{|t-s|<t/8} \fint_{|x-w|<t/2} \int_0^s \fint_{|w-y|<\tau/8} |\pd_\tau u(y,\tau)|\, dyd\tau\, dwds. 
\end{multline}
Now we notice, for fixed $s>0$, 
\begin{multline}\nonumber 
\int_0^s\fint_{|w-y|<\tau/8} |\pd_\tau u(y,\tau)|\, dyd\tau   \lesssim \sum_{k\geq 0} 2^{-k}s\fint_{2^{-k-1}s}^{2^{-k}s} \fint_{|w-y|<2^{-k-3}s} |\pd_\tau u(y,\tau)|\, dyd\tau \\
 \lesssim \sum_{k\geq 0} 2^{-k} s a_1^*(\pd_\tau u)(y, 2^{-k-1/2}s)  
  \lesssim \sum_{k\geq 0} 2^{-k}s \mntmo(\partial_\tau u)(y) 
  \lesssim s\mntmo(\pd_\tau u)(y). 
\end{multline}
Plugging this estimate into the inequality preceding it we arrive at 
\begin{equation}\nonumber
a_1^*(u-u_\varepsilon)(x,t) \lesssim \fint_{|t-s|<t/8}\fint_{|x-w|<t/2} s\mntmo(\partial_\tau u)(y)\, dy ds \lesssim t\m(\mntmo(\partial_\tau u))(z), 
\end{equation}
where we used the fact that $t\approx s$ for the last inequality. We conclude $I \lesssim \m(\mntmo(\partial_\tau u))(z)$, which, when combined with the estimate for $II$ yields the desired result in the case that $q=1$. 
\end{proof}

\begin{lemma}\label{lem-travelling-down-ntmax-weak-lp} Suppose that $\L$ satisfies Hypothesis A (see Definition \ref{def-hypothesis-a}). Let $u(x,t)=\pd_t \sl_t f(x)$ for some $f\in C_c^\infty(\rn)$ or $u(x,t)= \dl_t g$ for some $g\in C_c^\infty(\rn)$. There exists $m_6\in \NN$ and $\varepsilon_6>0$ such that if $m\geq m_6$ and $p\in (2-\varepsilon_6, 2+\varepsilon_6)$ then for every $q<p$, 
\begin{equation}\nonumber
\| \mntmq(\lthtm f)\|_\lprni \lesssim \| \s(\thtmo f)\|_\lprn + \| \v(\thto f)\|_\lprn ,
\end{equation} 
with implicit constants depending on $p,m,n$ and ellipticity; and where we have defined, in the case of $u(x,t)=\pd_t \sl_t f(x)$, 
\begin{equation}\nonumber
\thtm f:= t^m\pd_t^{m} \nabla \sl_t f=t^m\pd_t^{m-1}\nabla u,\quad\text{and}\quad \lthtm f:= t^m\pd_t^m \pd_t \sl_t f=  t^m\pd_t^{m}u,
\end{equation}
and in the case of $u(x,t)= \dl_t g$,
\[
\thtm f= t^m\pd_t^{m_1}\nb \dlp f,\quad\text{and}\quad \lthtm f= t^m\pd_t^m \dlp f.
\]  
Therefore the conclusion can be rewritten, in terms of $u$, as
\begin{equation}\nonumber
\| \mntmq(t^m\pd_t^m u)\|_\lprni \lesssim \| \s(t^{m+1}\pd_t^{m}\nabla u)\|_\lprn + \| \v(t\nabla u)\|_\lprn.
\end{equation}
\end{lemma}

\begin{proof}
For $m\geq 0$ let us define a modified version of the averages $a_q$ given by 
\begin{equation}\nonumber
\aqm(u)(x,t):= \Big( \dfint_\cxt |t^m\pd_s^m u(y,s)|^q\, dyds\Big)^{1/q} \approx \Big( \dfint_\cxt |s^m\pd_s^m u(y,s)|^q\, dyds\Big)^{1/q}.
\end{equation}
In particular, for $z\in \rn$, $\mntmq(\lthtm)(z) \approx \sup_{(x,t)\in \Gamma(z)} \aqm(u)(x,t)$.  Writing for $\lambda>0$,
\begin{multline}\nonumber 
|\{ z\in \rn: \mntmq (\lthtm f)(z)>\lambda \}|   \leq  |\{ z\in \rn : \mntmq(\lthtm f)(z)>\lambda, \, \s(\thtmo f)(z)\leq \gamma\lambda \}|\\
   + |\{ z\in \rn: \s(\thtmo f)(z)>\gamma\lambda \}|, 
\end{multline}
we see that it is enough to prove that, for $\gamma>0$, 
\begin{equation}\nonumber
|\ele|:=|\{ z\in \rn : \lmntmq(\lthtm f)(z)>\lambda, \, \s(\thtmo f)(z)\leq \gamma\lambda \}| \lesssim \lambda^{-p} \| \v(\thto f)\|_\lprn,
\end{equation}
with constants uniform in $\varepsilon$ and $\lambda$. Define the function $\vpe$ as follows:
\begin{equation}\nonumber
\vpe(x):=\inf \Big\{ t>\varepsilon: \sup_{(y,s)\in (x,t)+\Gamma(0)} \aqm(u)(y,s)\leq \lambda \Big\}.
\end{equation}
Recall that we have, from the control on slices in Corollary \ref{cor-lp-slices-estimates-grad-st-grad}, if $p\in (2-\varepsilon_0, 2+\varepsilon_0)$ (here $\varepsilon_0$ is as in Theorem \ref{thm-general-extrapolation-thm} and Corollary \ref{cor-lp-slices-estimates-grad-st-grad}) 
\begin{equation}\label{eq-ntestimate-sup-on-slices}
\sup_{t>0} \| t^m\pd_t^m u(\cdot,t)\|_\lprn <\infty,
\end{equation}
so in particular
\begin{equation}\label{eq-1-nlessthans}
\aqm(u)(y,s)\lesssim s^{-n/q} \sup_{t>0}\| t^m\pd_t^m u\|_\lprn.
\end{equation}
We conclude that $\lim_{s\to\infty}\sup_{y\in \rn}\aqm(u)(y,s) =0$, and so $\varepsilon\leq \vpe(x)<\infty$ for every $x\in \rn$. Moreover $\vpe$ is a Lipschitz function with constant 1, since it satisfies the appropriate uniform cone condition. We set $\Gamma_\vpe:= \{ (x,t)\in \reu : t=\vpe(x)\}$, the graph of $\vpe$. Finally recall that we denote $(u)_A$ the average of $u$ over a set $A\subset \reu$ of finite measure.

We first claim that, for every $z\in \ele$ and if we denote $Z_\varepsilon :=(z,\vpe(z))$, $\lambda \lesssim \m_{\Gamma_\vpe}(\aqm(u)) (Z_\varepsilon)$, for some implicit constant independent of $\lambda$ and $\varepsilon$, and where $\m_{\Gamma_\vpe}$ denotes the maximal function on $\Gamma_\vpe$ with its natural surface measure, which we denote by $\sigma$. To see this fix $z\in \ele$ and $(x,t)\in Z_\varepsilon+\Gamma(0)$ and note that, owing to \eqref{eq-1-nlessthans} there exists $R>0$ such that $\aqm(u)(x,t)\leq \lambda/2$ if $t>R$. Moreover using that $\lmntmq(\lthtm f)(z)>\lambda$, so that $\vpe(z)>\varepsilon$, there exists $(y,s)\in \overline{Z_\varepsilon+\Gamma(0)}$ satisfying $\aqm(u)(x,t)>\lambda$. By continuity of $\aqm$ in $\reu$ we conclude from the above that there exists a point $(x,t)\in\overline{Z_\varepsilon+\Gamma(0)}$ such that $\aqm(u)(x,t)=\lambda$, and $(x,t)\in \overline{\{\aqm(u)>\lambda \}}$. Note also that the above implies $(x,t)\in \Gamma_\vpe$, i.e. $t=\vpe(x)$: For every $\delta>0$ there exists $(y,s)\in B((x,t), \delta)$ such that $\aqm(y,s)>\lambda$, and so, since $B((x,t),\delta)\subset (x,t-\sqrt{2}\delta)+\Gamma(0)$, we have $\vpe(x)>t-\sqrt{2}\delta$. Since $\delta$ was arbitrary we conclude $\vpe(x)\geq t$. On the other hand, by the Lipschitz condition on $\vpe$ it can't happen that $\Gamma_\vpe$ intersects the interior of the cone $Z_\varepsilon+\Gamma(0)$, therefore $\vpe(x)\leq t$. Notice that, in fact, the above shows that $(x,t)=(x,\vpe(x))\in \partial(Z_\varepsilon+\Gamma(0))$. 

Given such a point $(x,t)=(x,\vpe(x)):=X_\varepsilon$, and for any $(y,s)\in B(X_\varepsilon, t/100)$ we have, by the Poincar\'e-Sobolev inequality and writing $v_m(w,\tau):=\pd_\tau^m u(w,\tau)$,
\begin{multline}\nonumber 
\lambda  = \aqm(u)(x,t) \leq \Big( \dfint_{\cxt}t^{2m}|v_m-(v_m)_{\cxt}|^q\, dwd\tau\Big)^{1/q}+ t^m|(v_m)_{\cxt}-(v_m)_{\mathcal{C}_{y,s}}|+t^m|(v_m)_{\mathcal{C}_{y,s}}|\\
  \lesssim \Big( \dint_\cxt |\tau^m \nabla v_m|^2 \tau^{1-n}\, dwd\tau \Big)^{1/2} + t^m|(v_m)_{\mathcal{C}_{y,s}}|	 \lesssim \s(\thtmo f)(x) + t^m|(v_m)_{\mathcal{C}_{y,s}}|\\
  \leq \gamma\lambda + t^m|(v_m)_{\mathcal{C}_{y,s}}|\leq \gamma\lambda + \aqm(u)(y,s), 
\end{multline}
where we also used the fact that $x\in \ele$ and that $\tau\approx s\approx t$. Choosing $\gamma<1$ small enough we can write
\begin{equation}\nonumber
\lambda \lesssim \aqm(u)(y,s), \quad \text{for each }(y,s)\in B(X_\varepsilon, t/100).
\end{equation}
From here \eqref{eq-1-nlessthans} follows easily: Integrating the above inequality on $\Gamma_\vpe$ we have  
\begin{equation}\nonumber
\lambda \lesssim \fint_{B(X_\varepsilon,t/100)\cap \Gamma_\vpe} \aqm(u)(W)\, d\sigma(W)\lesssim \fint_{B(X_\varepsilon, t)\cap \Gamma_\vpe}\aqm(u)(W)\, d\sigma(W),
\end{equation}
where we have used $\sigma(B)\approx r(B)$ for any ball centered on $\Gamma_\vpe$. Moreover, since $X_\varepsilon\in \partial(Z_\varepsilon+\Gamma(0))$, we have $|z-x|= |t-\vpe(z)| =t-\vpe(x)<t$, and so we conclude $Z_\varepsilon\in B(X_\varepsilon, t)$ and \eqref{eq-1-nlessthans} follows.  Since \eqref{eq-1-nlessthans} holds for any $z\in \ele$ we see that 
\begin{equation}\nonumber
|\ele|\lesssim |\{ W\in \Gamma_\vpe : \m_{\Gamma_\vpe}(\aqm(u))(W)\}| \lesssim \lambda^{-p}\int_{\Gamma_\vpe} \aqm(u)^p(W)\, d\sigma(W).
\end{equation}
Therefore, it is enough to prove  that
\begin{equation}\label{eq-2-nlessthans}
\int_{\Gamma_\vpe} \aqm(u)^p(W)\, d\sigma(W)\lesssim \| \v(t\nabla u)\|_\lprn^p.
\end{equation}
We make a further reduction as follows: Note that, by the $L^q$ Caccioppoli's inequality applied $m$ times, for any $(x,t)\in \reu$,
\begin{equation}\nonumber 
\aqm(u)(x,t)   = \Big( \dfint_\cxt |t^m\partial_t^m u(y,s)|^q\, dyds\Big)^{1/q} 
  \lesssim \Big( \dfint_\cxtd |u(y,s)|^q\, dyds\Big)^{1/q}  \approx a_q^*(u)(x,t). 
\end{equation}
Therefore the result would follow from the estimate
\begin{equation}\nonumber
\int_\rn |a_q^*(u)(x,\vpe(x))|^p\, dy\approx \int_{\Gamma_\vpe} |a_q^*(W)|^p\, d\sigma(W) \lesssim \| \v(t\nabla u)\|_\lprn.
\end{equation}
To simplify notation   we set $g(x):= a_q^*(u)(x,\vpe(x))$,  and $\nu(x)=\m(g)(x)^{p-q}$. We then have that
\begin{equation}\nonumber
\int_\rn |a_q^*(u)(x,\vpe(x))|^p\, dy = \int_\rn g^p(x)\, dx \leq \int_\rn g^q(x)\nu(x)\, dx.
\end{equation}
Note that
\begin{equation}\nonumber
\dfrac{4}{5}\vpe(y)\leq \vpe(x)\leq \dfrac{4}{3}\vpe(y), \quad \text{whenever } |x-y|<\dfrac{\vpe(x)}{4},
\end{equation}
owing to the fact $\vpe$ is Lipschitz with constant one.

We now go back to the definition of $g$ and $a_q^*$ to compute, 
\begin{multline}\nonumber 
\int_\rn g^q(x)\nu(x)\, dx   \approx \int_\rn \fint_{|x-y|<\vpe/4} \fint_{3\vpe(x)/4}^{5\vpe(x)/4} |u(y,s)|^q\, dsdy\, \nu(x) dx \\
  \lesssim \int_{4/5}^{5/3} \int_\rn \fint_{|x-y|<\vpe(x)/8} |u(y,\tau \vpe(y))|^q\, dy \, \nu(x) dx \, d\tau \\
  \lesssim \int_{4/5}^{5/3} \int_\rn \m\Big( |u(\cdot, \tau\vpe(\cdot))|^q\Big)(x) \, \nu(x)dx \, d\tau.  
\end{multline}
By H\"older's Inequality with exponents $p/q$ and $p/(p-q)$ we see 
\begin{equation}\nonumber
\int_\rn \m\Big( |u(\cdot, \tau\vpe(\cdot))|^q\Big)(x) \, \nu(x)dx \leq \| \m_q(u(\cdot, \tau\vpe(\cdot)))\|_{L^{p}(\rn)}^q\| \nu\|_{L^{p/(p-q)}(\rn)},
\end{equation}
and, since $q<p$, by the boundedness of $\m_q$ in $L^p$
\begin{equation}\nonumber
\| \m_q(u(\cdot, \tau\vpe(\cdot)))\|_{L^{p}(\rn)}^q \lesssim \Big(\int_\rn |u(x,\tau\vpe(x))|^p\, dx \Big)^{q/p}.
\end{equation}
Similarly, 
\begin{equation}\nonumber
\| \nu\|_{L^{p/(p-q)}(\rn)} = \| \m(g)\|_\lprn^{p-q} \lesssim \Big(\int_\rn |g(x)|^p\, dx \Big)^{(p-q)/p}
\end{equation}
Combining the above estimates, we obtain 
\begin{equation}\nonumber
\int_\rn g^p(x)\, dx \lesssim \int_{4/5}^{5/3}\Big( \int_\rn |u(x,\tau\vpe(x))|^p\, dx\Big)^{q/p} \Big(\int_\rn g^p(x)\, dx \Big)^{(p-q)/p}\, d\tau.
\end{equation}
Using H\"older's inequality again (perhaps  with $g(x)\bbm1_{g<M}$ if necessary in order to divide by $\| g\|_\lprn$), 
\begin{equation}\nonumber
\int_\rn g^p(x)\, dx \lesssim \int_{4/5}^{5/3} \int_\rn |u(x,\tau\vpe(x))|^p\, dx d\tau.
\end{equation}
We now note that $\tau\vpe$ is a Lipschitz function with constant $\tau$. Therefore the function $v(x,t)=u(x,t+\vpe(x))$ solves $\L_{\tau\vpe} v=0$ in $\reu$, where the operator $\L_{\tau\vpe}$ is of the same type as $\L$ and moreover its coefficients are controlled (in the appropriate norms) by those of $\L$. Therefore, by the control on slices by the vertical square function in Theorem \ref{thm-sup-on-slices} (see also Lemma \ref{lem-slices-on-Lipschitz-graph})
\begin{multline}\nonumber 
\int_\rn |u(x, \tau\vpe(x))|^p\, dx   = \int_\rn |v(x,0)|^p\, dx \lesssim \int_\rn \Big(\int_0^\infty |t\nabla v(x,t)|^2\, \dfrac{dt}{t}\Big)^{p/2}\, dx \\
 \lesssim \int_\rn \Big( \int_{\tau\vpe(x)}^\infty |\nabla u(x,t)|^2\, tdt\Big)^{p/2}\, dx  
  \leq \| \v(t\nabla u)\|_\lprn^p. 
\end{multline}
This yields \eqref{eq-2-nlessthans} and the result is proved.
\end{proof}

\begin{remark}
Notice that in the above lemma, the only things required for its proof were:
\begin{enumerate}
	\item $\lthtm$ satisfies good estimates on slices (as in \eqref{eq-ntestimate-sup-on-slices}).
	\item We already have, {\it for all} operators of the form $\L= -\div(A\nabla +B_1)+ B_2\cdot \nabla$ (with sufficient smallness of the first order coefficients), the control on slices by the vertical square function 
	\begin{equation}\nonumber
	\| u(\cdot, t)\|_\lprn \lesssim \| \v(t\nabla u)\|_\lprn.
	\end{equation} 
\end{enumerate} 
\end{remark}

The following result uses the fact, proved in the next section, that $\v(t^m\pd_t^{m-1}\nb(\sl_t\nb))$ satisfies $L^p$ bounds for all $m\geq 1$.

\begin{corollary}[$L^p$ estimate for non-tangential maximal functions of layer potentials]\label{cor-travelling-down-ntmax-strong-lp}
Suppose that $\L$ satisfies the hypotheses of Theorem \ref{thm-general-extrapolation-thm}. If $\varepsilon_0>0$ and $m_0$ are as in Theorem \ref{thm-general-extrapolation-thm}, and if $p\in (2-\varepsilon_0,2+\varepsilon_0)$ and $m\geq m_0\geq 1$ then 
\begin{equation}\nonumber
\| \mntmt(\thtm f)\|_\lprn \lesssim \| f\|_\lprn,
\end{equation}
where $\thtm$ is either $t^m\pd_t^{m}\nb(\sl_t\nb)$ or $t^m\pd_t^{m-1}\nb \dl_t$.
\end{corollary}

\begin{proof}
We treat only the single layer. The double layer argument is identical. Also, notice by $t$ independence it is enough to treat the operator with the inside gradient replaced by $\nbp$. First, from the pointwise inequality in Proposition \ref{prop-improvement-ntmax} and the dominated convergence theorem, we see that for any $q\geq 1$ and $p$ in the above range, and setting $\lthtm=t^m\pd_t^m(\sl_t\nbp)=-t^m\pd_t^m\sl_t\divp$, we have
\begin{equation}\nonumber
\| \mntmt(\thtm f)\|_{\lprni}\lesssim  \| \mntmq(\lthtm f)\|_\lprni + \sup_{t>0} \| \thtm f\|_\lprn.
\end{equation}
In particular by the slices estimates in Corollary \ref{cor-lp-slices-estimates-grad-st-grad} and Theorem \ref{thm-slices-double-layer}, and choosing $q$ as above,
\begin{equation}\nonumber
\| \mntmt(\thtm f)\|_\lprni \lesssim \| \s(\thtm f)\|_\lprn + \| \v(\thto f)\|_\lprn \lesssim \| f\|_\lprn, 
\end{equation}
where we have used Theorem \ref{thm.verticalsqfctnbounds} for the last step; ensuring that $\v(\thto f)$ is under control. The result now follows from real interpolation.
\end{proof}

\section{Traveling Down}

We first establish the vertical square function estimates, since there is little difficulty there. The discrepancy between these so-called traveling down procedures for the vertical and conical square functions should be contrasted with the situation in the extrapolation arguments.

\begin{theorem}\label{thm.verticalsqfctnbounds}
Suppose $\L$ satisfies the hypotheses of Theorem \ref{thm-general-extrapolation-thm}. Let $\varepsilon_0$ be as in Theorem \ref{thm-general-extrapolation-thm}, then for $p\in (2-\varepsilon_0, 2+\varepsilon_0) $ and every $m\geq 1$ it holds	
\begin{equation}\nonumber
\| \v(\thtm f) \|_\lprn \lesssim_m \| f\|_\lprn,
\end{equation}
where $\thtm$ is either $t^m\pd_t^{m-1}\nb (\sl_t\nb)$ or $t^{m}\pd_{t}^{m-1}\nb \dl_t$.
\end{theorem}

\begin{proof}[Proof of Theorem \ref{thm.verticalsqfctnbounds}]
We employ the same idea as in the $L^2$ case from \cite{bhlmp}; integrating by parts in $t$ to control the square function of $\thtm$ in terms of $\Theta_{t,m+1}$ plus terms that are bounded in $L^p$. Notice that for $m\geq m_0$ large enough the desired bound is a consequence of Theorem \ref{thm-general-extrapolation-thm} and Lemma \ref{lem-square-fuctions-with-B-inside} for the case of the double layer.

We start by defining, for $\eta>0$, 
\begin{equation}\nonumber
\int_\rn \Big(\int_0^\infty |\thtm f(x)|^2\, \dfrac{dt}{t}\Big)^{p/2}\, dx = \lim_{\eta\to 0^+} \int_\rn \Big( \int_{\eta}^{1/\eta} |\thtm f(x)|^2\, \dfrac{dt}{t}\Big)^{p/2}\, dx =: I_\eta.
\end{equation}
In particular, owing to the estimates on slices from Corollary \ref{cor-lp-slices-estimates-grad-st-grad}, we see that $I_\eta<\infty$ for all such $\eta$. We now  carry out the integration by parts
\begin{multline}\nonumber 
I_\eta   =  \int_\rn\Big( \int_\eta^{1/\eta} t^{2m-1}|\partial_t^{m} \nabla \sl_t \divp f(x)|^2\, dt \Big)^{p/2}\, dx\\
  \leq\int_\rn \Big(   |t^{m}\partial_t^{m} \nabla \sl_t \divp f(x)|^2 \Big|_{t=\eta}^{1/\eta}\Big)^{p/2}\, dx  + \int_\rn \Big(\int_\eta^{1/\eta} |\thtm f(x)||\thtmo f(x)|\dfrac{dt}{t} \Big)^{p/2} \, dx\\
  \leq C_m \| f\|_{\lprn}^p +   \int_\rn \Big(\int_\eta^{1/\eta} |\thtm f(x)||\thtmo f(x)|\dfrac{dt}{t} \Big)^{p/2} \, dx, 
\end{multline}
where we used the estimates on slices in Corollary \ref{cor-lp-slices-estimates-grad-st-grad} for the single layer and Theorem \ref{thm-slices-double-layer} for the double layer for the last line. Finally we use Cauchy's inequality with a parameter to obtain 
\begin{equation}\nonumber
\int_\rn \Big( \int_{\eta}^{1/\eta} |\thtm f(x)|^2\, \dfrac{dt}{t} \Big)^{p/2}\, dx \lesssim \| f\|_\lprn^p + \int_\rn \Big( \int_\eta^{1/\eta} |\thtmo f(x)|^2\, \dfrac{dt}{t}\Big)^{p/2}\, dx. 
\end{equation}
Letting $\eta\to 0$ we can write $\| \v (\thtm f) \|_\lprn \lesssim \| f\|_\lprn + \| \v(\thtmo f)\|_\lprn$. The result now follows by induction and Theorem \ref{thm-general-extrapolation-thm}.
\end{proof}

We now turn to the much harder task of traveling down with the conical square function. Here, although the idea is the same  integration by parts technique, the arguments become much more elaborate due to the `space averaging' happening alongside the integration over the transversal variable. To handle this we will make use of the non-tangential maximal function estimates from the previous section (see Lemma \ref{lem-travelling-down-ntmax-weak-lp}) through a modified version of the classical Carleson embedding lemma (see Lemma \ref{lem-prelim-modified-carleson}). However, the use of the non-tangential maximal function makes the traveling down procedure for either this object or the conical square function a bit subtle. It is our hope that the following lemma (which should be read with the results of the previous section in mind) and Theorem \ref{thm-ntmax-estimates-final-version} clarifies the intertwining of these two procedures.

We also note that for the range $p>2$ we already have the conical square function bounded by the vertical (see Proposition \ref{prop-prelim-comparability-square-functions}) by general facts about square functions, so it is only the case $p<2$ that is of interest.

\begin{lemma}\label{lem-traveling-down-conical}
Let $\thtm$ be either $t^m\pd_t^{m-1}\nb(\sl_t\nb)$ or $t^m\pd_t^{m-1}\nb\dl_t$. Suppose that $m\geq 1$ is given such that  
\begin{equation}\nonumber
\s(\thtm f), \s(\thtmo f), \v(\thtmo f), \v(\thtm f), \mntm(\thtm f)\in \ltrn,
\end{equation}
and $\sup_{t>0}\| \thtm f\|_\ltrn <\infty$, $\lim_{t\to\infty}\|\thtm f\|_\ltrn =0$, for every $f\in C_c^\infty(\rn)$. Then, for every $1<p<2$,
\begin{multline}\nonumber  
\| \s(\thtm f)\|_\lprn   \lesssim \| \s(\thtmo f)\|_\lprn+ \| \v(\thtmo f)\|_\lprn\\
   + \| \v(\thtm f)\|_\lprn + \| \mntm(\thtm f)\|_\lprn 
    + \sup_{t>0}\| \thtm f\|_\lprn. 
\end{multline} 
\end{lemma}

\begin{proof}
We fix $m\geq 1$ as in the hypotheses and define $g_m:= \sup_{t>0} |\thtm f|$ , $h_m= \mntm (\thtm)$,  and $H_m:=  \s(\thtmo f) + g_m + h_m$. Recall from Proposition \ref{prop-prelim-vertical-max} 
\begin{equation}\nonumber
\| g_m\|_\ltrn \lesssim \| \v(\thtm f)\|_\ltrn + \| \v(\thtmo f)\|_\ltrn + \sup_{t>0}\| \thtm f\|_{\ltrn}.
\end{equation}
Therefore $H_m\in \ltrn$ and so, by the Coifman-Rochberg theorem, we see that if we define $\nu(x):= \m( \s(\thtm f) + H_m)(x)^{p-2}$ for some $1<p<2$, then $\nu^M\in A_2$ for any $M>1$ satisfying $M(2-p)<1$ and moreover $[\nu^M]_{A_2}$ depends only on the quantity $M(2-p)$. 

We now mimic  the proof of the extrapolation lemma \ref{lem-restricted-extrapolation} and write, for fixed $1<p<2$, 
\begin{multline}\label{eq-1-traveling-down-lemma} 
\int_\rn |\s(\thtm f)(x)|^p\, dx   \leq \int_\rn \m(\s(\thtm f) + H_m)(x)^p\, dx
  = \int_\rn \m(\s(\thtm f) + H_m)(x)^2\nu(x)\, dx \\
  \lesssim \int_\rn \m(\s(\thtm f))(x)^2\nu(x)\, dx + \int_\rn \m( H_m)(x)^2\nu(x)\, dx\\
  \lesssim_{[\nu]_{A_2}} \int_{\rn} \s(\thtm f)(x)^2\nu(x)\, dx + \int_\rn \m(H_m)(x)^p\, dx 
  \lesssim \int_\rn \s(\thtm f)(x)^2\nu(x)\, dx + \int_\rn H_m(x)^p\, dx, 
\end{multline} 
where we used the boundedness of $\m$ in $L^2(\nu)$, by the above discussion, and in $\lprn$. By definition of $H_m$, and Proposition \ref{prop-prelim-vertical-max}, we have
\begin{multline}\label{eq-1-conical-traveling-down} 
\| H_m\|_\lprn   \lesssim_p \| \v(\thtm f)\|_\lprn + \|\v(\thtmo f)\|_\lprn + \sup_{t>0}\| \thtm f\|_\lprn\\
  \quad + \| \s(\thtmo f)\|_\lprn + \| \mntm (\thtm f)\|_\lprn. 
\end{multline}
It thus remains to estimate the first term in \eqref{eq-1-traveling-down-lemma}. For this we will try to emulate the procedure for the vertical square function, introducing an approximate identity $P_t$ to smooth-out the averaging implicit in the definition of $\s$. We will first fix the approximate identity: For $t>0$ we define $P_t:= e^{t^2\Delta}$, $Q_t:=t\partial_te^{t^2\Delta}=t\partial_tP_t$. We will also need to truncate our weight to formally justify our computations so we define, for $N>0$,  $\nu_N:= \min(\nu,N)$. We  compute, using Fubini's theorem,
\begin{multline}\nonumber 
\int_\rn \s(\thtm f)(x)^2\, \nu(x)dx  = \int_\rn \int_0^\infty \fint_{|x-y|<t} |\thtm f(y)|^2 \, dy\dfrac{dt}{t}\nu(x)\, dx \\
  = \int_\rn \int_0^\infty |\thtm f(y)|^2 \fint_{|x-y|<t} \nu(x)\, dx \dfrac{dt}{t}\, dy 
  \lesssim \int_\rn \int_0^\infty |\thtm f(y)|^2P_t \nu (y)\, \dfrac{dydt}{t} =: I.
\end{multline}
Now, by the monotone convergence theorem,
\begin{equation}\nonumber 
I   = \lim_{N\to \infty}\lim_{\varepsilon\to0} I_{\varepsilon,N} :=\lim_{N\to\infty}\lim_{\varepsilon\to0} \int_\rn \int_\varepsilon^{1/\varepsilon} |\thtm f(y)|^2 P_t \nu_N(y) \, \dfrac{dydt}{t}.  
\end{equation}
Recalling the definition of $\thtm$ and integrating by parts in $t$ we obtain, recalling that $Q_t:=t\partial_t P_t$, 
\newcommand{\them}{\Theta_{\varepsilon,m}}
\newcommand{\thoem}{\Theta_{1/\varepsilon,m}}
\begin{multline}\label{eq-2-traveling-down-conical} 
I_{\varepsilon,N}  = -\dfrac{1}{2m}\int_\rn \int_\varepsilon^{1/\varepsilon} \partial_t\Big(|\partial_t^m\sl_t f|^2 P_t\nu_N\Big) t^{2m}\, dydt +  \dfrac{1}{2m}\int_\rn |\thtm f|^2 P_t\nu_N \, dy\Big|_{t=\varepsilon}^{1/\varepsilon}\\
  \leq \dfrac{1}{m} \int_\rn\int_\varepsilon^{1/\varepsilon} |\thtm f||\thtmo f|P_t \nu_N \, \dfrac{dydt}{t} + \dfrac{1}{2m} \int_\rn \int_\varepsilon^{1/\varepsilon} |\thtm f|^2 |Q_t\nu_N| \, \dfrac{dydt}{t}\\
  \quad + \dfrac{1}{2m}\int_\rn |\thoem f|^2 P_{1/\varepsilon}\nu_N\, dy 
   -\dfrac{1}{2m}\int_\rn |\them f|^2P_\varepsilon\nu_N \, dy\\
 =: II_{\varepsilon,N}+ III_{\varepsilon,N} + IV_{\varepsilon,N} + V_{\varepsilon,N}. 
\end{multline}
We handle the boundary terms first. To start we note that by Theorem 1.4 in \cite{bhlmp} we have that $\lim_{t\to \infty} \| \thtm f\|_\ltrn \to 0$, so that, using $|P_t\nu_N|\lesssim N$ pointwise in $\rn$, we obtain
\begin{equation}\nonumber
IV_{\varepsilon,N} \lesssim N\int_\rn |\thoem f|^2\, dy \to 0 \qquad \text{as} \quad \varepsilon\to 0.
\end{equation}
Therefore 
\begin{equation}\label{eq-control-IV-travelling-down-conical}
\lim_{N\to \infty}\lim_{\varepsilon\to 0} IV_{\varepsilon,N} =0.
\end{equation}
On the other hand, as $\varepsilon\to 0$, we have $P_\varepsilon\nu_N\to \nu_N$ pointwise a.e. and $|\them f|^2\leq g_m^2$, by definition of $g_m$. By the Dominated Convergence Theorem (recall $g_m\in \ltrn$ and $P_\varepsilon \nu_N\lesssim N$) we get that
\begin{equation}\nonumber 
V_{\varepsilon,N}  \lesssim \int_\rn g_m^2 P_\varepsilon \nu_N\, dy  
  \substack{\varepsilon\to 0\\ \longrightarrow} \int_\rn g_m^2 \nu_N \, dy 
  \leq \int_\rn g_m^2 \nu\, dy   
  \lesssim \int_\rn g_m^p\, dy, 
\end{equation}
where we used the definition of $\nu$ in the last line to conclude $\nu(y)\leq \m(g_m)(y)^{p-2}$ and the Hardy-Littlewood maximal theorem. We conclude
\begin{equation}\label{eq-control-V-travelling-down-conical}
\limsup_{N\to \infty}\limsup_{\varepsilon\to 0} V_{\varepsilon,N} \lesssim \int_\rn g_m^p\, dx.
\end{equation}
The first term $II_{\varepsilon,N}$ we can treat as usual; using Cauchy's inequality with a parameter we see
\begin{multline}\label{eq-control-II-traveling-down-conical} 
II_{\varepsilon,N}  \lesssim \delta \int_\rn \int_\varepsilon^{1/\varepsilon} |\thtm f|^2 P_t\nu_N \, \dfrac{dydt}{t} + C(\delta) \int_\rn \int_\varepsilon^{1/\varepsilon} |\thtmo f|^2P_t\nu_N\, \dfrac{dydt}{t}\\
  = \delta I_{\varepsilon,N} + C(\delta)\int_\rn \int_\varepsilon^{1/\varepsilon} |\thtmo f|^2P_t\nu_N\, \dfrac{dydt}{t} 
\end{multline}
Choosing $\delta$ small enough we can hide the first term on the right hand side of \eqref{eq-2-traveling-down-conical}.

Finally, we rewrite $III_{\varepsilon,N}$ in the following way, using the Cauchy-Schwarz inequality,
\begin{equation}\nonumber 
III_{\varepsilon,N}  = \int_\rn \int_\varepsilon^{1/\varepsilon} |\thtm f| \dfrac{|Q_t\nu_N|}{|P_t\nu_N|}\sqrt{P_t\nu_N} |\thtm f|\sqrt{P_t\nu_N}\, \dfrac{dydt}{t}\\
  \leq I_{\varepsilon,N} \int_\rn \int_\varepsilon^{1/\varepsilon} |\thtm f|^2 d\mu_N(y,t), 
\end{equation}
where we have defined 
\begin{equation}\nonumber
d\mu_N(x,t):= \dfrac{|Q_t\nu_N(x)|^2}{|P_t\nu_N(x)|^2}P_t\nu_N(x)\, \dfrac{dxdt}{t}. 
\end{equation}
By Proposition \ref{prop-prelim-properties-heat-semigroup} and the modified Carleson's lemma (Lemma \ref{lem-prelim-modified-carleson}) we obtain 
\begin{equation}\nonumber
\int_\rn \int_\varepsilon^{1/\varepsilon} |\thtm f|^2 d\mu_N(y,t) \lesssim \int_\rn \mntm(\thtm)^2\, \nu_N\, dy= \int_\rn h_m \nu_N\, dy \lesssim \int_\rn h_m^p\, dy.
\end{equation}
Therefore, applying once again Cauchy's inequality with a parameter, we see that 
\begin{equation}\label{eq-control-III-traveling-down-conical}
III_{\varepsilon,N} \lesssim \delta I_{\varepsilon,N} + C(\delta) \int_\rn h_m^p\, dy.
\end{equation}
Combining the estimates \eqref{eq-control-II-traveling-down-conical}, \eqref{eq-control-III-traveling-down-conical} we arrive at 
\begin{equation}\nonumber
I_{\varepsilon,N} \lesssim \int_\rn \int_{\varepsilon}^{1/\varepsilon} |\thtmo f|^2P_t\nu_N \, \dfrac{dydt}{t} + \int_\rn H_m^p\, dy+ IV_{\varepsilon,N}+V_{\varepsilon,N},
\end{equation}
where we notice that $I_{\varepsilon,N}<\infty$ owing to the fact that $\nu_N\leq N$ and $\sup_{t>0} \|\thtm f\|_\ltrn<\infty$. 
We now use Proposition \ref{prop-prelim-properties-heat-semigroup} to get that $|P_t\nu_N(y)|\lesssim \fint_{|x-y|<t}\nu_N\, dx$ and so
\begin{equation}\nonumber
\int_\rn \int_\varepsilon^{1/\varepsilon} |\thtmo f(y)|^2 P_t \nu_N(y)\, \dfrac{dydt}{t}\lesssim \int_\rn \int_\varepsilon^{1/\varepsilon} \fint_{|x-y|<t} |\thtmo f(y)|^2\, \dfrac{dydt}{t} \nu(x)\, dx.
\end{equation}
Now taking first the limit as $\varepsilon\to 0$ and then as $N\to \infty$, and using \eqref{eq-control-IV-travelling-down-conical}, \eqref{eq-control-V-travelling-down-conical} and the previous equation,
\begin{equation}\nonumber
I\lesssim \int_\rn \s(\thtmo)^2\,\nu dx + \int_\rn H_m^p\, dx\lesssim \int_\rn H_m^p\, dx.  
\end{equation}
The result now follows from the definition of $H_m$, more specifically \eqref{eq-1-conical-traveling-down}. 
\end{proof}

\begin{remark}
As far as the hypotheses of the previous result are concerned, we note that we have good control on the quantities involving $\v$, by Theorem \ref{thm.verticalsqfctnbounds}. Moreover by \cite[Theorem 6.12]{bhlmp} (see also Hypothesis A in Definition \ref{def-hypothesis-a}), the conditions on the quantities $\|\thtm f\|_\ltrn$ are also satisfied. Therefore, under the same hypotheses as Theorem \ref{thm-general-extrapolation-thm}, we may rewrite the above as 
\begin{equation}\nonumber
\| \s(\thtm f)\|_\lprn \lesssim \| \s(\thtmo f)\|_\lprn + \| \mntmt(\thtm f)\|_\lprn.
\end{equation}
\end{remark}

As a corollary of this result and Lemma \ref{lem-travelling-down-ntmax-weak-lp} on the boundedness of the non-tangential maximal function we have the following

\begin{theorem}\label{thm-ntmax-estimates-final-version}
Suppose $\L$ satisfies the hypotheses of Theorem \ref{thm-general-extrapolation-thm}. Let $p\in (2-\varepsilon_0,2+\varepsilon_0)$, with $\varepsilon_0$ as in Theorem \ref{thm-general-extrapolation-thm}. Then, for every $f\in C_c^\infty(\rn)$, we have that
\begin{equation}\nonumber
\| \s(t\pd_t \nb  \sl_tf)\|_\lprn \lesssim \| f\|_\lprn, \qquad \|\mntmt(\nb \sl_t f)\|_\lprn \lesssim \| f\|_\lprn,
\end{equation}
and that
\begin{equation}\nonumber
\| \s(t\nb \dl_t f)\|_\lprn \lesssim \| f\|_\lprn, \qquad \| \mntmt(\dl_t f)\|_\lprn\lesssim \| f\|_\lprn.
\end{equation}
\end{theorem}

\newcommand{\thtmz}{\Theta_{t,m_0}}
\newcommand{\thtmzo}{\Theta_{t,m_0-1}}
\newcommand{\thtmzt}{\Theta_{t,m_0-2}}

\begin{proof}
We define $\thtm$ to be either $t^m\pd_t^{m-1}\nb(\sl_t\nb)$ or $t^m\pd_t^{m-1}\nb\dl_t$. 
For $p>2$ we have by Proposition \ref{prop-prelim-comparability-square-functions} and Theorem \ref{thm.verticalsqfctnbounds}, $\| \s(\thto f)\|_\lprn + \| \v(\thto f)\|_\lprn \lesssim \| f\|_\lprn$.  It remains to show the non-tangential maximal function bound when $p < 2$ and $p > 2$ and the conical square function bound when $p < 2$. We will show both the square function and non-tangential maximal function bounds in the case $p<2$. (The  case of the non-tangential maximal function bounds when $p > 2$ the same.)

We treat the case of the single layer first. By Theorem \ref{thm-general-extrapolation-thm}, together with the traveling down for vertical square functions in Theorem \ref{thm.verticalsqfctnbounds} and Corollary \ref{cor-travelling-down-ntmax-strong-lp} we see that for some $m_0\geq 1$, 
\begin{equation}\label{eq-1-traveling-down-conical}
\| \s(\thtmz f)\|_\lprn + \| \mntmt(\thtmz f)\|_\lprn \lesssim \| f\|_\lprn.
\end{equation} 
We shall show that \eqref{eq-1-traveling-down-conical} holds with $m_0$ replaced by $m_0-1$, as long as $m_0-1\geq 1$. To treat the non-tangential maximal function we appeal to Corollary \ref{cor-travelling-down-ntmax-strong-lp} to obtain 
\begin{equation}\nonumber
\| \mntmt(\thtmzo)\|_\lprn \lesssim \| \s(\thtmz f)\|_\lprn +\sup_{t>0}\| \thtmzo f\|_\lprn + \| f\|_\lprn \lesssim \|f\|_\lprn.
\end{equation}
This gives the desired bound as long as $m_0-1\geq 1$. By the traveling down procedure for the conical square function (Lemma \ref{lem-traveling-down-conical}) we have (recall that the vertical square function is under control for any $m$ by Theorem \ref{thm.verticalsqfctnbounds})
\begin{multline}\nonumber 
\| \s(\thtmzo f)\|_\lprn  \lesssim \| \s(\thtmz f)\|_\lprn + \| \mntm(\thtmzo f)\|_\lprn \\ 
   +\sup_{t>0}\| \thtmzo\|_\lprn + \| f\|_\lprn \lesssim \| f\|_\lprn, 
\end{multline}
and this gives the desired square function bound for $m_0-1\geq 1$. We have shown  by induction  that
\begin{equation}\nonumber
\| \s(\thto f)\|_\lprn, \| \mntmt(\thto)\|_\lprn \lesssim \| f\|_\lprn.
\end{equation}
To get the bound for $\mntmt(\nb\sl_t)$ we use Proposition \ref{prop-improvement-ntmax} to get 
\begin{equation}\nonumber
\|\mntmt(\nb\sl_t f)\|_\lprn \lesssim \| \mntmq(\pd_t\sl_t f)\|_\lprn + \sup_{t>0}\| \nb\sl_t f\|_\lprn,
\end{equation}
for any $1\leq q$. In particular, choosing $q<p$ we can apply directly Lemma \ref{lem-travelling-down-ntmax-weak-lp} and interpolation to obtain the result. Lastly, the double layer is handled in the same way, owing to the appropriate estimates from Theorem \ref{thm-sqaure-fctn-bounds-double-layer}, Theorem \ref{thm-slices-double-layer} and Corollary \ref{cor-travelling-down-ntmax-strong-lp}.
\end{proof}

\section{Existence}\label{existence.sect}

In this section, we tackle the existence of solutions to the problems $\Di_2$, $\Ne_2$, and $\Reg_2$; the case $p\neq2$ is addressed in Section \ref{sec.lp} and consists of essentially the same analysis. We must first probe the mapping properties of the single and double layer potentials (see Proposition \ref{mappingpropl2.prop}, Corollary \ref{NRcorollary.cor}, and Lemma \ref{Duniquelem1.lem}), which will allow us to deduce the jump relations (Lemma \ref{jumprelation.lem} and Proposition \ref{jumprelationsy12.prop}) and the invertibility of the associated boundary operators (Corollary \ref{boundaryvalues.cor} and Theorem \ref{invertibility.thm}). By developing this machinery, in Theorem \ref{existence.thm}, we give the desired existence result.

\begin{proposition}[Mapping Properties, Part I]\label{mappingpropl2.prop}
The operators $\sl : C_c^\infty(\rn) \to \sltp$, $\cD^{\cL,+} : C_c^\infty(\rn)\to \dltp$ both have unique continuous extensions to $L^2(\rn)$; that is,
\begin{equation}\nonumber
\sl:\ltrn\to \sltp, \quad\cD^{\cL,+} : L^2(\rn)\to \dltp. 
\end{equation}
Moreover, for $f,g\in \ltrn$ we have that $\sl g$, $\cD^{\cL,+}f\in W^{1,2}_{\loc}(\reu)$ are solutions of $\cL w=0$ in $\reu$, and we have the  square function estimates
\begin{equation}\nonumber
\|\s( t\partial_t \nabla \sl_t g)\|_\ltrn \lesssim \| g\|_\ltrn, \qquad \| \s(t\nabla \cD^{\cL,+}_t f)\|_\ltrn\lesssim \| f\|_\ltrn.
\end{equation}
Similar considerations also hold in the lower half space (in this case we work with $\cD^{\cL,-}$).
\end{proposition}

The proof is a simple density argument (using the fact that $\sltp$ and $\dltp$ are Banach spaces), and as such is omitted. 

The next result is a statement about Sobolev functions that will allow us to eventually assign boundary values to the extensions of the layer potentials defined above.

\begin{proposition}\label{prop.bvagree}
Let $u\in Y^{1,2}(\reu)$. The following statements hold.
\begin{enumerate}[(i)]
	\item If $u\in \dltp$ then  $u_0:= \lim_{t\to 0^+} u(t)$ exists  as a weak limit in $\ltrn$. Moreover $u_0$ agrees with the trace of $u$ in the sense that for every $\Phi\in C_c^\infty(\ree)$
	\begin{equation}\nonumber
	( u_0, \Phi(\cdot,0)) = \dint_{\reu} \Big( u \partial_t \Phi + \partial_t u \Phi\Big),
	\end{equation}
	where $(\cdot, \cdot)$ denotes the inner product in $\ltrn$.
	\item If $ u \in \sltp$ then $U_0:= \lim_{t\to 0^+} u(t)$ exists  as a weak limit in $\yotn$. Moreover, $U_0$ agrees with the trace of $u$ in the sense described in i).
\end{enumerate}
\end{proposition}

\begin{proof}
To prove i), we start by noticing that since $u\in \dltp$, there exists a subsequence $t_k\to 0^+$ and a function $u_0\in \ltrn$ such that $u_{t_k} \to u_0$ weakly in $\ltrn$ as $k\to \infty$. 
Now, again since $u\in \dltp$, we only need to show that $\lim_{t\to0^+} (u(t), \phi)= (u_0, \phi)$,  for each $ \phi\in C_c^\infty(\rn)$.  Consider $\Phi(x,t)=\phi(x)\eta(t)$, where $\eta\in C_c^\infty(-2,2)$ is such that $\eta \equiv 1$ in $(-1,1)$, so that $\Phi \in C_c^\infty(\ree)$. Now, the hypotheses imply that, for fixed $t>0$,  
\begin{equation}\nonumber
( u(t), \Phi(\cdot, t))= \dint_{\ree_{t}} \Big( u\partial_t \Phi + \partial_t u \Phi\Big) ,
\end{equation}
holds, which, for our particular choice of $\Phi$ and $t<1$, implies that
\begin{equation}\nonumber
(u(t), \phi) = \dint_{\ree_{t}} \Big( u\partial_t \Phi + \partial_t u \Phi\Big) .
\end{equation}
Hence, the dominated convergence theorem yields the desired result since $u\in L^{\frac{2(n+1)}{(n-1)}}(\reu)$ and $\nabla u\in L^2(\reu)$. The second part of the statement in i) now follows by the fact that $\Phi(\cdot, t)\to \Phi(\cdot,0)$ strongly in $L^2(\rn)$ for any $\Phi\in C_c^\infty(\ree)$.

The proof of ii) follows similar ideas. Arguing as before, we can prove that there exists a weak limit $U_0\in L^{2n/(n-2)}(\rn)$, and that it agrees with the usual trace in the sense described in i). Similarly, along a subsequence $t_k\to0$, we have that $\nablap u(t_k)\to v\in L^2(\rn)$ weakly for some $v\in L^2(\bb R^n)^n$. If we can show that $v=\nablap U_0$, then the result would follow from the uniqueness of the limit. To this end, we fix $\phi\in C_c^\infty(\rn;\cn)$ and compute that $( v, \phi)= \lim_{k\to \infty} (\nablap u(t_k),\phi)= -\lim_{k\to \infty} (u(t_k),\divp \phi) = -(U_0, \divp\phi)$, as desired.\end{proof}

Combining the two previous propositions, we arrive at  the following corollary whose proof is standard.

\begin{corollary}\label{boundaryvalues.cor}
For every $f,g\in \ltrn$ we can define the bounded linear operators
\begin{equation}\nonumber
\sl_0: L^2(\rn)\to \yotn, \quad \cD^{\cL,+}_0:\ltrn\to\ltrn,
\end{equation}
given by
\begin{equation}\nonumber
\sl_0 g:= \lim_{t\to0^+} \sl_t g, \quad \cD_0^{\cL,+}f:= \lim_{t\to0+} \cD^{\cL,+}_t f,
\end{equation}
where both are weak limits, the first being in $\yotn$ and the second in $\ltrn$.
\end{corollary}

We may remove the condition  $u\in \yotp$ in Proposition \ref{prop.bvagree} for solutions with trace decay at infinity.

\begin{proposition}\label{weaklimits2.prop}
Suppose that $u\in W^{1,2}_{\loc}(\reu)\cap \sltp$ satisfies that $\cL u=0$ in $\reu$. Then there exists $u_0\in \yotn$ such that $\lim_{t\to 0^+} u(t) = u_0$ exists weakly in $\yotn$. Moreover, since $u\in W^{1,2}_{\loc}(\reu)\cap S^2_+ \subset  W^{1,2}(I_R^+)$ for any $R>0$, the trace $\tro u$ exists as an element of $L^2_{\loc}(\rn)$, and $\tro u = u_0$ as distributions.

In particular, the conclusion holds for $u=\sl g$ for $g\in L^2(\rn)$  or $u=\cD^{\cL,+} f$ with $f\in \yotn$ (see Corollary \ref{NRcorollary.cor}).
\end{proposition}

\begin{proof}
Since $u\in \sltp$ there exists a subsequence $t_k\to 0^+$ and  $u_0\in L^{2n/(n-2)}(\rn)$ such that $\lim_{k\to\infty} u(t_k)=u_0$ weakly in $L^{2n/(n-2)}(\rn)$. Now, since $u\in Y^{1,2}(\Sigma_0^b)$ for any $b>0$, we have that for each $\Phi\in C_c^{\infty}(\ree)$, there exists $b>0$ such that
\begin{equation}\nonumber
(u(t_k), \Phi(t_k)) = -\dint_{\Sigma_{t_k}^b} (u \dno \Phi + \dno u \Phi).  
\end{equation} 
Fixing $\phi\in C_c^\infty(\rn)$ and extending it to $\ree$ so that $\Phi(\cdot, t)\equiv \phi(\cdot)$ in a neighborhood of $t=0$, we see that
\begin{equation}\nonumber
(u_0, \phi)= - \dint_{\Sigma_{0}^b} (u \dno \Phi + \dno u \Phi),
\end{equation} 
which gives the uniqueness of the limit $u_0$. Therefore, $\lim_{t\to0} u(t)=u_0$ exists as a weak limit in $L^{2n/(n-2)}(\rn)$. To see that $u_0\in \yotn$, we proceed as follows: Since for any weak limit $v$ in $\ltrn$ of $\nablap u$, we have that for any $\phi\in C_c^\infty(\rn;\cn)$, $(v, \phi)= \lim_{k\to \infty}( \nablap u(t_k),\phi)= - \lim_{k\to \infty} (u(t_k), \divp \phi) =- (u_0, \divp\phi)$, and we conclude that $v=\nablap u_0$. This shows that the weak limits are unique and thus, since $u\in \sltp$, the full limit exists; that is, $\lim_{t\to0} \nablap u(t) = \nablap u_0$ weakly in $\ltrn$. Now, we recall that every element $\ell\in \dyotn$ can be written in the form
\begin{equation}\nonumber
\ell( w)=  \int_{\rn} (\psi_0w + \psi\cdot \nablap w), \qquad\text{for all } w\in \yotn,
\end{equation}
for some $\psi_0\in L^{2n/(n+2)}(\rn)$ and $\psi=(\psi_1, \ldots , \psi_n) \in L^2(\bb R^n)^n$. This gives that $u(t)\to u_0$ weakly in $\yotn$.

We now turn to the proof of the final statement in the proposition. Since for every $\Phi(\cdot, t)\in C_c^\infty(\ree)$ we have that $\Phi(\cdot, t)\to \Phi(\cdot, 0)$ strongly in $\ltrn$, we need only check that $\lim_{t\to0} (u(t), \Phi(t))= (u_0,\Phi(0))$. Along these lines it is enough to prove that
\begin{equation}\nonumber
(u(t),\Phi(t))= -\dint_{\ree_{>t}} (u \dno \Phi +  \dno u \Phi) , \qquad\text{for all } t>0,
\end{equation}
but this in turn follows from \cite[Proposition 2.16]{bhlmp}, since $u\in W^{1,2}_{\loc}(\Sigma_{t/2}^\infty)$.\end{proof}

\begin{proposition}[Conormal derivative of solutions in slice spaces]\label{conormals2.prop}
Suppose that $u\in W^{1,2}_{\loc}(\reu)\cap \sltp$ satisifies $\cL u=0$ in $\reu$. Then there exists  $g$ in $L^2(\rn)$ such that 
\begin{equation}\nonumber
(g, \phi)= \dint_{\reu} \Big( ( A\nabla u + B_1)\overline{\nabla\Phi} + B_2\cdot \nabla u \overline{\Phi} \Big), \qquad \text{for all } \Phi\in C_c^\infty(\ree),
\end{equation}
where $\phi(\cdot)=\Phi(\cdot,0)$. We write $g=\partial_{\nu^{\cL,+}}u$. Moreover,   $g= \lim_{t\to0^+} -e_{n+1}\cdot\trt( A\nabla u + B_1 u)$, where the limit is taken in the weak sense in the space $\ltrn$.
In particular, this notion of the conormal derivative agrees with our previous definition in $\yotp$ whenever both exist.
\end{proposition}

\begin{proof}
We follow the proof of \cite[Lemma 4.3 (iii)]{AAAHK}. We will first show that for any $R>0$, there exists $g_R\in (C_c^\infty (\Delta_R))^*$ such that for any $\Phi\in C_c^\infty(I_R)$,
\begin{equation}\label{var.con.eq}
\langle g_R, \Phi(\cdot, 0)\rangle = \dint_{\reu} \Big( (A\nabla u + B_1)\cdot \overline{\nabla \Phi} + B_2\cdot \nabla u \overline{\Phi}\Big).
\end{equation}
In particular, this allows us to define $g\in (C_c^\infty(\rn))^*$ such that $g=\lim_{R\uparrow \infty} g_R$ in the sense of distributions and \eqref{var.con.eq} holds for any $\Phi\in C_c^\infty(\ree)$ and $g$ in place of $g_R$. Thus, fix $R>0$, $\phi\in C_c^\infty(\Delta_R)$, and $\Phi\in W_0^{1,2}(I_R)$ any extension of $\phi$ (that is, $\tro \Phi=\phi$). We define the (anti-)linear functional $\Lambda_R: C_c^\infty(\Delta_R)\to \CC$   by
\begin{equation}\nonumber
\Lambda_R( \phi)= \dint_{\reu}\Big( (A\nabla u + B_1 u)\cdot \overline{\nabla\Phi} + B_2\cdot \nabla u \overline{\Phi}\Big).
\end{equation} 
To see this is indeed well-defined, that is,  it does not depend on the extension $\Phi$, we simply note that for any two extensions $\Phi_1,\Phi_2$, we have that $\Phi_1-\Phi_2\in W_0^{1,2}(I_R^+)$, and $u\in \sltp$ solves $\cL u=0$ in $\reu$. Now, as in the proof of the Lax-Milgram theorem, we have  that $|\Lambda_R(\phi)|\lesssim \|\nabla u \|_{L^2(\Sigma_0^R)}\|\nabla \Phi\|_{L^2(I_R^+)}$. Construct $\Phi$ to satisfy that $\Delta\Phi=0$ in $I_R^+$, $\Phi=\varphi$ on $\Delta_R$, and $\Phi=0$ on $\partial I_R\cap\bb R^{n+1}_+$. In this case, we have that $
\|\nabla \Phi\|_{L^2(I_R^+)} \lesssim \| \phi\|_{\dot{H}^{1/2}(\Delta_R)}$ by the usual extension theorem. Combining these last two estimates, we arrive at $|\Lambda_R(\phi)|\lesssim \| \nabla u\|_{L^2(\Sigma_0^R)} \| \phi\|_{\dot{H}^{1/2}(\Delta_R)}$, whence   via the Riesz representation theorem there exists $g_R\in (\dot{H}^{1/2}(\Delta_R))^*$ such that $\langle g_R, \phi\rangle = \Lambda_R(\phi)$ for each $\phi\in C_c^\infty(\Delta_R)$. From the definition of $\Lambda_R$, we see that the restriction of $g_R$ to $\Delta_{R'}$ equals $g_{R'}$ whenever $R'<R$. In particular, this allows us to define a distribution $g$ such that 
\begin{equation}\nonumber
\langle g, \phi\rangle = \dint_{\reu} \Big( ( A\nabla u + B_1)\overline{\nabla\Phi} + B_2\cdot \nabla u \overline{\Phi} \Big),  
\end{equation}
for all $\Phi\in C_c^{\infty}(\ree)$ with $\Phi(\cdot, 0)=\phi$. It remains to show that $g\in \ltrn$. For this, note that via  the previous procedure we can define a conormal at height $t\geq 0$, which we denote as $g^t$, as the distribution which satisfies
\begin{equation}\label{varcon.eq}
\langle g^t, \Phi^t\rangle = \dint_{\bb R^{n+1}_{t}} \Big( ( A\nabla u + B_1)\overline{\nabla\Phi} + B_2\cdot \nabla u \overline{\Phi} \Big),  
\end{equation}
for all $\Phi\in C_c^\infty(\ree)$ where $\Phi^t(\cdot)= \Phi(\cdot, t)$. This formula shows that the conormal $\partial_{\nu,t}^{\cL,+} u$ in \cite[Definition 4.9]{bhlmp} agrees with $g^t$, as distributions in $\rn$. In particular, from the proof of \cite[Lemma 4.11 (i)]{bhlmp}  we see that, for any $t>0$, $g^t\in L^2(\rn)$ and $g^t= -e_{n+1} \cdot \trt(A\nabla u + B_1 u)$. Moreover, since $u\in S^2_+$, we have that $\| g^t\|_\ltrn \lesssim \| u\|_{S^2_+}$. By weak compactness, we can extract a subsequence $t_k\to 0$ and $\widetilde{g}$ such that $g^{t_k} \to \widetilde{g}$ weakly in $\ltrn$. From \eqref{varcon.eq} it is then easy to see that $\widetilde{g}=g^0=g$ as distributions, and the result follows. \end{proof}

We now take the first step towards proving existence of layer potential solutions, by proving the appropriate so-called jump relations for the Double Layer and the conormal derivative of the Single Layer.

\begin{lemma}[Jump Relations]\label{jumprelation.lem}
There exist bounded linear operators $K,\widetilde{K}: \ltrn\to\ltrn$ such that for every $f,g\in L^2(\rn)$ we have that $\big( \pm\frac{1}{2}I+ \widetilde{K}\big)g= \partial_{\nu}^{\cL,\pm} (\sl g)$, and $\big( \mp\frac{1}{2}I + K\big)f = \cD^{\cL,\pm}_0 f$.
\end{lemma}

\begin{proof}
First,   by \cite[Propositions 4.18 and 4.22]{bhlmp},  we can define operators $K: \yoh \to \yoh$, $\widetilde{K}: \dyoh\to\dyoh$,
such that the identities in the lemma  are satisfied for $f,g\in C_c^\infty(\rn)$. Moreover, by Propositions \ref{mappingpropl2.prop} and \ref{conormals2.prop}, we obtain that $K, \widetilde{K}$ are   $\ltrn$ bounded (that is, admit a unique linear, continuous extension to $\ltrn$); the result now follows via a density argument.
\end{proof}

\begin{corollary}[Additional mapping property of $\mathcal D$]\label{NRcorollary.cor} Suppose $\L$ satisfies Hypothesis A (see Definition \ref{def-hypothesis-a}). Assume further that the operator $(\sl_0)^{-1}:\yotn \to \ltrn$ exists and is bounded. Then we have
\begin{equation}\nonumber
\sup_{t>0} \| \cD_t^{\L, +} f\|_\yotn \lesssim \| f\|_\yotn,
\end{equation}
with implicit constants depending on dimension, ellipticity of $\L$, and the norm of $(\sl_0)^{-1}$.
\end{corollary}

\begin{proof}
We know by Theorem \ref{invertibility.thm} that the map $\sl_0 : \ltrn \to \yotn$ is bounded and invertible. In particular we have that the set
\begin{equation}\nonumber
\F := \Big\{ f\in \yotn : f= \sl_0 \psi, \, \psi\in C_c^\infty(\rn)\Big\} 
\end{equation}
is dense in $\yotn$. We note that for $f\in \F$ we have $f\in \yoh\cap \yotn$ by \cite[Proposition 4.7 (iii)]{bhlmp}. For such an $f$ and $\psi:=(\sl_0)^{-1} f$ we set $u(\cdot, \tau):= \mathcal{S}_\tau^{\cL} \psi,\quad \tau<0$. Then by \cite[Theorem 4.16 (iv)]{bhlmp}, applied to $u$ in $\ree_-$, we have
\begin{equation}\label{mappropd1.eq}
\cD^{\cL,+}(f)= -\sl(\partial_{\nu^{\cL,-}} u), \quad \textup{ in } \reu.
\end{equation}
Now recall from the jump relations (see \cite[Proposition 4.22 (ii)]{bhlmp}) that 
\begin{equation}\nonumber
\partial_{\nu^{\cL,-}} u = \Big( -\dfrac{1}{2}I + \widetilde{K}\Big) \psi, \quad \textup{ in } \dyoh,
\end{equation}
so that, using the definition of $u$, \eqref{mappropd1.eq} becomes
\begin{equation}\label{mapprop2.eq}
\cD^{\cL,+} (f)= -\sl \Big( \Big( -\frac{1}{2}I+\widetilde{K}\Big) \psi \Big)=-\sl \Big( \Big( -\frac{1}{2}I+\widetilde{K}\Big) (\mathcal{S}_0^{\cL})^{-1} f \Big).
\end{equation}
Finally from Proposition \ref{mappingpropl2.prop} and Theorem \ref{invertibility.thm} we know the following maps are bounded
\begin{gather}\nonumber 
\sl: \ltrn \to \sltp,   \Big(- \dfrac{1}{2}I +\widetilde{K}\Big): \ltrn \to \ltrn, \quad \\
(\mathcal{S}_0^{\cL})^{-1}:   \yotn \to \ltrn,\nonumber 
\end{gather}
which gives the desired bound
\begin{equation}\nonumber
\| \cD^{\cL,+} (f) \|_{\sltp} \lesssim \| f\|_{\yotn}, \quad f\in \F.
\end{equation}
We conclude the claimed inequality from the density of $\F$ in $\yotn$.
\end{proof}

\begin{proposition}\label{jumprelationsy12.prop}
Suppose that $\L$ satisfies Hypothesis A (see Definition \ref{def-hypothesis-a}). Assume further that $(\sl_0)^{-1}:\yotn\to \ltrn$ exists and is bounded. Let $f\in \yotn$, then the operator $K$ from Lemma \ref{jumprelation.lem} extends to a bounded operator from $\yotn$ to itself. Moreover we retain the jump relation
\begin{equation}\nonumber
\Big( -\dfrac{1}{2} I + K\Big)  f = \lim_{t\to 0} \cD_t^{\cL,+} (f),
\end{equation}
where the limit on the right is a weak limit in $\yotn$.
\end{proposition}

\begin{proof}
The proof of the first statement follows from Corollary \ref{NRcorollary.cor}, which together with Proposition \ref{weaklimits2.prop} guarantees the existence of a weak limit in $\yotn$ for $f\in C_c^\infty(\rn)$, and Lemma \ref{jumprelation.lem} which gives the desired identity.
\end{proof}

\begin{lemma}[Additional mapping property of $\mathcal S$]\label{Duniquelem1.lem}
Suppose $\L$ satisfies Hypothesis A (see Definition \ref{def-hypothesis-a}). Assume further that the inverse operators $(\mathcal{S}^{\L^*}_0)^{-1}, (\sl_0)^{-1}:\ltrn\to [\yotn]^*$ and 
\begin{equation}\nonumber
\Big(-\dfrac{1}{2} + K\Big)^{-1}: \ltrn\to \ltrn,
\end{equation}
exist and are bounded. Then the operator $\mathcal S$ extends as a bounded operator $\mathcal S: [Y^{1,2}(\rn)]^* \to \dltp$, that is,
$$\sup_{t > 0} \| \sl_t g \|_{L^2(\rn)} \lesssim \| g\|_{[Y^{1,2}(\rn)]^*},$$
with implicit constants depending on dimension, ellipticity of $\L$ and the norm of $(\mathcal{S}^{\cL}_0)^{-1}, (\sl_0)^{-1}$.
\end{lemma}

\begin{proof}
Notice that, by the mapping properties of $-\frac{1}{2}I+K$ (see Corollary \ref{NRcorollary.cor}), and using the smallness of $\| B_i\|_{\lnrn}$, we obtain that 
\begin{equation}\nonumber
-\frac{1}{2} I + \widetilde{K} : \yotn \to \yotn
\end{equation}
is bounded and invertible. From this and the Green's Formula (see \cite[Theorem 4.16 (iv)]{bhlmp}) we have that for any $g\in C_c^\infty(\rn)$
\begin{equation}\label{mapprops1.eq}
\cD^{\cL, +}(\sl_0 g) = -\sl \Big( \Big( -\dfrac{1}{2}I + \widetilde{K}\Big) g\Big).
\end{equation} 

By Proposition \ref{jumprelationsy12.prop} and Corollary \ref{boundaryvalues.cor} we have that, taking weak limits in $\yotn$ as $t\to 0$
\begin{equation}\nonumber
\Big( -\dfrac{1}{2}I +K\Big) (\sl_0 g) = - \sl_0 \Big(-\dfrac{1}{2}I + \widetilde{K}\Big) g,
\end{equation}
or equivalently 
\begin{equation}\nonumber
-(\sl_0)^{-1} \Big( -\dfrac{1}{2} I +K\Big) \sl_0 g = \Big( -\dfrac{1}{2}I + \widetilde{K}\Big) g=: h,
\end{equation}
which means, using the corresponding mapping properties for $-(1/2)I+K$  and $\sl_0$ (see Proposition \ref{jumprelationsy12.prop} and the fact that $adj(\sl_0)= \mathcal{S}^{\cL^*}_0$), that we can extend
\begin{equation}\nonumber
-\dfrac{1}{2}I + \widetilde{K}: \dyotn \to \dyotn
\end{equation}
as a bounded and, with smallness of $\| B_i\|_{\lnrn}$, invertible operator. In particular $\| g\|_{\dyotn}\approx \| h\|_{\dyotn}$. Using this in \eqref{mappropd1.eq} we arrive at the fact that, for $g\in C_c^\infty(\rn)$ it holds $\sl h \in \dltp$ and
\begin{equation}\nonumber
\| \sl_t h\|_{\dltp} \lesssim \| \sl_0 g\|_\ltrn \lesssim \| g\|_{\dyotn} \approx \| h\|_{\dyotn}.
\end{equation}
Since the set 
\begin{equation}\nonumber
\Big\{ h=\Big( -\dfrac{1}{2}I + \widetilde{K}\Big)g: g\in C_c^\infty(\rn)\Big\}
\end{equation}
is dense in $\dyotn$ we conclude the desired property by a density argument. 
\end{proof}

\begin{definition}[Hypothesis B]\label{def-hypothesis-b}
We will say $\L$ satisfies Hypothesis B if the following properties hold.
\begin{enumerate}
	\item $\L$ satisfies Hypothesis A, along with the hypotheses of Theorem \ref{thm.verticalsqfctnbounds} and Theorem \ref{thm-ntmax-estimates-final-version}.
	\item The following operators are invertible
	\begin{equation}\nonumber
	\mathcal{S}^{\L^*}_0, \sl_0: \ltrn\to \yotn, \qquad -\dfrac{1}{2} + K : \ltrn \to \ltrn.
	\end{equation}
	\item The following operators are invertible
	\begin{equation}\nonumber
	\pm \dfrac{1}{2}+ K: \yotn\to\yotn,
	\end{equation}
	\begin{equation}\nonumber
	\pm\dfrac{1}{2}+ \widetilde{K}: \ltrn\to\ltrn.
	\end{equation}
\end{enumerate}
\end{definition}

The first condition in Hypothesis B ensures that we have the right square and non-tangential maximal function estimates in $\lprn$ for the layer potentials associated to $\L$ and $\L^*$. In particular the first condition implies that the objects in item (2) are well-defined and bounded (not necessarily invertible in general). The objects in item (2), more specifically their inverses, are used in the previous Propositions to define the objects in (3); this is the reason for the statement to be written in this way.

\begin{theorem}[Invertibility of Layer Potentials]\label{invertibility.thm}
Suppose $\L_0$ satisfies hypothesis B (see Definition \ref{def-hypothesis-b}), with coefficients $A^0, B_i^0$ for $i=1,2$, and let $\L_1$ be defined by 
\begin{equation}\nonumber
\L_1= -\div((A^0+M)\nabla +(B_1^0+B_1)) + (B_2^0+B_2)\cdot \nabla.
\end{equation}
There exists $\rho>0$ depending on dimension, ellipticity of $\L_0$, and the norms of the inverse operators in item (2) of Hypothesis B with the property that if 
\begin{equation}\nonumber
\max\{ \| M\|_{L^\infty(\rn)}, \| B_1\|_\lnrn, \| B_2\|_\lnrn \}<\rho,
\end{equation}
then $\L_1$ satisfies Hypothesis B.
\end{theorem}

\begin{proof} Set $\|M\|_{\infty} = 1$ and $\| B_i\|_n=1$, $i=1,2$,  and then define the operator
\begin{equation}\nonumber
\cL_zu:= -\div((A + zM)\nabla u + (B_1^0+zB_1)u)+ (B_2^0+zB_2)\cdot \nabla u, \quad z\in \CC,\, u\in \yot. 
\end{equation}
We write  $\cL_z= \cL_0-z\cM$. The idea will be to show that $K_z$, $\widetilde{K}_z$ and $\slz_0$ are analytic in $z$ in a neighborhood of the origin. Note that, by Lax-Milgram, $\cL_0$ is always invertible, and thus there exists $\varepsilon_0$ such that if $z\in B_{\varepsilon_0}= B(0,\varepsilon_0)$, then $\cL_z$ is also invertible, and moreover $\cL_z^{-1} = \cL_0^{-1} \sum_{k=0}^\infty (z\cM \cL_0^{-1})^k$, the series converging in the operator norm of $\mathcal{B}(\dyot; \yot)$. In particular, the map $z\mapsto \cL_z^{-1}$ is analytic in $B_{\varepsilon_0}$. Now fix $t\geq0$. By definition of the single layer, we conclude that $\slz_t$ is also analytic with values in $\mathcal{B}(\dyoh; \yoh)$. Since $\nablap: \yoh\to \dyoh$, we have that $\nablap \slz_t$ is analytic in $B_{\varepsilon_0} $ with values in $\mathcal{B}(\dyoh;\dyoh)$. Thus, for $f\in C_c^\infty(\rn)$ and $g\in C_c^\infty(\rn;\cn)$, we have that the map $z\mapsto (\nablap \slz_t f, g)$ is analytic, and
\begin{equation}\nonumber
\sup_{z\in B_{\varepsilon_0}} \sup_{t\geq 0}\| \slz_t \|_{\ltrn\to\yotn} \lesssim 1.
\end{equation}
It follows by \cite[Theorem 3.12]{K} that   $z\mapsto \slz_t$ is a holomorphic map with values in $\mathcal{B}(\ltrn; \yotn)$, for any $t\geq0$. In particular, we have that $\slz_0$ is analytic. Similarly, $\partial_{\nu}^{\cL_z,+}\slz$ is analytic with values in $\mathcal{B}(\dyoh; \dyoh)$, and for $f,g\in C_c^\infty(\rn)$, we have that
\begin{equation}\label{conormunif.eq}
\sup_{z\in B_{\varepsilon_0}} \| \partial_{\nu}^{\cL_z,+}\slz \|_{\ltrn\to \ltrn}\lesssim 1,
\end{equation}
and the map $z\mapsto (\partial_{\nu^{\cL_z,+}}\slz f, g)= \lim_{t\to0} (A\nabla \slz_t f +zB_1\slz_t f, g)$ is analytic. Thus we obtain that $\partial_{\nu}^{\cL_z,+}(\slz)$ is analytic with values in $\mathcal{B}(\ltrn)$. By the jump relations in Lemma \ref{jumprelation.lem},   $\widetilde{K}_z$ is also analytic with values in $\mathcal{B}(\ltrn)$. The analyticity of $K_z$ follows from that of $\widetilde{K}_z$ by noting that $\langle \cD_0^{\cL_z,+} f,g\rangle = \langle f, \partial_{\nu}^{\cL_z^*,+} \mathcal{S}^{\cL_z^*} g\rangle - \langle f,g\rangle$ for $f,g\in C_c^\infty(\rn)$ (see \cite[Proposition 4.18 (ii)]{bhlmp}).

We have thus shown that the maps $z\mapsto \slz_0$,  $z\mapsto K_z$, $z\mapsto \widetilde{K}_z$ are all analytic. Therefore, by the Cauchy integral formula, we obtain that
\begin{equation}\nonumber
\sup_{z\in B_{\varepsilon_0}/2} \big\| \tfrac{d}{dz} \widetilde{K}_z\big\|_{\ltrn\to\ltrn} \lesssim_{\varepsilon_0} \sup_{z\in B_{\varepsilon_0}} \| \widetilde{K}_z\|_{\ltrn\to\ltrn} \lesssim_{\varepsilon_0} 1, 
\end{equation} 
where we used \eqref{conormunif.eq}. Consequently, for any $z,w\in B_{\varepsilon_0/2}$, we have that \\$\| \widetilde{K}_z-\widetilde{K}_w\|_{\ltrn\to\ltrn} \lesssim |z-w|$. This implies that for all $z$ small enough, $\widetilde{K}_z$ is invertible. The other boundary operators are treated similarly.\end{proof}

\begin{theorem}[Existence of Solutions]\label{existence.thm} Suppose $\L$ satisfies Hypothesis B (see Definition \ref{def-hypothesis-b}). Then the boundary value problems $\Di_2$, $\Ne_2$, and $\Reg_2$, as given in (\ref{eq.d2})-(\ref{eq.r2}), admit a solution.
\end{theorem}

\begin{proof}

To solve the Dirichlet problem, we fix $f\in \ltrn$ and set $F=\Big(-\dfrac{1}{2}I+K\Big)^{-1}f$, which is well-defined by Theorem \ref{invertibility.thm} as an element of $\ltrn$. Let $u:=\cD^{\cL,+}F$. Then the fact that $u\in \dltp$ follows from Proposition \ref{mappingpropl2.prop}, the non-tangential maximal function estimate follows from Theorem \ref{thm-ntmax-estimates-final-version}, while Lemma \ref{jumprelation.lem} gives the weak convergence to $f$. 

To upgrade the convergence of $\cD^{\cL,+}_t f$ to strong convergence in $\ltrn$, we mimmick the proof of \cite[Lemma 4.23]{AAAHK}. First, we note that by Theorem \ref{invertibility.thm} we have that $\cA:= \{ \sl_0 \divp g : g\in C_c^\infty(\rn)\}$ is dense in $\ltrn$. Indeed, since $adj(\sl_0)=\mathcal{S}^{\cL^*}_0$, we have that $\sl_0: \dyotn \to \ltrn$ is invertible, and therefore any $h\in \ltrn$ may be written as $h=\sl_0 H$ for some $H\in \ltrn$. Moreover,  any $H\in \dyotn$ can be written as $H=\divp g$ for some $g\in L^2(\bb R^n)^n$ (as can be seen for instance by embedding $\yotn \to L^2(\bb R^n)^n$ via $u\mapsto \nablap u$ and using the Hahn-Banach and the Riesz Representation Theorems). These observations yield the claim.

Now fix $f=\sl_0(\divp g)$ for some $g\in C_c^\infty(\rn)$ and define $u=\sl_s(\divp g)$ for $s<0$. By \cite[Theorem 4.16 (iv)]{bhlmp}, we have that $\cD_t^{\cL,+} f = -\sl(\partial_{\nu}^{\cL,-} u)$  in $\reu$. Therefore, for any $0<t'<t$, we have that
\begin{equation*}
\| \cD_t^{\cL,+}f-\cD_{t'}^{\cL,+}f\|_\ltrn= \big\| \int_{t'}^t \partial_\tau \sl_\tau(\partial_{\nu}^{\cL,+} u)\, d\tau \big\|_\ltrn \lesssim (t-t')\| \partial_{\nu}^{\cL,+} u\|_\ltrn,
\end{equation*}
where we used the estimates on slices from Theorem \ref{thm-sup-on-slices}. Thus $\{\cD_t^{\cL,+} f\}_t$ is a Cauchy sequence in $\ltrn$ as $t\to 0$.

For the Neumann problem we proceed in a similar way, with $w:=\sl(1/2I+\widetilde{K})^{-1} h$, and we appeal to Lemma \ref{jumprelation.lem}, Theorem \ref{invertibility.thm}, Proposition \ref{mappingpropl2.prop}, and Theorem \ref{thm-ntmax-estimates-final-version}.

Finally for the regularity problem we set $v:= \sl (S_0^{\cL})^{-1} g$, and make use of Corollary \ref{boundaryvalues.cor}, Theorem \ref{invertibility.thm}, and Theorem \ref{thm-ntmax-estimates-final-version}.

It remains to show the non-tangential convergence statements, to which we now turn.

The convergence for the regularity Problem goes as follows: By Proposition \ref{RegNTconverge.prop} we have that the solution $v$ has a non-tangential limit, call it $g_0$, so we only need to show $g=g_0$. We know that $v(\cdot, t)$ converges weakly to $g$ in $L^{2^*}(\rn)$ as $t \to 0^+$. Define
\[v'(\cdot,t) = \fint_{t/2}^{3t/2} v(\cdot, s) \, ds\]
then $w'(\cdot, t)$ converges weakly to $g$ in $L^{2^*}(\rn)$ as $t \to 0^+$. Indeed, for fixed $\varphi \in L^{2_*}$ we have
\[\Big|\int_{\rn} v(x,t) \varphi(x) - \int_{\rn} g(x) \varphi(x) \Big| \le \epsilon_\varphi(t),\]
where $\epsilon_\varphi(t) \downarrow 0$ as $t \to 0^+$. From this one may establish $v'(\cdot, t)$ converges weakly to $g$ in $L^{2^*}(\rn)$. Moreover, if for $f \in L^1_{\loc}(\rn)$ we define
\[(A_t f)(x) := \fint_{|y -x|_\infty < t}f(y)\, dy\]
then $\widetilde{v}(x,t) = (A_tv')(x)$ and it follows that $\widetilde{v}(x,t)$ converges weakly to $g$ in $L^{2^*}(\rn)$ as $t \to 0^+$. Indeed, for if $\varphi \in L^{2_*}$ then
\[\int_{\rn} (A_tv')(x) \varphi(x) \, dx = \int_{\rn} v'(x, t) (A_t \varphi)(x) \, dx\]
and since $v'(\cdot, t)$ converges weakly to $g$ in $L^{2^*}(\rn)$ as $t \to 0^+$ and $(A_t \varphi)(x)$ converges strongly in $L^{2_*}$ as $t \to 0^+$ we have that 
\[\int_{\rn} v'(x, t) (A_t \varphi)(x) \, dx \to \int_{\rn} g(x) \varphi(x) \,dx, \quad \text{ as } t \to 0^+.\]
It follows that $g_0(x) = g(x)$ for a.e. $x \in \rn$.

For the Dirichlet problem, we use compatible well-posedness (see below in the proof) to get that for smooth initial data $f\in C_c^\infty(\rn)$ the  solutions to the Dirichlet and regularity problems obtained via layer potentials agree. In particular, if $u_f=\dlp(-1/2+K)^{-1}f$, then $u_f$ has a non-tangential limit. Since $C_c^\infty(\rn)$ is dense in $\ltrn$ and we have the maximal function estimate 
\begin{equation}\nonumber
\| \mntmt(u_f)\|_\ltrn \lesssim \| f\|_\ltrn,
\end{equation} 
the existence of a limit for general $f\in \ltrn$ follows a standard argument.

Now we turn to the compatible well-posedness statement: If $f\in C_c^\infty(\rn)$ and we set
\begin{equation}\nonumber
u_f:= \dl\Big(-\dfrac{1}{2}+K\Big)^{-1}f, \qquad v_f:= \sl(\sl_0)^{-1}f,
\end{equation}
the layer potential solutions of the Dirichlet and regularity problems with data $f$ respectively, we claim then $u_f=v_f$ and both agree with the solution furnished via Lax-Milgram with Dirichlet data $f$.

We first prove that $u_f$ agrees with the Lax-Milgram solution. For this, by the mapping properties of the double layer (see \cite[Definition 4.6]{bhlmp}) it is enough to show that 
\begin{equation}\nonumber
Tf:= \Big(-\dfrac{1}{2}+K\Big)^{-1} f \in H_0^{1/2}(\rn).
\end{equation}
We know (see Theorem \ref{invertibility.thm} and Proposition \ref{jumprelationsy12.prop}) that $T$ maps $\ltrn$ and $\yotn$ to itself, so in particular $Tf\in W^{1,2}(\rn)\subset H^{1/2}_0$.

For $v_f$ we proceed similarly, noting that $(\sl_0)^{-1}$ maps $\yotn$ to $\ltrn$ and $\ltrn \to [\yotn]^*$ (see  Theorem \ref{invertibility.thm} and Lemma \ref{Duniquelem1.lem}). It's thus enough, by the mapping properties of the single layer (see  \cite[Proposition 4.2]{bhlmp}), to prove that 
\begin{equation}
	[\yotn]^*\cap \ltrn \subset H^{-1/2}(\rn).
\end{equation}
This follows from the fact that elements of the first space are of the form $G\in \ltrn$ such that $G=\div H$ for some $H\in L^2(\rn;\cn)$, while the second space contains all elements of the form $(-\Delta)^{1/2}F$ with $F\in H^{1/2}_0(\rn)$. Fix $G,H$ as above. By the Riesz representation theorem in $Y^{1,2}(\rn)$ there exists a weak solution $F_1\in Y^{1,2}(\rn)$ of the problem $G=\div H = -\Delta F_1$; set $F:= (-\Delta)^{1/2}F_1$, so that it's enough to prove $F\in H^{1/2}(\rn)$. First, clearly $F\in L^2(\rn)$ by Plancherel's Theorem, since $\nabla F_1\in L^2(\rn)$; moreover, since $(-\Delta)^{1/2}F=G\in L^2(\rn)$, we have that $F\in W^{1,2}(\rn)$, and by interpolation, $F\in H^{1/2}(\rn)$ as desired.
\end{proof}

\section{Uniqueness}\label{sec.uniqueness} 

In this section, we show that the solutions to the problems $\Di_2$, $\Ne_2$ and $\Reg_2$ are unique among the wider classes of good $\mathcal{D}$ solutions and good $\mathcal{N}$/$\mathcal{R}$ solutions respectively (see Definitions \ref{def-good-d-solutions} and \ref{def-good-nr-solutions}). The methods mainly rely in showing that good solutions satisfy a Green's formula when the operator $\L$ is already known to have invertible layer potentials. For the case $p\neq2$, see Section \ref{sec.lp}, although the methods are mostly the same.

We first state a technical lemma   before moving on to the uniqueness of solutions for the Neumann and regularity problems; the uniqueness of the Dirichlet problem is dealt with last. The next lemma    will allow us to prove a representation formula for solutions to the  Neumann and regularity problems.

\begin{proposition}\label{commutingtderivativesandlayer.prop}
Suppose $\L$ satisfies Hypothesis B (see Definition \ref{def-hypothesis-b}). Let $u\in W^{1,2}_{\loc}(\reu)\cap S^2_+$ be a solution of $\cL u=0$ in $\reu$. Then for every $\tau_0>0$ and every $t>0$ we have
\begin{equation*}
\partial_{\tau}|_{\tau =\tau_0} \cD^{\cL,+}_t(\tro u_\tau)= \cD^{\cL,+}_t(\tro (\dno u_{\tau_0})),
\end{equation*}
and 
\begin{equation*}
\partial_{\tau}|_{\tau=\tau_0} \sl_t(\partial_{\nu^{\cL,+}} u_\tau) = \sl_t(\partial_{\nu^{\cL,+}}( \dno u_{\tau_0})).
\end{equation*}
\end{proposition}

\begin{proof}
We work with the double layer first. For this we consider, for $t>0$ fixed, the following functions:
\begin{equation}\nonumber
f(\tau):= \tro u_\tau= \trtau u, \quad H(\tau):= \cD_t^{\cL,+} (f(\tau)).
\end{equation}
We note that, by hypothesis and Corollary \ref{NRcorollary.cor}, we have  
\begin{equation}\nonumber
f\in C((0,\infty); \yotn), \quad H\in C((0,\infty); \yotn).
\end{equation}
The idea is now to use \cite[Theorem 2.14]{bhlmp} to get the desired differentiability of $f$. For this purpose define
\begin{equation}\label{vaphidiff.eq}
\varphi(\tau):= \|  \trtau(\dno u)\|_{\yotn}= \| \nablap\trtau (\dno u)\|_\ltrn \in L^2_{\loc}((0,\infty); \RR).
\end{equation}
First we note that by \cite[Lemma 2.3]{bhlmp} we have that $f:(0,\infty)\to \yotn$ and $\varphi: (0,\infty)\to \RR$ are continuous functions. By the Lebesgue Differentiation Theorem it is then enough to show 
\begin{equation}\label{comm1.eq}
\Big(\fint_{-\varepsilon}^{\varepsilon} \| f(\tau_2+s)-f(\tau_1+s)\|_{\yotn}^2\, ds \Big)^{1/2}\leq  \int_{\tau_1}^{\tau_2} \Big( \fint_{-\varepsilon}^{\varepsilon}\varphi^2(s+\tau)\, ds\Big)^{1/2} d\tau ,
\end{equation}
for all $\varepsilon$ small enough (depending on $\tau_1$ and $\tau_2$). For this purpose we compute, calling $I$ the left hand side of \eqref{comm1.eq}, 
\begin{multline}\nonumber 
I  =\Big( \fint_{-\varepsilon}^\varepsilon \int_\rn |\nablap \tr_{\tau_2+s} u (x) - \nablap \tr_{\tau_1+s} u(x) |^2\, dxds\Big)^{1/2}\\
  = \Big( \fint_{-\varepsilon}^\varepsilon \int_\rn |\nablap u_{\tau_2}(x,s)- \nablap u_{\tau_1}(x,s)|^2\, dx ds\Big)^{1/2}
  = \Big( \fint_{-\varepsilon}^\varepsilon \int_\rn \Big| \int_{\tau_1}^{\tau_2} \nablap \partial_\tau u(x,\tau+s)\, d\tau\Big|^2\, dx ds \Big)^{1/2}\\
 \leq \int_{\tau_1}^{\tau_2} \Big( \fint_{-\varepsilon}^\varepsilon \int_\rn |\nablap \partial_\tau u(x,\tau+s)|^2\, dx ds\Big)^{1/2}\, d\tau 
  = \int_{\tau_1}^{\tau_2}\Big( \fint_{-\varepsilon}^\varepsilon  \varphi^2(s+\tau)\, ds \Big)^{1/2} d\tau,   
\end{multline}
where we used the Fundamental Theorem of Calculus   and Minkowski's inequality. As mentioned above this shows that $f\in W^{1,2}_{\loc}((0,\infty); \yotn)$. Now we will show that 
\begin{equation}\label{formfprime.eq}
f'(\tau)= \trtau(\dno u), \qquad \text{for each } \tau>0,
\end{equation}
and moreover the difference quotients converge weakly
\begin{equation}\nonumber
\Delta^h f(\tau)\ra f'(\tau), \qquad \textup{ for every } \tau>0.
\end{equation}
For this fix $\psi\in C_c^\infty(0,\infty)$, $\phi\in C_c^\infty(\rn;\cn)$ and let $\ell:=-\divp \phi \in \dyotn$. Using that the function $\tau\mapsto f(\tau)\psi'(\tau) \in \yotn$ is continuous (see again \cite[Lemma 2.3]{bhlmp}) and compactly supported on $(0,\infty)$ and properties of the Bochner integral (see for instance \cite[Proposition 1.4.22]{CH}) we obtain  
\begin{multline}\nonumber 
\Big\langle \int_0^\infty f(\tau)\psi'(\tau)\, d\tau , \ell\Big\rangle  = \int_0^\infty \langle f(\tau)\psi'(\tau), \ell\rangle \, d\tau   = \int_0^\infty \int_\rn \nablap\trtau u(x)\psi'(\tau) \phi(x)\,dx d\tau \\
  = \int_0^\infty \int_\rn \nablap u(x,\tau) \psi'(\tau)\phi(x)\, dx d\tau 
  = -\int_0^\infty \int_\rn \nablap \dno u(x,\tau)\psi(\tau)\phi(x)\, dx d\tau\\
  = -\int_0^\infty \int_{\rn} \nablap\trtau (\dno u)(x) \psi(\tau)\phi(x)\, dx d\tau\\
  = -\int_0^\infty \langle \trtau(\dno u) \psi(\tau), \ell\rangle\, d\tau 
  = \Big\langle -\int_0^\infty \trtau(\dno u) \psi(\tau)\, d\tau, \ell \Big\rangle, 
\end{multline}
where we used integration by parts in the fourth line. Now we conclude, since the collection $\{ \divp \phi: \phi\in C_c^\infty(\rn;\cn)\}$ is dense in $\dyotn$, that indeed \eqref{formfprime.eq} holds. The convergence of the difference quotients is a consequence of the fact that $f'\in C((0,\infty);\yotn)$ and the Fundamental Theorem of Calculus. In fact we get strong convergence in $\yotn$ as $h\to0$ of $\Delta^h f(\tau)$ for every $\tau>0$.

With this we can conclude the argument for the double layer: Define 
\begin{equation}\nonumber
H(\tau):= \cD^{\cL,+}_t \tro u_\tau= \cD^{\cL,+}_t \trtau u.
\end{equation}
We claim that $H\in C^1((0,\infty);\yotn)$ and 
\begin{equation}\nonumber
H'(\tau)= \cD^{\cL,+}_t(\tro (\dno u_\tau))= \cD^{\cL,+}_t(\trtau(\dno u)).
\end{equation}
Notice first that $H\in C((0,\infty);\yotn)$ by the mapping properties of the Double Layer (see Corollary \ref{NRcorollary.cor}) and the fact that $H(\tau)=\cD^{\cL,+}_t (f(\tau))$ (recall that $t>0$ is fixed throughout). Morever, using these two facts again we see
\begin{equation}\nonumber
\| H(\tau_1)-H(\tau_2)\|_{\yotn} \lesssim \| f(\tau_1)-f(\tau_2)\|_{\yotn}\leq \Big| \int_{\tau_1}^{\tau_2} \varphi(s)\, ds\Big|,
\end{equation}
where $\varphi$ is defined in \eqref{vaphidiff.eq}. This shows that $H\in W^{1,2}_{\loc}((0,\infty); \yotn)$. Moreover we have
\begin{equation}\nonumber
\Delta^h H(\tau)= \cD^{\cL,+}_t(\Delta^hf(\tau)),
\end{equation}
so that, by the linearity and continuity of $\cD^{\cL,+}_t$ in $\yotn$ and the weak convergence of $\Delta^h f(\tau)$ in $\yotn$, we obtain for every $\tau>0$
\begin{equation}\nonumber
\Delta^h H(\tau) \to \cD^{\cL,+}(f'(\tau))= \cD^{\cL,+}_t(\tro(\dno u_\tau))
\end{equation}
weakly in $\yotn$ as $h\to0$,  as desired.  

The proof for the Single Layer follows the same lines. Define, for $t>0$ fixed and $\tau>0$,
\begin{equation}\nonumber
g(\tau):= \partial_{\nu^{\cL,+}} u_\tau= \partial_{\nu^{\cL,+}_\tau} u, 
\end{equation} 
where the second equality follows from \cite[Lemma 4.11 (i)]{bhlmp}. As before we first claim that $g\in C^1((0,\infty); \ltrn)$ and we have
\begin{equation}\nonumber
g'(\tau)= \partial_{\nu^{\cL,+}}(\dno u_\tau) = \partial_{\nu^{\cL,+}_\tau} (\dno u).
\end{equation}
For this purpose we use the $L^2$ characterization of the conormal derivative (see again \cite[Lemma 4.11]{bhlmp})  
so that  
\begin{multline}\nonumber 
g(\tau) = N\cdot \tro( A\nabla u_\tau + B_1 u_\tau ) = N\cdot \trtau( A\nabla u + B_1 u)\\
 = N\cdot \trtau( \widetilde{A}\nablap u + B_1 u) + N\cdot \trtau( \vec{a}D_{n+1} u) 
  =: g_1(\tau)+ g_2(\tau),  
\end{multline}
where $\widetilde{A}:=(a_{ij})_{1\leq i\leq n+1, 1\leq j\leq n}$ and $\vec{a}:= (a_{i,n+1})_{1\leq i\leq n+1}$. We note that by the H\"older's and Sobolev's inequalities 
\begin{equation}\nonumber
\| g_1(\tau_2)- g_1(\tau_1)\|_\ltrn \lesssim \| f(\tau_2)-f(\tau_1)\|_{\yotn} \leq \Big| \int_{\tau_1}^{\tau_2} \varphi(s)\, ds\Big|,
\end{equation}
where $f,\varphi$ are as in the proof for the Double Layer. Therefore it is enough to control $g_2$, and for this we can proceed exactly in the same way as we did for $f$: For fixed $\tau_2, \tau_1$ and $\varepsilon>0$ small 
\begin{multline}\nonumber 
\fint_{-\varepsilon}^\varepsilon\| g(\tau_2+s)-g(\tau_1+s)\|_\ltrn^2\, ds  = \fint_{-\varepsilon}^{\varepsilon} \int_\rn |\dno( u_{\tau_2}(x,s)-u_{\tau_1}(x,s))|^2\, dx ds \\
 = \fint_{-\varepsilon}^\varepsilon  \int_\rn\Big| \int_{\tau_1}^{\tau_2} \dno^2 u(s+\tau)\, d\tau \Big|^2\, dx ds 
  \lesssim \fint_{-\varepsilon}^\varepsilon \int_{\tau_1}^{\tau_2} \| \tr_{\tau+s}( \dno^2 u)\|_{\ltrn} \, d\tau ds, 
\end{multline} 
and  $\widetilde{\varphi}(\tau):= \| \trtau(\dno^2 u)\|_{\ltrn} \in L^2_{\loc}(0,\infty)$. Therefore by \cite[Theorem 2.14]{bhlmp} we get that $g\in W^{1,2}_{\loc}((0,\infty); \ltrn)$ and the difference quotients converge a.e. to $g'$. To verify the formula for $g'$ we compute, for $\phi\in C_c^\infty(\rn)$ and $\psi\in C_c^\infty(0,\infty)$,
\begin{multline}\nonumber 
\Big\langle \int_0^\infty g(\tau) \psi'(\tau)\, d\tau , \phi\Big\rangle_{\ltrn}   = \int_0^\infty \int_\rn N\cdot(A\nabla u(x,\tau)+B_1u(x,\tau))\psi'(\tau)\phi(x)\, dx d\tau\\
  =- \int_0^\infty \int_\rn N\cdot(A\nabla \dno u(x,\tau)+B_1\dno u(x,\tau))\psi(\tau)\phi(x)\, dx d\tau
  = \Big\langle \int_0^\infty \partial_{\nu^{\cL,+}_\tau}(\dno u) \psi(\tau), \phi\Big\rangle_{\ltrn}.  
\end{multline}
This gives the desired representation for $g'(\tau)$. Moreover, using this representation we see that $g'\in C((0,\infty);\ltrn)$ and so the difference quotients satisfy $\Delta^h g(\tau) \to g'(\tau)$ weakly for every $\tau>0$. The result now follows from the mapping properties of the single layer (see Proposition \ref{mappingpropl2.prop}).\end{proof}

\subsection{Uniqueness for the Neumann and regularity problems}

We begin with a lemma that gives a representation of good $\mathcal N/ \mathcal R$ solutions above a positive height.

\begin{lemma}\label{NRuniquelem2.lem} Suppose that $\L$ satisfies Hypothesis B (see Definition \ref{def-hypothesis-b}). Let $u$ be a good $\mathcal N/ \mathcal R$ solution and $u_\tau(\cdot, \cdot) = u(\cdot, \cdot + \tau)$, as above.
Then 
\begin{equation}\label{NRuniquelem2.eq}
u_\tau = -\mathcal D^{\L,+}(\Tr_0u_\tau) + \mathcal S^{\L}(\partial_\nu u_\tau),
\end{equation}
where $\Tr_0u_\tau \in Y^{1,2}(\rn)$, $\partial_\nu u_\tau \in L^2(\rn)$, and $\mathcal D^{\L,+}$ and $\mathcal S^{\L}$ are viewed (as their natural extensions) from these spaces mapping into $\sltp$.
\end{lemma}

\begin{proof}
We have $\text{Tr}_0u_\tau(\cdot) = u(\cdot, \tau) \in Y^{1,2}(\rn)$ (by the fact that $u\in \sltp$) and $\partial_\nu u_\tau \in L^2(\rn)$ (by Proposition \ref{conormals2.prop}). For the latter we may consider $u_{\tau/2} \in Y^{1,2}(\ree_+)$, a solution in $\ree_+$, since the operator $\cL$ is $t$-independent. The  mappings of the layer potentials into $S^2_+$ come  from Proposition \ref{mappingpropl2.prop} and Corollary \ref{Duniquelem1.lem}.

By Proposition \ref{commutingtderivativesandlayer.prop} together with the Green's formula for $\partial_t u_\tau\in \yot$ (see \cite[Theorem 4.16 (ii)]{bhlmp}) we know that for $\tau_0 > 0$
$$\partial_\tau u_\tau|_{\tau = \tau_0}(x,t) =- \partial_\tau[\mathcal D^{\L,+}(\text{Tr}_0u_\tau) + \mathcal S^{\L}(\partial_\nu u_\tau)]\big|_{\tau = \tau_0}(x,t),$$
as functions in $\sltp$.
We may now integrate in $\tau_0$ to obtain 
$$u_\tau = -\mathcal D^{\L,+}(\text{Tr}_0u_\tau) + \mathcal S^{\L}(\partial_\nu u_\tau),$$
where we must use  the decay at infinity hypothesis in the definition of $\sltp$. \end{proof}

Now we push the representation above down to the boundary.

\begin{lemma}\label{NRuniquelem3.lem}
Suppose that $\L$ satisfies Hypothesis B (see Definition \ref{def-hypothesis-b}), and that u is a good $\mathcal N/ \mathcal R$ solution. Then
\begin{equation}\label{NRuniquelem3.eq}
u =-\mathcal D^{\L,+} f + \mathcal S^{\L}g,
\end{equation}
where $f \in Y^{1,2}(\rn)$ and $g \in L^2(\rn)$ are as in Propositions \ref{weaklimits2.prop} and \ref{conormals2.prop} respectively.
\end{lemma}
\begin{proof}

By Propositions \ref{weaklimits2.prop} and \ref{conormals2.prop}, $u_\tau(\cdot, 0) \to f \in Y^{1,2}(\rn)$ and $\partial_\nu u_\tau \to g \in L^2(\rn)$ weakly in $Y^{1,2}(\rn)$ and $L^2(\rn)$ respectively as $\tau \to 0^+$. Set $u_\tau(\cdot, 0) = f_\tau$ and $\partial_\nu u_\tau = g_\tau$, then rephrasing the above, we have $\vec{h}_\tau = (f_\tau, g_\tau)$ converges to $(f,g)=: \vec{h}$ weakly in $Y^{1,2}(\rn) \times L^2(\rn)$. Let $\tau_k \downarrow 0$ then by Mazur's lemma there exists a sequence $\{\widetilde{h}_l\}_{l = 1}^\infty \subset Y^{1,2}(\rn) \times L^2(\rn)$ such that
$\widetilde{h}_l \to \vec{h}$ {\it strongly} in $Y^{1,2}(\rn) \times L^2(\rn)$ with
$$\widetilde{h}_l = \sum_{k = l}^{N(l)} \lambda_{k,l} \vec{h}_{\tau_k},$$
where $l \le N(l) < \infty$,  $\lambda_{k,l} \in [0,1]$ and  $\sum_{k = l}^{N(l)} \lambda_{k,l} = 1$. 
Set $$\widetilde{u} := \mathcal D^{\L,+} f + \mathcal S^{\L}g.$$ To prove the lemma it is enough to show that for each $t > 0$,
$\widetilde u(\cdot, t) = u(\cdot, t)$ as elements of $Y^{1,2}(\rn)$. 

We have from Lemma \ref{NRuniquelem2.lem} that
$$u_{\tau} = -\mathcal D^{\L,+}(f_\tau) + S^{\L}(g_\tau).$$
Set 
$$u_l := \sum_{k = l}^{N(l)} \lambda_{k,l} u_{\tau_k}.$$
We show $u_l(\cdot, t)$ converges strongly to both $u(\cdot, t)$ and $\widetilde u(\cdot, t)$ in $Y^{1,2}(\rn)$. From the bounded mappings $\mathcal D^{\L,+} : Y^{1,2}(\rn) \to \sltp$ and $\mathcal S^{\L} : L^2(\rn) \to \sltp$
we have 
$$\| \nabla[\widetilde{u}(\cdot, t) - u_l(\cdot, t)] \|_{\ltrn} \le \|\vec{h} - \widetilde{h}_l\|_{Y^{1,2}(\rn) \times L^2(\rn)} \to 0 \text{ as } l \to \infty,$$
where we used the strong convergence of $\widetilde{h}_l$ to $\vec{h}$. To show $u_l(\cdot, t)$ converges strongly to $u(\cdot,t)$ in $Y^{1,2}(\rn)$ we write for $l \ge 0$,
\begin{align*}
\| \nabla u_l(\cdot, t) - \nabla u (\cdot, t)\|_{L^2(\rn)} &= \| \sum_{k = l}^{N(l)} \lambda_{k,l} \nabla[u_{\tau_k} - u](\cdot, t)\|_{L^2(\rn)}
\\ & \le   \sum_{k = l}^{N(l)} \lambda_{k,l} \| \nabla[u_{\tau_k} - u](\cdot, t)\|_{L^2(\rn)}
\\ & \le \sup_{k \ge l} \| \nabla[u(\cdot, t + \tau_k) - u(\cdot, t)]\|_{L^2(\rn)},
\end{align*}
where we used $\sum_{k = l}^{N(l)} \lambda_{k,l} = 1$ and $u(\cdot, \cdot + \tau) = u_{\tau}(\cdot, \cdot) = \mathcal D^{\L,+}(f_\tau) + S^{\L}(g_\tau)$.
We can then use the continuity of $\nabla u(\cdot, t)$ in $L^2(\rn)$ (see \cite[Lemma 2.3]{bhlmp})  
along with $\tau_k \downarrow 0$ to obtain $\| \nabla u_l(\cdot, t) - \nabla u (\cdot, t)\|_{L^2(\rn)} \to 0$ as $l$ tends to infinity. \end{proof}

\begin{theorem}[Uniqueness of the regularity problem among good $\mathcal N / \mathcal R$ solutions]\label{Rgooduniquethrm.thrm}
Suppose $\L$ satisfies Hypothesis B (see Definition \ref{def-hypothesis-b}). Suppose $u$ is a good $\mathcal N / \mathcal R$ solution, with $u(\cdot, 0) = 0$, interpreted in the sense of Proposition \ref{weaklimits2.prop} (i.e. $\lim_{t\to 0} u(t)=0$ weakly in $\yotn$). Then $u\equiv 0$ in $\reu$.
\end{theorem}
\begin{proof}
By Lemma \ref{NRuniquelem3.lem}, we have $u = -\mathcal D^{L,+} f + \mathcal \sl g$, where $f$ and $g$ are as in Lemma \ref{NRuniquelem3.lem}. It follows that $u = \sl g$, since $f  = 0$ (see the proof of Lemma \ref{NRuniquelem3.lem}). Moreover, by taking traces (in the sense of Proposition \ref{weaklimits2.prop}) in $Y^{1,2}(\rn)$ we obtain $0 = \sl_0 g$, for $g \in L^2(\rn)$. It follows from the invertibility of $\sl_0: L^2(\rn) \to Y^{1,2}(\rn)$ that $g = 0$. This gives $u \equiv 0$. \end{proof}

\begin{theorem}[Uniqueness of the Neumann problem among good $\mathcal N / \mathcal R$ solutions]\label{Ngooduniquethrm.thrm}
Suppose that $\L$ satisfies Hypothesis B (see Definition \ref{def-hypothesis-b}), and that $u$ is a good $\mathcal N / \mathcal R$ solution, with $\partial_\nu u = 0$,   in the sense of Proposition \ref{conormals2.prop}. Then $u\equiv 0$ in $\reu$.
\end{theorem}
\begin{proof}
By Lemma \ref{NRuniquelem3.lem}, we have $u = -\mathcal D^{\L,+} f + \mathcal S^{\L} g$, where $f$ and $g$ are as in Lemma \ref{NRuniquelem3.lem}. It follows that $u = -\mathcal D^{\L,+} f$, since $g = 0$ (see the proof of Lemma \ref{NRuniquelem3.lem}), where $f \in Y^{1,2}(\rn)$. From \eqref{mapprop2.eq}, we have after taking conormal derivatives, in the sense of
Proposition \ref{conormals2.prop}\footnote{We note that, having obtained the mapping property $\mathcal D \to \sltp$, the equality of \eqref{mapprop2.eq} holds on every $t$-slice in the space $Y^{1,2}(\rn)$ therefore the weak $L^2(\rn)$ limits, in $t$, of the co-normal derivatives $\partial_{\nu_t}$ are the same.},
and using the jump relations for the conormal of the single layer potential 
$$0 = \partial_{\nu^{\cL,+}} u = -\partial_{\nu^{\cL,+}} \mathcal D^{\L,+} f = -(-\tfrac{1}{2}I + \widetilde{K})(\tfrac{1}{2}I + \widetilde K)S_0^{-1}f \quad \text{ in } L^2(\rn).$$
The invertibility of $\pm \tfrac{1}{2} I + \widetilde K : L^2(\rn) \to L^2(\rn)$ and $S_0^{-1}: Y^{1,2}(\rn) \to L^2(\rn)$ yields that $f = 0$ and hence $u \equiv 0$. \end{proof}

\subsection{Uniqueness for the  Dirichlet problem}

The first lemma here simply states that the conormal derivatives are uniformly bounded in the transversal variable, for good $\mathcal D$ solutions.

\begin{lemma}\label{Duniquelem2.lem}
	Suppose $\L$ satisfies Hypothesis B (see Definition \ref{def-hypothesis-b}). Assume $u$ is a good $\mathcal D$ solution. Then for every $\tau > 0$, $\partial_{\nu_\tau} u \in [Y^{1,2}(\rn)]^*$, with the bound
	$$\sup_{\tau > 0}\| \partial_{\nu_\tau}u \|_ {[Y^{1,2}(\rn)]^*} \le C \sup_{t > 0}\|u \|_{L^2(\rn)}.$$
\end{lemma}

\begin{proof}
	By symmetry of hypotheses and Lemma \ref{Duniquelem1.lem} and Theorem \ref{invertibility.thm} the operator 
	$S^{\cL^*}_0 : L^2(\rn) \to Y^{1,2}(\rn)$ is bounded and invertible. Then the collection of functions $\F:= \{\varphi \in Y^{1,2}(\rn): \varphi = S^{\cL^*}_0 f, f \in C_c^\infty\}$ is dense in $Y^{1,2}(\rn)$. Notice that $v_\varphi = \mathcal S^{\cL^*} ([S^{\cL^*}_0]^{-1} \varphi) = S^{\cL^*}f  \in Y^{1,2}(\ree)$ since $f \in C_c^\infty(\rn) \subset H^{-1/2}(\rn)$. Also, $\text{Tr}_0 u_\tau \in H^{1/2}_0(\rn)$ since $u_\tau \in Y^{1,2}(\ree_+)$, where, as above $u_\tau(\cdot, \cdot) := u(\cdot, \cdot + \tau)$. Then by definition of $\partial_{\nu^*} v_\varphi \in H^{-1/2}(\rn)$ (see  \cite[Definition 4.9]{bhlmp}) with 
	\begin{align*}
		(\text{Tr}_0u_\tau, \partial_{\nu^*}v_\varphi) &= \overline{(\partial_{\nu^*}v_\varphi, \text{Tr}_0u_\tau)}
		\\ & = \overline{B_{\cL^*}[v_\varphi, u_\tau]} = B_{\cL}[u_\tau, v_\varphi].
	\end{align*}
	Now $B_{\cL}[u_\tau, v_\varphi] = (\partial_{\nu_\tau}u , \varphi)$, since $v_\varphi$ solves the regularity problem with data $\varphi$ by Theorem \ref{existence.thm}. In particular $\varphi$ is the weak limit of $v_\varphi(\cdot, t)$ in $Y^{1,2}(\rn)$ as $t\to 0$.

	Having established 
	$$(\partial_{\nu_\tau}u, \varphi) = (\text{Tr}_0u_\tau, \partial_{\nu^*}v_\varphi)$$
	for $\varphi$ in $\F$ we see that $\partial_{\nu_\tau}u \in [Y^{1,2}(\rn)]^*$ by the fact that the map 
	$$F_\varphi := \partial_{\nu^*}v_\varphi = \partial_{\nu^*} \mathcal S^{\cL^*} ([S^{\cL^*}_0]^{-1} \varphi)$$
	maps $Y^{1,2}(\rn) \to L^2(\rn)$, by Proposition \ref{conormals2.prop}, the mapping property mentioned at the start of the proof for $\mathcal{S}^{\cL^*}_0$, and the density of $\F$ in $Y^{1,2}(\rn)$. \end{proof}

Next, we prove a Green's formula for good $\mathcal D$ solutions.

\begin{lemma}\label{Duniquelem3.lem}
Suppose  that $\L$ satisfies Hypothesis B (see Definition \ref{def-hypothesis-b}) and let $u$ be a good $\mathcal D$ solution. For $\tau > 0$, set $(f_\tau, g_\tau): = (\text{Tr}_0 u_\tau, \partial_\nu u_\tau) = (\text{Tr}_0 u_\tau, \partial_{\nu_\tau} u) \in L^2(\rn) \times [Y^{1,2}(\rn)]^*$, where we use Lemma \ref{Duniquelem2.lem} to identify $\partial_{\nu_\tau} u$ as an element of $[Y^{1,2}(\rn)]^*$. Then
$$u = -\mathcal D^{\L,+}f + \mathcal S^{\L}g,$$
where the pair $(f,g) \in L^2(\rn) \times [Y^{1,2}(\rn)]^*$ is {\bf any} convergent weak limit of
$(f_{\tau_k}, g_{\tau_k})$, $\tau_k \downarrow 0$ in the space $L^2(\rn) \times [Y^{1,2}(\rn)]^*$.
\end{lemma}

\begin{remark}
We note that the existence of at least one such limiting pair $(f,g)$ is guaranteed by the fact that $f_\tau$, $g_\tau$ are uniformly bounded in $\ltrn$ and $\dyotn$ respectively (the first by the hypothesis $u\in \dltp$ and the second by Lemma \ref{Duniquelem2.lem}) together with the fact that both of these spaces are reflexive. 

Unlike the case of good $\mathcal{N}/\mathcal{R}$ solutions, here we make no assertion about the uniqueness of such a limiting pair.
\end{remark}

\begin{proof}

The proof is quite similar to Lemma \ref{NRuniquelem3.lem}, but we provide the details here. We have that $u_{\tau} \in Y^{1,2}(\ree_+)$ with $\cL u_\tau = 0$ we have
$$u_\tau = -\mathcal D^{\L,+}(\text{Tr}_0 u_\tau) + S^{\L}(\partial_\nu u_\tau)$$
for all $\tau > 0$. Let $\vec{h}_{\tau_k}:= (f_{\tau_k}, g_{\tau_k}) \rightharpoonup (f,g) =: \vec{h} \in L^2(\rn) \times [Y^{1,2}(\rn)]^*$ be as in the statement of the lemma. Using Mazur's lemma there exists a sequence $\{\widetilde{h}_l\}_{l = 1}^\infty \subset L^2(\rn) \times [Y^{1,2}(\rn)]^*$ such that
$\widetilde{h}_l \to \vec{h}$ {\it strongly} in $L^2(\rn) \times [Y^{1,2}(\rn)]^*$ with
$$\widetilde{h}_l = \sum_{k = l}^{N(l)} \lambda_{k,l} \vec{h}_{\tau_k},$$
where $l \le N(l) < \infty$,  $\lambda_{k,l} \in [0,1]$ and  $\sum_{k = l}^{N(l)} \lambda_{k,l} = 1$. 
Set $$\widetilde{u} := \mathcal D^{\L,+} f + \mathcal S^{\L}g$$
and
$$u_l := \sum_{k = l}^{N(l)} \lambda_{k,l} u_{\tau_k}.$$
We show $u_l(\cdot, t)$ converges strongly to both $u(\cdot, t)$ and $\widetilde u(\cdot, t)$ in $L^2(\rn)$. From the bounded mappings $\mathcal D^{\L,+} : L^2(\rn) \to \dltp$ and $\mathcal S^{\L} : [Y^{1,2}(\rn)]^* \to \dltp$
we have 
$$\| \widetilde{u}(\cdot, t) - u_l(\cdot, t) \|_{L^2(\rn)} \le \|\vec{h} - \widetilde{h}_l\|_{L^2(\rn) \times [Y^{1,2}(\rn)]^*} \to 0 \text{ as } l \to \infty.$$

To show $u_l(\cdot, t)$ converges strongly to $u(\cdot,t)$ in $L^2(\rn)$ we write for $l \ge 0$,
\begin{align*}
\| u_l(\cdot, t) -  u (\cdot, t)\|_{L^2(\rn)} &= \| \sum_{k = l}^{N(l)} \lambda_{k,l} [u_{\tau_k} - u](\cdot, t)\|_{L^2(\rn)}
\\ & \le   \sum_{k = l}^{N(l)} \lambda_{k,l} \| [u_{\tau_k} - u](\cdot, t)\|_{L^2(\rn)}
\\ & \le \sup_{k \ge l} \| u(\cdot, t + \tau_k) - u(\cdot, t)\|_{L^2(\rn)},
\end{align*}
where we used $\sum_{k = l}^{N(l)} \lambda_{k,l} = 1$ and $u(\cdot, \cdot + \tau) = u_{\tau}(\cdot, \cdot) = \mathcal D^{\L,+}(f_\tau) + \mathcal{S}^{\L}(g_\tau)$.
We can then use the continuity of $u(\cdot, t)$ in $L^2(\rn)$  (see \cite[Lemma 2.3]{bhlmp})\footnote{We may modify this Lemma, using now the function space $W^{1,2}(\Sigma_a^b)$ instead of $Y^{1,2}(\Sigma_a^b)$ to obtain the desired continuity in $L^2(\rn)$ instead of $L^{2^*}(\rn)$.}  along with $\tau_k \downarrow 0$ to obtain $\|  u_n(\cdot, t) -  u (\cdot, t)\|_{L^2(\rn)} \to 0$ as $n$ tends to infinity. Therefore $u = \widetilde{u}$ in  $\dltp$ and the lemma is shown.\end{proof}

\begin{theorem}[Uniqueness of the Dirichlet problem among good $\mathcal D$ solutions]\label{Dgooduniquethrm.thrm}
Suppose that $\L$ satisfies Hypothesis B (see Definition \ref{def-hypothesis-b}), and that $u$ is a good $\mathcal D$ solution, with $u(\cdot, t) \to 0$ weakly in $L^2(\rn)$. Then $u \equiv 0$.  
\end{theorem}

\begin{proof}
By Lemma \ref{Duniquelem3.lem}, we have that $u = \mathcal S^{\L}g$ for some $g \in [Y^{1,2}(\rn)]^*$, where $g \in [Y^{1,2}(\rn)]^*$ is any weak limit of $g_{\tau_k} = \partial_{\nu^{\cL,+}} u_{\tau_k}$, $\tau_k \downarrow 0$ as in Lemma \ref{Duniquelem3.lem}. We also have (see \eqref{mapprops1.eq})
$$\mathcal S_t^{\L} g = \mathcal D_t^{\L,+}(\mathcal{S}_0^{\L}[-\tfrac{1}{2}I + \widetilde{K}]^{-1}g),$$
where we used $[-\tfrac{1}{2}I + \widetilde{K}]^{-1}: [Y^{1,2}(\rn)]^* \to [Y^{1,2}(\rn)]^*$ and $\mathcal{S}^{\L}_0:  [Y^{1,2}(\rn)]^* \to L^2(\rn)$. Taking weak limits in $L^2(\rn)$ we obtain 
$$0 = [-\tfrac{1}{2}I + K]\mathcal{S}^{\L}_0[-\tfrac{1}{2}I + \widetilde{K}]^{-1}g$$
The invertibility of the mappings $-\tfrac{1}{2}I + K : L^2(\rn) \to L^2(\rn)$, $\mathcal{S}^{\L}_0:  [Y^{1,2}(\rn)]^* \to L^2(\rn)$ and  $[-\tfrac{1}{2}I + \widetilde{K}]^{-1}: [Y^{1,2}(\rn)]^* \to [Y^{1,2}(\rn)]^*$ give that $g = 0$ in $[Y^{1,2}(\rn)]^*$ and hence $u \equiv 0$. \end{proof}

\section{$L^p$ solvability for $p$ in a window around $2$}\label{sec.lp}

In this final section, we  extend the $L^2$ existence and uniqueness results of the last two sections to an $L^p$ solvability result for the boundary value problems considered, provided that $p$ is close to $2$, hence finishing the proof of Theorem \ref{L2Solv2.thrm}.

The square function and non-tangential maximal function $L^p$ estimates for the case $p\in(2-\ep_0,2+\ep_0)$ for $\ep_0$ small enough are already contained in Theorem \ref{thm-ntmax-estimates-final-version}; and thus, if one assumes that $f\in C_c^{\infty}(\bb R^n)$, then the boundary value problems considered admit a solution, represented via layer potentials, with the appropriate $L^p$ estimates. If $f\in L^p(\bb R^n)$ is not smooth and $p\neq2$, then an approximation argument via smooth $f_k\in C_c^{\infty}(\bb R^n)$ will work once we have established the mapping properties of the layer potentials and the jump relations that work with boundary data in the $L^p$ space. As the methods here are very similar to those of Section 4 of \cite{bhlmp} and Sections \ref{existence.sect} and \ref{sec.uniqueness} of the present document, we will omit many details for the sake of brevity.

 \noindent\emph{Proof of Theorem \ref{L2Solv2.thrm}, case $p\neq2$.} Recall that the space $Y^{1,p}(\bb R^{n+1})$ has been defined in (\ref{eq.y1p}), and let $q$ be the H\"older conjugate of $p$, so that $\frac1p+\frac1q=1$. When $\rho_0$ is small enough, the operator $\L$ associated to the sesquilinear form $B_\L$ defined in Section \ref{sec.prop} maps  $Y^{1,p}(\bb R^{n+1})\ra(Y^{1,q}(\bb R^{n+1})^*$ and is a bounded, linear invertible operator by the Lax-Milgram theorem.  The horizontal traces of $Y^{1,p}(\bb R^{n+1})$ continuously embed into the Besov space $B_0^{1-\frac1p,p}(\bb R^n)$ consisting of functions vanishing at infinity and with finite $B^{1-\frac1p,p}(\bb R^n)$ seminorm\footnote{For a definition of the Besov space, see Definition 14.1 of \cite{L}} (see \cite[Theorem 15.20]{L}). Moreover, given a function $f\in B_0^{1-\frac1p,p}(\bb R^n)$, there exists an extension $F\in Y^{1,p}(\bb R^{n+1}_+)$ with $\Tr\, F=f$ on $\partial\bb R^{n+1}_+$ (\cite[Theorem 15.21]{L}\footnote{Technically, Leoni considers the non-homogeneous case; however his proof easily gives the result stated here}). We also have the Sobolev embedding
  \[
  B_0^{1-\frac1p,p}(\bb R^n)\hookrightarrow L^{\frac{np}{n+1-p}}(\bb R^n),
  \]
  whenever $p<n+1$ (\cite[Theorem 14.29]{L}), and $C_c^{\infty}(\bb R^n)$ is dense in $B_0^{1-\frac1p,p}(\bb R^n)$ (see \cite[Theorem 3.1]{bgcv21}).
  
   The   properties mentioned above ensure that the single layer potential
  \[
  \sl:(B_0^{1-\frac1q,q}(\bb R^n))^*\ra Y^{1,p}(\bb R^n),
  \]
  defined via the formula (\ref{eq.slp}),  is still a well-defined bounded linear operator. Similarly, one may check that the double layer potential 
  \[
  \mathcal{D}^{\L,+}:B_0^{1-\frac1p,p}(\bb R^n)\ra Y^{1,p}(\bb R^{n+1}_+),
  \]
  is well-defined via the formula (\ref{eq.dlp}), with the appropriate modifications. Then, via a density argument and using the slice estimates Theorems \ref{thm-sup-on-slices} and \ref{thm-slices-double-layer}, we may show that the layer potentials extend uniquely as
  \[
  \sl:L^p(\bb R^n)\ra S^p_+,\qquad\mathcal{D}^{\L,+}:L^p(\bb R^n)\ra D^p_+,
  \]
  where $S^p_+$ and $D^p_+$ are the slice spaces of Definition \ref{def.slicespace}. Furthermore, Theorem \ref{thm-ntmax-estimates-final-version} ensures that we have the required square function bounds.
  
  We now briefly sketch the existence argument. We are able to carry out the arguments from Section \ref{existence.sect}   and obtain the boundary operators $\sl_0:L^p(\bb R^n)\ra Y^{1,p}(\bb R^n)$, $\mathcal{D}^{\L,+}_0: L^p(\bb R^n)\ra L^p(\bb R^n)$, which are given by the formulas of Corollary \ref{boundaryvalues.cor}. On the other hand, the conormal derivative $\partial_\nu^{\L,\pm}:Y^{1,p}(\bb R^{n+1}_\pm)\ra(B_0^{1-\frac1q,q}(\bb R^n))^*$ is also well-defined (as in \cite[Definition 4.9]{bhlmp}, with appropriate modifications), and one may thus show that the several variations of Green's formula from \cite[Theorem 4.16]{bhlmp} and the jump relations \cite[Theorem 4.22]{bhlmp} hold in this setting, with essentially the same proofs. These facts allow us to construct, as in Lemma \ref{jumprelation.lem}, the bounded linear operators $K, \widetilde K: L^p(\bb R^n)\ra L^p(\bb R^n)$ which satisfy the identities
  \[
  (\pm\tfrac12I+\widetilde K)g=\partial_\nu^{\L,\pm}(\sl g),\qquad (\mp\tfrac12I+K)f=\mathcal{D}_0^{\L,\pm}f
  \]
  for any $f,g\in L^p(\bb R^n)$. Then  one may prove proper analogues  of Corollary \ref{NRcorollary.cor}, Proposition \ref{jumprelationsy12.prop}, and Lemma \ref{Duniquelem1.lem} under the assumption of bounded invertibility of the boundary operators $\sl_0:L^p(\bb R^n)\ra Y^{1,p}(\bb R^n)$, $\mathcal S^{\L^*}_0:L^q(\bb R^n)\ra Y^{1,q}(\bb R^n)$, and $(-\frac12I+K):L^p(\bb R^n)\ra L^p(\bb R^n)$. This assumption (which is an analogue of hypothesis B from Definition \ref{def-hypothesis-b}) is satisfied for the operator $\L_0=-\div A_0\nabla$  when $A_0$ is either Hermitian, block form or constant\footnote{An application of Sneiberg's Lemma, and the known $L^2$ results (see the introduction), reduces the invertibility of the boundary operators in Hypothesis B to the uniform boundedness of said operators in $L^p$ for $p$ in a neighborhood of $2$. In turn this last is achieved by the methods of this paper.}, and by the method of analytic perturbations in Theorem \ref{invertibility.thm}, the operator $\L$ also satisfies this assumption, showing the invertibility of layer potentials.
  
  With the invertibility of the layer potentials and the appropriate analogues of the mapping properties at hand, we may finally obtain the existence of solutions for the problems $\Di_p$, $\Ne_p$, and $\Reg_p$ in the same way as in Theorem \ref{existence.thm}. 
  
  We turn to the uniqueness of the solutions to the boundary value problems here considered. As in the case of $p=2$, we are able to consider uniqueness in the wider class of good $\mathcal D$ solutions (Definition \ref{def-good-d-solutions}) for the Dirichlet problem with exponent $p$, and good $\mathcal N$/$\mathcal R$ solutions (Definition \ref{def-good-nr-solutions}) for the Neumann and regularity problems with exponent $p$\footnote{Of course, the definitions of good $\mathcal D$/$\mathcal N$/$\mathcal R$ solutions have to be appropriately modified to work with the exponent $p$ and the slice spaces $D^p_+$ and $S^p_+$.}. Once again, the methods of Section \ref{sec.uniqueness} work in this setting, with very little change , since we have the appropriate analogues of the Green's formulas from \cite[Theorem 4.16]{bhlmp} and the various analogues of the mapping properties and jump relations from Section \ref{existence.sect}. We omit further details.\hfill{$\square$}

\hypersetup{linkcolor=toc}

\bibliography{bhlmprefs} 
\bibliographystyle{alpha-sort-max} 

\end{document}